\documentclass[oneside]{amsart}
\usepackage{amssymb}
\usepackage{amsmath}
\usepackage{amsfonts}
\usepackage{geometry}
\usepackage{soul}
\usepackage{bbm}
\usepackage[colorlinks=true]{hyperref}
\usepackage{tikz}
\usepackage{mathrsfs}
\usepackage{multirow}
\usetikzlibrary{matrix,arrows,decorations.pathmorphing}
\usepackage{verbatim }
\usepackage{mathtools}
\usepackage{enumitem}
\usepackage[author={}]{pdfcomment}
\usepackage{todonotes}
\usepackage[normalem]{ulem}

\newcommand{\change}[1]{{\color{black} #1}}

\newcounter{cprop}[section]

\newtheorem{theorem}[cprop]{Theorem}

\theoremstyle{plain}

\newtheorem{corollary}[cprop]{Corollary}

\newtheorem{lemma}[cprop]{Lemma}
\newtheorem{proposition}[cprop]{Proposition}

\newtheorem{assumption}[cprop]{Assumption}

\numberwithin{equation}{section}
\usepackage[foot]{amsaddr}
\theoremstyle{definition}
\newtheorem{definition}[cprop]{Definition}
\newtheorem{example}[cprop]{Example}

\theoremstyle{remark}
\newtheorem{remark}[cprop]{Remark}

\makeatletter
\def\namedlabel#1#2{\begingroup
	\def\@currentlabel{#2}%
	\label{#1}\endgroup
}
\makeatother

\makeatletter
\def\@tocline#1#2#3#4#5#6#7{\relax
	\ifnum #1>\c@tocdepth 
	\else
	\par \addpenalty\@secpenalty\addvspace{#2}%
	\begingroup \hyphenpenalty\@M
	\@ifempty{#4}{%
		\@tempdima\csname r@tocindent\number#1\endcsname\relax
	}{%
		\@tempdima#4\relax
	}%
	\parindent\z@ \leftskip#3\relax \advance\leftskip\@tempdima\relax
	\rightskip\@pnumwidth plus4em \parfillskip-\@pnumwidth
	#5\leavevmode\hskip-\@tempdima
	\ifcase #1
	\or\or \hskip 1em \or \hskip 2em \else \hskip 3em \fi%
	#6\nobreak\relax
	\hfill\hbox to\@pnumwidth{\@tocpagenum{#7}}\par
	\nobreak
	\endgroup
	\fi}
\makeatother

\def\R{\mathbb{R}}
\def\C{\mathbb{C}}
\def\N{\mathbb{N}}

\def\P{\mathbb{P}}
\def\E{\mathbb{E}}

\def\P{\mathbb{P}}
\def\cA{\mathcal{A}}

\def\cC{\mathcal{C}}
\def\cD{\mathcal{D}}

\def\cF{\mathcal{F}}

\def\cH{\mathcal{H}}

\def\cL{\mathcal{L}}

\def\cO{\mathcal{O}}

\def\txtd{{\textnormal{d}}}

\def\txtD{{\textnormal{D}}}

\def\Id{{\textnormal{Id}}}

\newcommand{\vertiii}[1]{{\left\vert\kern-0.25ex\left\vert\kern-0.25ex\left\vert #1 
		\right\vert\kern-0.25ex\right\vert\kern-0.25ex\right\vert}}

\begin{document}
	\title[A mild rough Gronwall Lemma]{A mild rough Gronwall Lemma with applications to non-autonomous evolution equations}
	
	\author{Alexandra Blessing Neam\c tu, Mazyar Ghani Varzaneh~~and~~Tim Seitz}
	
	\address{University of Konstanz, Department of Mathematics and Statistics,  Universit\"atsstra\ss{}e~10 78464 Konstanz, Germany.}
	\email{alexandra.blessing@uni-konstanz.de, mazyar.ghani-varzaneh@uni-konstanz.de}
	\email{tim.seitz@uni-konstanz.de} 
	
	\begin{abstract}
		We derive a Gronwall type inequality for mild solutions of non-autonomous parabolic rough partial differential equations (RPDEs).~This inequality together with an analysis of the Cameron-Martin space associated to the noise, allows us to obtain the existence of moments of all order for the solution of the corresponding RPDE and its Jacobian when the random input is given by a Gaussian Volterra process.~Applying further the multiplicative ergodic theorem, these integrable bounds entail the existence of Lyapunov exponents for RPDEs.~We illustrate these results for stochastic partial differential equations with multiplicative boundary noise.
	\end{abstract}    
	\maketitle
	{\bf Keywords}: rough partial differential equations, mild Gronwall lemma, Lyapunov exponents, rough boundary noise. \\
	{\bf Mathematics Subject Classification (2020)}: 60G22, 60L20, 
	60L50, 
	37H10, 37L55. 

	{
		\hypersetup{hidelinks}
		\tableofcontents
	}
	\section{Introduction}
	
	The main goal of this work is to derive a Gronwall inequality for mild solutions of parabolic rough partial differential equations {of the form}
	\begin{align} \label{Main_Equation}
		\begin{cases}\txtd u_t = [A(t)u_t +F(t,u_t) ]~\txtd t + G(t,u_t)~\txtd \mathbf{X}_t \\
			u_0\in E_\alpha,
		\end{cases}
	\end{align}
	on a family of Banach spaces $(E_\alpha)_{\alpha\in\R}$. Here $\mathbf{X}$ is the rough path lift of a Gaussian Volterra process and the coefficients $A, F$ and $G$ satisfy suitable assumptions specified in Section~\ref{integrable}. Our approach complements the results in~\cite{DGHT19,Hof18} that establish a Gronwall inequality for rough PDEs with transport-type noise using energy estimates in the framework of unbounded rough drivers. \\
	
	Furthermore, the mild Gronwall inequality stated in Lemma~\ref{lem:RoughGronwall} allows us to obtain a-priori bounds for the global solution of~\eqref{Main_Equation} together with its linearization around an arbitrary trajectory, which turn out to be crucial in establishing the existence of Lyapunov exponents for rough PDEs. Motivated by applications in fluid dynamics~\cite{BeBlPS:22,BeBlPs2:22}  and bifurcations in infinite-dimensional stochastic systems~\cite{BlEnNe:23,BN:23}, Lyapunov exponents recently captured lots of attention.~However, to our best knowledge, there are no works that systematically analyze Lyapunov exponents in the context of RPDEs. Here we contribute to this aspect and first provide, based on Gronwall's Lemma, a-priori integrable bounds for the solution of~\eqref{Main_Equation} and its Jacobian, which entail the existence of Lyapunov exponents for a fixed initial data based on the multiplicative ergodic theorem. \\
	
	Since we are considering parabolic RPDEs on a scale of Banach spaces, a natural question is whether the Lyapunov exponents depend on the underlying norm. This turns out not to be the case, as shown in \cite{BLPS} and applied to models arising from fluid dynamics perturbed by noise which is white in time.~This is natural, since Lyapunov exponents reflect intrinsic dynamical properties of the system and should therefore be independent of the chosen norm.~We provide a proof of this statement in the context of rough PDEs in Section~\ref{indep:LY} using a version of the multiplicative ergodic theorem stated in Theorem~\ref{METT} together with a duality argument inspired by~\cite{GTQ15} and~\cite{GVR23A}.\\
	
	We emphasize that the existence of Lyapunov exponents for rough PDEs based on the multiplicative ergodic theorem is strongly related to the existence of moments of all orders for the solution of equation~\eqref{Main_Equation} and its Jacobian, which is known to be a challenging task. In the finite-dimensional case, such integrable bounds are also essential for the existence of densities of rough differential equations under H\"ormander's condition.~The existence of moments of all order for the Jacobian of the solution flow of  differential equations driven by Gaussian rough paths, have been obtained in the seminal work~\cite{CLL13}.~Later~\cite{GH19} proved that the finite-dimensional projections of solutions of rough PDEs admit densities with respect to the Lebesgue measure, circumventing the integrability issue.~However, for our aims in Section~\ref{sec:InvSets} which follow a random dynamical systems based approach, integrable bounds of the solution of~\eqref{Main_Equation} and its Jacobian are crucial.~Generalizing the finite-dimensional results in~\cite{CLL13},~\cite{GVR25} obtained such bounds under additional assumptions on the Cameron-Martin space associated to the noise.~This assumption can be checked for fractional Brownian motion but is challenging to verify for other Gaussian processes. Here, we analyze in Subsection~\ref{cm} the Cameron-Martin space associated to Volterra processes, which can be represented as an integral of a kernel with respect to the Brownian motion. We provide conditions, which can easily be verified under natural assumptions on the kernel, in order to guarantee integrable bounds for~\eqref{Main_Equation} driven by the rough path lift of such processes.  \\

\\

This manuscript is structured as follows.~In Section~\ref{prelim}, we state basic concepts from rough path theory and parabolic evolution families. Section \ref{integrable} is devoted to the local and global well-posedness of~\eqref{Main_Equation} using a controlled rough path approach.~The local and global well-posedness of rough PDEs has recently received lots of attention due~\cite{GH19,GHT21,HN22} and~\cite{Tappe}.~The works~\cite{GH19,GHT21} 
consider parabolic (non-autonomous) rough PDEs, where the differential operator $A$ in~\eqref{Main_Equation} generates an analytic semigroup, respectively a parabolic evolution family in the non-autonomous case, and the noise is a finite-dimensional rough path.~The work of~\cite{Tappe} deals with differential operators $A$ which generate arbitrary $C_0$-semigroups and consider infinite-dimensional noise.~As already mentioned, here we go a step further and obtain the existence of moments of all order for the controlled rough path norm of the solution and its Jacobian. Therefore, we first replace the H\"older norms of the random input by suitable control functions~\cite{CLL13,GVR25} which enjoy better integrability properties.~These are incorporated in the sewing Lemma~\ref{lem:IneqRV}, which allows us to define the rough integral. We point out that these techniques heavily rely on the assumption that the diffusion coefficient $G$ of~\eqref{Main_Equation} is bounded.~\change{This restriction was recently removed in \cite{BGV25} by a different approach, which uses another concept of controlled rough paths and control functions.}~In Subsection~\ref{cm}, we analyze the Cameron-Martin space associated to the noise, providing a criterion for integrable bounds for~\eqref{Main_Equation} driven by Gaussian Volterra processes.\\

In Section~\ref{sec:gronwall}, we derive the Gronwall inequality in Lemma~\ref{lem:RoughGronwall} using the mild formulation of~\eqref{Main_Equation}, regularizing properties of parabolic evolution families, and a suitable discretization argument.~We present an application of this result in Subsection~\ref{sec:LinDynamics}, where we linearize~\eqref{Main_Equation} along an arbitrary trajectory. The bound entailed by the mild rough Gronwall inequality is crucial for our analysis of Lyapunov exponents in Section~\ref{sec:InvSets}.~This section contains further the application of the results in Section~\ref{sec:gronwall} to random dynamical systems.~Since the coefficients of~\eqref{Main_Equation} are time-dependent, we first enlarge the probability space in order to incorporate this dependency to use the framework of random dynamical systems. One could also work with non-autonomous dynamical systems, as for e.g.~\cite{CL17}.~However, our approach makes the application of the multiplicative ergodic theorem more convenient.~This is a main goal of our work, since we address the existence of Lyapunov exponents for~\eqref{Main_Equation}.~To this aim, we  obtain integrable bounds for the solution of the linearization of~\eqref{Main_Equation} along a stationary solution using the mild rough Gronwall lemma.~Furthermore, in Subsection~\ref{indep:LY}, in order to show the independence of the Lyapunov exponents on the underlying norm in Theorem~\ref{METT2} and Theorem~\ref{unnns}, we first associate to each finite Lyapunov exponent, a unique finite-dimensional space called fast-growing space.~We prove that these spaces do not depend on the underlying norm, which is, to the best of our knowledge, the first result in this direction.~As a consequence of the multiplicative ergodic theorem, under further sign information on the Lyapunov exponents, one can derive the existence of invariant sets for the corresponding random dynamical system. We illustrate this for stable manifolds in Subsection~\ref{inv:m}.~These are infinite-dimensional invariant sets of the phase space which contain solutions
starting from initial data that asymptotically exhibit an exponential decay.~Their existence for stochastic partial differential equations in the Young regime was stated as a conjecture in~\cite{LS11} and was later obtained in~\cite{LNZ23} for a trace-class fractional Brownian using tools from fractional calculus and~\cite{GVR24} using rough path theory.~To analyze the existence of stable manifolds, we additionally derive a stability statement for the difference of two solutions of the linearizations of~\eqref{Main_Equation} along a suitable trajectory in Subsection~\ref{sec:LinDynamics}, which are again based on Gronwall's inequality.~By analogue arguments, one can derive the existence of random unstable and center manifolds, significantly extending the results obtained in~\cite{GVR25,GVR23C, KN23,LNZ23} by different techniques.\\

We conclude with two applications in Section~\ref{sec:app}.~These are given by parabolic RPDEs with time-dependent coefficients and SPDEs with rough boundary noise.~In the case of white noise, non-autonomous SPDEs were considered in~{\cite{Ver10}}, where the generators are additionally allowed to be time-dependent. Here, we further assume that the generators have bounded imaginary powers, which implies that the interpolation spaces are time-independent.~Otherwise, one would need another concept of controlled rough paths according to a monotone time-dependent scale of interpolation spaces reflecting an interplay between the regularity of the noise, the spatial regularity and the time-dependency.~This aspect will be investigated in a future work.~Moreover, it would also be desirable to combine the rough path approach presented here with the theory of maximal regularity for SPDEs, see~\cite{AV25} for a recent survey on this topic.  \\

Furthermore, it is well-known that stochastic partial differential equations (SPDEs) with boundary noise are challenging to treat. For instance~\cite{DPZ93}, the well-posedness of SPDEs with Dirichlet boundary conditions fails for the Brownian motion, see for e.g. \cite{AB02,DPZ93,GP23} for more details and alternative approaches. 
However, for a fractional Brownian motion with Hurst parameter $H>3/4$, also Dirichlet boundary conditions can be incorporated.~This aspect was investigated for the heat equation in~\cite{DPDM02} and the 2D-Navier Stokes equation in~\cite{ABL24} perturbed by an additive fractional boundary noise. 
On the other hand, the well-posedness theory in the case of Neumann boundary noise is more feasible and well-established~\cite{DFT07,Mun17,SV11,AL24}. To the best of our knowledge, all references specified above deal with additive noise, while nonlinear multiplicative noise was considered in \cite{NS23}, using rough path theory.~This turned out to be very useful for the analysis of the long-time behavior of such systems. Due to the noise acting on the boundary, one cannot perform flow-type transformations in order to reduce such equations into PDEs with random non-autonomous coefficients and obtain the existence of a random dynamical system. This issue does not occur in a pathwise approach, which was exploited in~\cite{NS23,BS24} to establish the well-posedness of PDEs with nonlinear multiplicative boundary noise and study their long-time behavior by means of random attractors. 
However, the influence of boundary noise on the long-term behavior of such systems has not been fully analyzed. For example, stability criteria were investigated in \cite{AB02}, a stabilization effect by boundary noise was shown for the Chaffee-Infante equation in \cite{FSTT19}, and the existence of attractors was investigated in \cite{BS24}. We further refer to~\cite{BDK24} for the analysis of warning signs for a Boussinesq model with boundary noise. Here we establish the existence of Lyapunov exponents based on the techniques developed in Sections~\ref{sec:gronwall} and~\ref{sec:InvSets}, which is, to our best knowledge, the first result in this direction.~We further mention that, in applications to fluid dynamics, for e.g.~in the context of a simplified version of the 3D-Navier Stokes system called the primitive equation~\cite{BHHS24}, the boundary noise models random wind-driven boundary effects.  \\

Finally, we provide two appendices on stationary solutions for SPDEs with boundary noise and translation compact functions.~Their properties are used in Section~\ref{sec:InvSets} in order to obtain an autonomous random dynamical system, enlarging the probability space by incorporating the non-autonomous dependence of~\eqref{Main_Equation}.

\subsubsection*{Acknowledgements} A. Blessing and M. Ghani Varzaneh acknowledge support from DFG CRC/TRR 388 {\em Rough Analysis, Stochastic Dynamics and Related Fields}, Project A06. The authors thank the referee for the numerous valuable comments and suggestions.

\section{Preliminaries.~Rough path theory and parabolic evolution families}\label{prelim}
We first provide some fundamental concepts from rough path theory and parabolic evolution families.

For $d\geq 1$ we consider a $d$-dimensional $\gamma$-H\"older rough path $\textbf{X}:=(X,\mathbb{X})$, for $\gamma\in(\frac{1}{3},\frac{1}{2}]$ with $X_0=0$. 
More precisely, we have for $T>0$ that
\begin{align*}
	X\in C^{\gamma}([0,T];\mathbb{R}^d) ~~\mbox{ and } ~~ \mathbb{X}\in C_2^{2\gamma}(\Delta_{[0,T]};\mathbb{R}^d\otimes\mathbb{R}^d)
\end{align*}
where $\Delta_{J}\coloneqq\{(s,t)\in J\times J~:~s\leq t\}$ for $J\subset \R$ and the connection between $X$ and $\mathbb{X}$ is given by Chen's relation
\begin{align*}
	\mathbb{X}_{s,t}- \mathbb{X}_{s,u}-\mathbb{X}_{u,t}=(\delta X)_{s,u}\otimes (\delta X)_{u,t},
\end{align*}
for $s\leq u\leq t$, where we write $(\delta X)_{s,u}:=X_u-X_s$ for an arbitrary path. Here, we denote by $C^\gamma$ the space of $\gamma$-H\"older continuous paths, as well as by $C^{2\gamma}_2$ the space of $2\gamma$-Hölder continuous two-parameter functions. We further set $\rho_{\gamma,[s,t]}(\mathbf{X}):=1+[X]_{\gamma,\R^d.[s,t]}+[\mathbb{X}]_{2\gamma,\R^d\otimes \R^d,[s,t]}$, where $[\cdot]$ denotes the Hölder semi-norm. If it is clear from the context, we omit the interval in the index.

Since we consider parabolic RPDEs, we work with families $(E_\alpha)_{\alpha\in \R}$ of interpolation spaces endowed with the norms $(|\cdot|_{\alpha})_{\alpha\in \R}$, such that $E_\beta \hookrightarrow E_\alpha$ for $\alpha<\beta$ and the following interpolation inequality holds
\begin{align}\label{interpolation:ineq}
	|x|^{\alpha_3-\alpha_1}_{\alpha_2} \lesssim |x|^{\alpha_3-\alpha_2}_{\alpha_1} |x|^{\alpha_2-\alpha_1}_{\alpha_3},
\end{align}
for $\alpha_1\leq \alpha_2\leq \alpha_3$ and $x\in  E_{\alpha_3}$. Tailored to this setting, we define the notion of a controlled rough path according to such a family of function spaces, as introduced in \cite{GHT21}. 
\begin{definition}\label{def:crp}
	Let $\alpha\in\R$. We call a pair $(y,y')$ a controlled rough path if $$(y,y')\in C([0,T];E_\alpha) \times (C([0,T];E_{\alpha-\gamma} ) \cap C^{\gamma}([0,T];E_{\alpha-2\gamma}))^d$$ and the remainder  
	\begin{align*}
		(s,t)\in \Delta_{[0,T]}\mapsto R^y_{s,t}:= (\delta y)_{s,t} -y'_s \cdot (\delta X)_{s,t}
	\end{align*}
	belongs to $ C_2^{\gamma}(\Delta_{[0,T]};E_{\alpha-\gamma})\cap C_2^{2\gamma}(\Delta_{[0,T]};E_{\alpha-2\gamma})$, where $y'_s \cdot (\delta X)_{s,t}=\sum_{i=1}^d y_s^{i,\prime}(\delta X^{i})_{s,t}$. The component $y'$ is referred to as the Gubinelli derivative of $y$. 
	The space of controlled rough paths is denoted by $\cD^{\gamma}_{\mathbf{X},\alpha}([0,T])$ and endowed with the norm $\|\cdot,\cdot\|_{\cD^{\gamma}_{\mathbf{X},\alpha}([0,T])}$ given by
	\begin{align}\label{g:norm}
		\begin{split}
			\|y,y'\|_{\cD^{\gamma}_{\mathbf{X},\alpha}([0,T])}:= \left\|y \right\|_{\infty,E_\alpha} 
			+ \|y' \|_{\infty,E^d_{\alpha-\gamma}}
			+ \left[y'\right]_{\gamma,E^d_{\alpha-2\gamma}}
			+\left[R^y \right]_{\gamma,E_{\alpha-\gamma}}    + \left[R^y \right]_{2\gamma,E_{\alpha-2\gamma}},
		\end{split}
	\end{align}
	where $|y'_s|_{E_\alpha^d} \coloneqq \sup\limits_{1\leq i\leq d}|y_s^{i,\prime}|_{\alpha}$.
\end{definition}
In this context, we mostly omit the time dependence if it is clear from the context, meaning that we write $\mathcal{D}^\gamma_{\mathbf{X},\alpha}([0,T])=\mathcal{D}^\gamma_{\mathbf{X},\alpha}$ and $C^\gamma(E_\alpha)=C^\gamma([0,T];E_\alpha)$. Also, we write for simplicity $\|y\|_{\infty,\alpha}:=\|y\|_{\infty,E_\alpha}, \|y^\prime\|_{\infty,\alpha-\gamma}:=\|y^\prime\|_{\infty,E^d_{\alpha-\gamma}}$ and $[y']_{\gamma,\alpha-2\gamma}:=[y']_{\gamma,E^d_{\alpha-2\gamma}}$ and analogously for the remainder. Then, the first index always indicates the time regularity, and the second one stands for the space regularity.
\begin{remark}\label{rem:nDimCRP}
	If the path component $y=(y^k)_{k=1,\ldots,d}$ is $d$-dimensional, the resulting Gubinelli derivative $y^\prime:=(y^{kl,\prime})_{0\leq k,l\leq d}$ is matrix valued. We then write for simplicity $(y,y^\prime):=(y^k,y^{k,\prime})_{1\leq k\leq d}\in (\cD^\gamma_{\mathbf{X},\alpha})^d$.
\end{remark}
\begin{remark}\label{rem:Hcontpath}
	Let $(y,y')\in \cD^{\gamma}_{\mathbf{X},\alpha}$. Then we have for $i=1,2$
	\[ [y]_{\gamma,\alpha-i\gamma} \leq \|y'\|_{\infty,\alpha-i\gamma} [X]_{\gamma,\R^d} + [R^y]_{\gamma,\alpha-i\gamma} \leq \rho_{\gamma,[0,T]}(\mathbf{X}) \|y,y'\|_{\cD^\gamma_{\mathbf{X},\alpha}}. \]
\end{remark}
Before we define the rough convolution, let us recall some sufficient conditions on the linear part to ensure the existence of an evolution family. 
\begin{itemize}\namedlabel{ass:A}{\textbf{(A)}}
	\item[\textbf{(A1)}\namedlabel{ass:A1}{\textbf{(A1)}}] The family $(A(t))_{t\in [0,T]}$ consists of closed and densely defined operators $A(t):E_1\to E_0$ on a time independent domain $D(A)=E_1$. Furthermore, they have bounded imaginary powers, i.e. there exists $C>0$ such that
	\begin{align*}
		\sup_{|s|\leq 1} \|(-A(t))^{is}\|_{\mathcal{L}(D(A))}\leq C
	\end{align*}
	for every $t,s\in \R$, where $i$ denotes the imaginary unit. 
	\item[\textbf{(A2)}\namedlabel{ass:A2}{\textbf{(A2)}}] There exists $\vartheta\in (\pi,\frac{\pi}{2})$ and a constant $\change{M_0>0}$ such that $\Sigma_{\vartheta}:=\{z\in \C~:~|\arg(z)|<\vartheta\}\subset R(A(t))$ where $R(A(t))$ denotes the resolvent set of $A(t)$ and
	\begin{align*}
		\lVert (z-A(t))^{-1}\rVert_{\cL(E_k)}\leq \frac{\change{M_0}}{1+|z|},
	\end{align*}
	for all $z\in \Sigma_\vartheta, \change{k=0,1}$ and $t\in [0,T]$. \change{Further assume there exists a constant $M_1>0$ such that 
		\begin{align*}
			\lVert (z-A(t))^{-1}\rVert_{\cL(E_0;E_1)}\leq M_1.
	\end{align*}}
	\item[\textbf{(A3)}\namedlabel{ass:A3}{\textbf{(A3)}}] There exists a $\varrho\in(0,1]$ such that 
	\begin{align*}
		\lVert A(t)-A(s)\rVert_{\cL(E_1;E_0)}\lesssim |t-s|^\varrho,
	\end{align*}
	for all $s,t\in[0,T]$.
\end{itemize}
These conditions are known as the Kato-Tanabe assumptions and are often used in the context of non-autonomous evolution equations, see for example~\cite[p.~150]{Pazy} and \cite{Ama86}. In particular, \ref{ass:A2} implies that the operator $A(t)$ is sectorial. Therefore, we can define $E_\alpha:=D((-A(t))^\alpha)$ endowed with the norm $|\cdot|_\alpha:=|(-A(t))^\alpha \cdot|_{E_0}$. 
Under these assumptions, we obtain an evolution family which is a generalization of a semigroup in the non-autonomous setting.
\begin{theorem}{\em(\cite[Theorem 2.3]{AT87})}\label{thm:OpFam}
	Let $(A(t))_{t\in [0,T]}$ satisfy Assumption \ref{ass:A1}-\ref{ass:A3}. Then there exists a unique parabolic evolution family $(U_{t,s})_{0\leq s\leq t\leq T}$ of linear operators $U_{t,s}:E_0\to E_0$ such that the following properties hold:
	\begin{itemize}
		\item[i)] For all $0\leq r\leq s\leq t\leq T$ we have
		\begin{align*}
			U_{t,s}U_{s,r}=U_{t,r}
		\end{align*}
		as well as $U_{t,t}=\Id_{E_0}$.
		\item[ii)] The mapping $(s,t)\mapsto U_{t,s}$ is strongly continuous. 
		\item[iii)] For $s\leq t$ we have the identity
		\begin{align*}
			\frac{\txtd}{\txtd t} U_{t,s}=A(t)U_{t,s}.
		\end{align*}
	\end{itemize}
\end{theorem}
From now on, we say $(A(t))_{t\in [0,T]}$ satisfies Assumption \ref{ass:A} if $(A(t))_{t\in [0,T]}$ satisfies \ref{ass:A1}-\ref{ass:A3} on $(E_\alpha,E_{\alpha+1})$ for every $\alpha>0$. Then the resulting evolution family satisfies for $t>s$ similar estimates as in the autonomous case, i.e. there exist constants $C_{\alpha,\sigma_1}, \tilde{C}_{\alpha,\sigma_2}$ such that
\begin{align}\label{PTR}
	\begin{split}
		|(U_{t,s}-\text{Id}) x|_{\alpha}&\leq C_{\alpha,\sigma_1} |t-s|^{\sigma_1} |x|_{\alpha+\sigma_1},\\
		|U_{t,s}x|_{\alpha+\sigma_2}&\leq \widetilde{C}_{\alpha,\sigma_2}|t-s|^{-\sigma_2}|x|_\alpha,
	\end{split}
\end{align}
for $\sigma_2\in [k_-,k_+]$ and $\sigma_1\in [0,1]$, where $k_-<k_+$ are fixed natural numbers and the constants {$C_{\alpha,\sigma_1},\widetilde{C}_{\alpha,\sigma_2}>0$ in \eqref{PTR}} may depend on $k_-,k_+$\change{, see \cite[Theorem 3.9]{GHT21}.}
\begin{remark}
	\begin{itemize}
		\item[i)]  We suppose in Assumption \ref{ass:A1} that the domain of $A(t)$ is independent of $t$. However, this is not enough to ensure that the fractional power spaces do not depend on time. Since we further assume that $A(t)$ has bounded imaginary powers, the fractional power spaces can be identified using complex interpolation \cite[Theorem V.1.5.4]{Amann1995}. This means that for any $\alpha\in(0,1)$ we have $E_\alpha=[E_0,D(A)]_\alpha=D((-A(t))^\alpha)$ and therefore $E_\alpha$ does not depend on time.~For examples in this setting, we refer to Section~\ref{sec:app}. 
		\item[ii)] It is also possible to consider non-autonomous evolution equations in the context of time-dependent domains. In this setting, the stated Kato-Tanabe conditions \ref{ass:A2}-\ref{ass:A3} are not enough to ensure the existence of a parabolic evolution family. With stronger conditions, for example, under the assumptions of Acquistapace-Terreni \cite[Hypothesis I-II]{AT87}, a similar statement as in Theorem \ref{thm:OpFam} holds. For a detailed discussion on different assumptions for non-autonomous evolution equations, see \cite[Section 7]{AT87} and also \cite{Acquistapace, Yagi}. 
		\change{\item iii) As a convention, for $s \leq t$, we write $U_{t,s}$ to denote the evolution family, and $(\delta X)_{s,t}$ and $\mathbb{X}_{s,t}$ to denote the corresponding components of the rough path.} 
	\end{itemize}
\end{remark}
Coming back to the equation \eqref{Main_Equation}, we need to define the rough convolution in the sense of \cite{GHT21} in order to make sense of its mild forumulation.~Therefore, we define for $s<t$ the partition $\pi$ of $[s,t]$. 
Then it was shown in \cite[Theorem 4.1]{GHT21} that for  $(y,y^\prime)\in(\mathcal{D}_{\mathbf{X},\alpha}^{\gamma}([s,t]))^{d}$ the rough convolution
\begin{align}\label{SEW}
	\int_{s}^{t}U_{t,r}y_{r}~\mathrm{d}\mathbf{X}_{r}:=\lim_{|\pi|\rightarrow 0} \sum_{[u,v]\in\pi} U_{t,u}\left(y_{u}\cdot (\delta X )_{u,v}+y^{\prime}_{u}\circ\mathbb{X}_{u,v}\right).
\end{align}
exists, where $|\pi| = \max_{[u,v]\in \pi}|v-u|$ is the mesh size, and satisfies the estimate
\begin{align}\label{est:RPIntegral}
	\left\|\int_s^\cdot U_{\cdot,r} y_r~\txtd \mathbf{X}_r,y\right\|_{\cD^{\gamma}_{\mathbf{X},\alpha+\sigma}([s,t])}&\lesssim \rho_{\gamma,[s,t]}(\mathbf{X}) (|y_s|_\alpha+|y^\prime_s|_{\alpha-\gamma}+(t-s)^{\gamma-\sigma}\|y,y^\prime\|_{\cD^\gamma_{\mathbf{X},\alpha}([s,t])}),
\end{align}
with $\sigma\in [0,\gamma)$.  Here we use 
\begin{align*}
	y^{\prime}_{u}\circ\mathbb{X}_{u,v} \coloneqq \sum_{1\leq k,l\leq d} y^{kl,\prime}_{u}\mathbb{X}^{kl}_{u,v}.
\end{align*}
Given~\eqref{SEW} we can define a solution concept for \eqref{Main_Equation}. 
\begin{definition}\label{Definition of solution}
	We say that $(\change{u},\change{u}^\prime)\in \mathcal{D}_{\mathbf{X},\alpha}^{\gamma}([0,T])$, solves equation \eqref{Main_Equation} with initial datum $\change{u}_{0}\in E_\alpha$ if the path component satisfies the mild formulation 
	\begin{align}\label{eq:MildFormulation}
		\change{u}_t=U_{t,0}\change{u}_0+\int_{0}^{t}U_{t,r}F(r,\change{u}_r)~\mathrm{d}r+\int_{0}^{t}U_{t,r}G(r,\change{u}_r)~\mathrm{d}\mathbf{X}_r
	\end{align}
	with Gubinelli derivative $\change{u}_t^\prime=G(t,\change{u}_t)$ for $t\in [0,T]$.
\end{definition}
The assumptions on the nonlinearities $F$ and $G$ will be specified in Section \ref{integrable}, where we will also prove that the rough convolution in \eqref{eq:MildFormulation} is well-defined. 

To obtain an integrable bound as in \cite{GVR25}, which is a key part of our computations, we need to replace the H\"older-norms of the noise, appearing in $\rho_{\gamma,[0,T]}(\mathbf{X})$, by suitable controls which will lead to better integrability conditions. The controls are specified in the following definition.
\begin{definition}\label{def:control}
	For $0\leq \eta<\gamma$ define the function $W_{\mathbf{X},\gamma,\eta} : \Delta_{[0,T]}\rightarrow \mathbb{R}$ through
	\begin{align}
		W_{\mathbf{X},\gamma,\eta}(s,t) \coloneqq \sup_{\pi\subset [s,t]} \left\{ \sum_{[u,v]\in\pi}(v-u)^{\frac{-\eta}{\gamma-\eta}}\big{[}|(\delta X)_{u,v}|^{\frac{1}{\gamma-\eta}}+|\mathbb{X}_{u,v}|^{\frac{1}{2(\gamma-\eta)}} \big{]} \right\}.
	\end{align}
	where the supremum is taken over all partitions $\pi$ of $[s,t]$ and $|\cdot|$ is the norm in $\R^d$ respectively $\R^d\otimes\R^d$. It is easy to show that $W$ is continuous and satisfies the subadditivity property, i.e.~for $s\leq \change{r}\leq t$ \change{we have}
	\begin{align*}
		W_{\mathbf{X},\gamma,\eta}(s,\change{r})+W_{\mathbf{X},\gamma,\eta}(\change{r},t)\leq 	W_{\mathbf{X},\gamma,\eta}(s,t).
	\end{align*}
\end{definition}


\section{Existence and integrable bounds of global solutions}\label{integrable}
\subsection{\texorpdfstring{Local and global well-posedness}{}}
In this section, we examine the solvability of the non-autonomous RPDE, allowing nonlinearities with explicit time dependencies. To the best of our knowledge, there are only a few results on non-autonomous RPDEs. In \cite{GHT21}, the linear part has a time-dependence, and in \cite{HN24}, the authors investigated quasilinear equations with a time-dependent drift term.  Recently, \cite{Tappe} investigates equations that are not parabolic and uses a different approach for the space of controlled rough paths, which does not require an analytic semigroup but also allows time-dependent data. In this article, we stick to the approach of \cite{GHT21}, since this fits nicely in our setting of parabolic equations, and extend this approach to non-autonomous drift and diffusion terms.

Thus, we must first examine the behavior of the controlled rough paths in terms of Definition \ref{def:crp} by composition with time-dependent nonlinearities. For this, we state the following assumptions on the coefficients.
\begin{itemize}
	\item[\textbf{(F)}\namedlabel{ass:F}{\textbf{(F)}}] There exists $\delta\in[0,1)$ such that $F:[0,T]\times E_{\alpha}\to E_{\alpha-\delta}$ is Lipschitz continuous in $E_\alpha$, uniformly in $[0,T]$. That means, for every $t\in [0,T]$ there exists a constant $L_{F,t}>0$ such that $F(t,\cdot)$ is Lipschitz and $L_F\coloneqq \sup_{t\in [0,T]} L_{F,t}<\infty$. In particular, we have  for all $x,y\in E_\alpha$ and $t\in [0,T]$ that
	\begin{align*}
		|F(t,x)-F(t,y)|_{\alpha-\delta}&\leq L_F |x-y|_{\alpha},\\
		|F(t,x)|_{\alpha-\delta}&\leq C_F (1+|x|_\alpha),
	\end{align*}
	where $C_F:=\max\{L_F,\sup_{t\in [0,T]}|F(t,0)|_{\alpha-\delta}\}<\infty$. 
	
	\item[\textbf{(G1)}\namedlabel{ass:G1}{\textbf{(G1)}}] There exists $\sigma<\gamma$ such that $G:[0,T]\times E_{\alpha-i\gamma}\to E^d_{\alpha-i\gamma-\sigma}$ for $i=0,1,2$ satisfies the following conditions:
	\begin{itemize}
		\item[i)] For every $t\in [0,T]$, $G(t,\cdot)$ is bounded and three times continuously Fréchet differentiable with  bounded derivatives uniformly in time. 
		\item[ii)] For every $x\in E_{\alpha-i\gamma}$, $G(\cdot,x)$, as well as $\txtD_2G(\cdot,x),\txtD_2^2G(\cdot,x)$ and $\txtD_2^3 G(\cdot,x)$, are H\"older continuous with parameter $ 2\gamma$. We further assume that these H\"older constants are uniform in $E_{\alpha-i\gamma}$.
	\end{itemize}
	We set $C_G$ as the maximum of all constants involving the bounds of $G$ and its derivatives.
	\item[\textbf{(G2)}\namedlabel{ass:G2}{\textbf{(G2)}}] For every $t\in [0,T]$, the derivative of $$\txtD_2G(t,\cdot)G(t,\cdot):E_{\alpha-2\gamma-\sigma}\to E^{d\times d}_{\alpha-\gamma}$$ is bounded.
\end{itemize}
\begin{remark}
	\begin{itemize}
		\item[i)] To prove the local existence, it is enough to assume \ref{ass:G1}. In fact, \ref{ass:G1} is even stronger than actually necessary for the existence of a local solution, the boundedness of $G$ could be dropped, see for example \cite[Theorem 2.15]{GHT21}. Since we need an integrable bound for the solution, we need that $G$ is bounded, see also Remark \ref{rem:Gbound}. 
		\item[ii)] To ensure the existence of a global-in-time solution, we must also assume \ref{ass:G2} as originally developed in \cite{HN22}. Note that it is possible to prove that \ref{ass:G1} implies \ref{ass:G2} due to the boundedness of $G$. However, we have decided to state \ref{ass:G2} separately in order to emphasize an additional condition that is required to obtain a global solution. 
	\end{itemize}
\end{remark}
\begin{lemma}\label{lem:NonAutoCRP}
	Let $(y,y^\prime)\in \mathcal{D}^\gamma_{\mathbf{X},\alpha}$ be a controlled rough path, and $G$ a nonlinearity satisfying \ref{ass:G1}. Then we have $(G(\cdot,y),\txtD_2G(\cdot,y)y^\prime)\in (\mathcal{D}^\gamma_{\mathbf{X},\alpha})^d$, where we write $\txtD_2G(\cdot,y)y^\prime:=(\txtD_2G^k(\cdot,y)y^{l,\prime})_{1\leq k,l\leq d}$, see Remark \ref{rem:nDimCRP}.
\end{lemma}
\begin{proof}
	For the sake of completeness, we provide a proof for pointing out the main differences from the autonomous case~\cite[Lemma 4.7]{GHT21}.~Without loss of generality, we assume $d=1$ since the generalization can be made componentwise.
	We first note that $G(\cdot,y)\in C(E_{\alpha-\sigma})$ due to~\ref{ass:G1} i), as well as
	\begin{align*}
		\lVert \txtD_2G(\cdot,y)y^\prime \rVert_{\infty,\alpha-\gamma-\sigma}\lesssim \lVert y^\prime \rVert_{\infty,\alpha-\gamma}\lesssim \lVert y, y^\prime\rVert_{\mathcal{D}^\gamma_{\mathbf{X},\alpha}}.
	\end{align*}
	To establish the Hölder continuity of the Gubinelli derivative, we use~\ref{ass:G1} ii) to obtain 
	\begin{align*}
		| \txtD_2G(t,y_t)y_t^\prime&-\txtD_2G(s,y_s)y_s^\prime |_{\alpha-2\gamma-\sigma}\leq | (\txtD_2G(t,y_t)-\txtD_2G(t,y_s))y_t^\prime |_{\alpha-2\gamma-\sigma}\\
		&+| (\txtD_2G(t,y_s)-\txtD_2G(s,y_s))y_t^\prime |_{\alpha-2\gamma-\sigma}+
		| \txtD_2G(s,y_s)(y_t^\prime-y_s^\prime) |_{\alpha-2\gamma-\sigma}\\
		&\lesssim |(\delta y)_{s,t}|_{\alpha-2\gamma}|y_t^\prime|_{\alpha-2\gamma} + (t-s)^{2\gamma} |y^\prime_{t}|_{\alpha-2\gamma}+|(\delta y^\prime)_{s,t}|_{\alpha-2\gamma}\\
		&\lesssim (t-s)^\gamma \rho_{\gamma,[0,T]}(\mathbf{X})\lVert y,y^\prime\rVert_{\mathcal{D}^\gamma_{\mathbf{X},\alpha}}(1+\lVert y,y^\prime\rVert_{\mathcal{D}^\gamma_{\mathbf{X},\alpha}})+(t-s)^{2\gamma} \lVert y,y^\prime\rVert_{\mathcal{D}^\gamma_{\mathbf{X},\alpha}},
	\end{align*}
	which leads to
	\begin{align}\label{ineq:quadraticTerm}
		\lVert \txtD_2G(\cdot,y)y^\prime\rVert_{\gamma,\alpha-2\gamma-\sigma}\lesssim  \lVert y,y^\prime\rVert_{\mathcal{D}^\gamma_{\mathbf{X},\alpha}}(1+\lVert y,y^\prime\rVert_{\mathcal{D}^\gamma_{\mathbf{X},\alpha}})+T^{\gamma} \lVert y,y^\prime\rVert_{\mathcal{D}^\gamma_{\mathbf{X},\alpha}}.
	\end{align}
	\change{A straightforward computation leads to the following representation of the remainder
		\begin{align*}
			R^{G(\cdot,y)}_{s,t}&=G(t,y_t)-G(s,y_s)-\txtD_2G(s,y_s)\big(y_s^\prime (\delta X_{s,t})\big)\\
			&=G(t,y_t)-G(s,y_t)+G(s,y_t)-G(s,y_s)-\txtD_2G(s,y_s)\big(y_s^\prime (\delta X_{s,t})\big)\\
			&=G(t,y_t)-G(s,y_t)+G(s,y_t)-G(s,y_s)-\txtD_2G(s,y_s)\big((\delta y_{s,t})-R_{s,t}^y\big)\\
			&=G(t,y_t)-G(s,y_t)+\txtD_2G(s,y_s)R_{s,t}^y\\
			& +\int_0^1\int_0^1 \tilde{r}\txtD_2^2G(s,y_s+r\tilde{r} (\delta y_{s,t}))(\delta y_{s,t})(\delta y_{s,t})~\txtd r \txtd\tilde{r},
	\end{align*}}
	where we used the $2\gamma$-H\"older continuity of $G(\cdot,x)$ to estimate the difference $G(t,y_t)-G(s,y_t)$.~In this case we obtain
	\begin{align*}
		\Vert R^{G(\cdot,y)}\Vert_{i\gamma,\alpha-i\gamma-\sigma}\change{\lesssim 1+\varrho_{\gamma,[0,T]}(\mathbf{X})^2\|y,y^\prime\|_{\cD^\gamma_{\mathbf X,\alpha}}\Big(1+\|y,y^\prime\|_{\cD^\gamma_{\mathbf X,\alpha}}\Big),}
	\end{align*}
	\change{for $i=1,2$.}
\end{proof}
\begin{remark}
	Instead of $\txtD_2G(t,y_t)y_t^\prime$ as the Gubinelli derivative, we could also choose $DG(t,y_t)\circ (1,y_t^\prime)=\txtD_1G(t,y_t)+\txtD_2G(t,y_t)y_t^\prime$, \change{provided that $G$ is differentiable with respect to time.}
\end{remark}

The computations to obtain a solution to \eqref{Main_Equation} are similar to those in \cite{GHT21} for the local existence, and \cite{HN22} for the global existence. For the sake of completeness, we give an outline of the proofs, highlighting the main differences from the autonomous case. To simplify the presentation, we assume that $T<1$. 

\begin{theorem}\label{ex:nona}
	Fix $\alpha\in \R, \gamma\in (\frac{1}{3},\frac{1}{2}]$. Let $(A(t))_{t\in [0,T]},F$ and $G$ satisfy Assumption \ref{ass:A},\ref{ass:F} and \ref{ass:G1}. Then there exists for every $\change{u}_0\in E_\alpha$ a time $T^*\leq T$ and an unique controlled rough path $(\change{u},\change{u}^\prime)\in \mathcal{D}^\gamma_{\mathbf{X},\alpha}([0,T^*))$ such that $\change{u}^\prime_t=G(t,y_t)$ and 
	\begin{align*}
		\change{u}_t:=U_{t,0}\change{u}_0+\int_0^t U_{t,r}F(r,\change{u}_r)~\txtd r+\int_0^t U_{t,r}G(r,\change{u}_r)~\txtd\mathbf{X}_r,
	\end{align*}
	for $t\in[0,T^*]$.
\end{theorem}
\begin{proof}
	To obtain a mild solution for \eqref{Main_Equation}, we seek a fixed point of
	\begin{align*}
		P_T(y,y^\prime):=\left(U_{\cdot,0}y_0+\int_0^\cdot U_{\cdot,r}F(r,y_r)~\txtd r+\int_0^\cdot U_{\cdot,r}G(r,y_r)~\txtd\mathbf{X}_r, G(\cdot,y)\right).
	\end{align*}
	Instead of proving the existence of a fixed point in $\mathcal{D}^\gamma_{\mathbf{X},\alpha}$, we define for $\gamma^\prime<\gamma$ 
	\begin{align*}
		B_T(y_0):=\left\{(y,y^\prime)\in\mathcal{D}^{\gamma^\prime}_{\mathbf{X},\alpha}([0,T])~\colon~(y_0,y_0^\prime)=(y_0,G(0,y_0))~\text{and}~\lVert y-\zeta,y^\prime-\zeta^\prime\rVert_{\mathcal{D}^{\gamma^\prime}_{\mathbf{X},\alpha}([0,T])}<1\right\},
	\end{align*}
	where ${\zeta}_t:=U_{t,0}y_0+\int_0^t U_{t,r}G(r,y_0)~\txtd \mathbf{X}_r$ and ${\zeta}_t^\prime:=G(t,y_0)$.  Similar to \cite{GHT21} it is possible to show that there exists a time $T^*>0$ such that $P_{T^*}:B_{T^*}\to B_{T^*}$ is invariant and contractive. Then Banach's fixed point theorem ensures the existence of $(\change{u},\change{u}^\prime)\in B_{T^*}$ such that $\change{u}$ satisfies \eqref{eq:MildFormulation}.
\end{proof}
To get a global-in-time solution, we use the same strategy as established in \cite{HN22}. This means that we exploit the fact that the solution of \eqref{Main_Equation} has the form $(y,G(\cdot,y))$ for $(y,y')\in \cD^{\gamma}_{\mathbf{X},\alpha}$. Therefore, we obtain a bound on the solution which does not involve quadratic terms such as \eqref{ineq:quadraticTerm}. 
\begin{lemma}
	Let $G$ satisfy \ref{ass:G1}-\ref{ass:G2} and $(y,G(\cdot,y))\in \mathcal{D}^\gamma_{\mathbf{X},\alpha}$. Then $(G(\cdot,y),\txtD_2G(\cdot,y)G(\cdot,y))\in (\mathcal{D}^\gamma_{\mathbf{X},\alpha-\sigma})^d$ and we have the \change{estimate}
	\begin{align*}
		\lVert G(\cdot,y),\txtD_2G(\cdot,y)\rVert_{(\mathcal{D}^\gamma_{\mathbf{X},\alpha})^d}\lesssim 1+\lVert y,G(\cdot,y)\rVert_{\mathcal{D}^\gamma_{\mathbf{X},\alpha}}.
	\end{align*}
\end{lemma}
\begin{proof}
	Due to \ref{ass:G2} we have the Lipschitz type estimate
	\begin{align*}
		|(\txtD_2G(t,x)-\txtD_2G(t,y))G(t,x)|_{\alpha-2\gamma-\sigma}\lesssim |x-y|_{\alpha-\gamma},
	\end{align*}
	for every $x,y\in E_{\alpha-\gamma}$. Using that the Gubinelli derivative is given by $G(\cdot,y)$, we conclude as in \cite[Lemma 3.6]{HN22}.
\end{proof}
With this essential estimate, it is now possible to state the existence of a global-in-time solution to \eqref{Main_Equation}. We omit the proof of this theorem, since it is similar to \cite[Theorem 3.9]{HN22}.
\begin{theorem}\label{thm:GlobEx}
	Fix $\alpha\in \R, \gamma\in (\frac{1}{3},\frac{1}{2}], \sigma \in [0,\gamma)$ and $\delta\in[0,1)$.  Let $(A(t))_{t\in [0,T]},F$ and $G$ satisfy Assumption \ref{ass:A},\ref{ass:F} and \ref{ass:G1}-\ref{ass:G2}. Then there exists for every $\change{u}_0\in E_\alpha$ an unique controlled rough path $(\change{u},\change{u}^\prime)\in \mathcal{D}^\gamma_{\mathbf{X},\alpha}([0,T])$ such that $\change{u}^\prime_t=G(t,\change{u}_t)$ and 
	\begin{align*}
		\change{u}_t=U_{t,0}\change{u}_0+\int_0^t U_{t,r}F(r,\change{u}_r)~\txtd r+\int_0^t U_{t,r}G(r,\change{u}_r)~\txtd\mathbf{X}_r,~\quad \text{ for } t\in[0,T].
	\end{align*}
\end{theorem}
\subsection{\change{Sewing lemma revisited}}
In \cite{GVR25}, the existence of moments of all orders was shown for the controlled rough path norm of the solution of an autonomous semilinear rough partial differential equation with a bounded diffusion coefficient. Here we extend the results to the non-autonomous case and also extend the class of possible rough inputs.~The main idea is to accordingly modify the sewing lemma replacing the H\"older norms of the rough input by controls as in Definition \ref{def:control}, since such controls have better integrability properties compared to H\"older norms. This is the topic of the next lemma, which is a generalization of \cite[Proposition 2.7]{GVR25} to the non-autonomous setting. 
\begin{lemma}\label{lem:IneqRV}
	Let $(y,y^\prime)\in \mathcal{D}^\gamma_{\mathbf{X},\alpha}$, $\sigma\in[0,\frac{1-\gamma}{2})$ and choose $\varepsilon>0$ such that $\sigma+\varepsilon<\gamma$. Then we obtain for $i=0,1,2$ the inequality
	\begin{align}\label{ineq:VarzanehRiedel}
		\begin{split}
			&\left|\int_s^t U_{t,r}G(r,y_r)~\txtd \mathbf{X}_r-U_{t,s}\left(G(s,y_s)\cdot (\delta X)_{s,t}-\txtD_2G(s,y_s)G(s,y_s)\circ \mathbb{X}_{s,t}\right)\right|_{\alpha-i\gamma}\\
			&\lesssim (t-s)^{i\gamma} \lVert y,y^\prime\rVert_{\mathcal{D}^\gamma_{\mathbf{X},\alpha}}\max\{(t-s)^\varepsilon W_{\mathbf{X},\gamma,\sigma+\varepsilon}(s,t)^{\gamma-\sigma-\varepsilon},(t-s)^{2\varepsilon} W_{\mathbf{X},\gamma,\sigma+\varepsilon}(s,t)^{2(\gamma-\sigma-\varepsilon)}\} \\
			&+\max\limits_{k=1,2,3}\{(t-s)^{i\gamma+k(\gamma-\sigma)}\} P([X]_{\gamma,\R^d},[\mathbb{X}]_{2\gamma,\R^{d}\otimes \R^d}),
		\end{split}
	\end{align}
	where $P(\cdot,\cdot)$ is a polynomial.
\end{lemma}
\begin{proof}
	We define for $s\leq u\leq v\leq t$
	\begin{align*}
		\Xi_{s,t}^{u,v}:=U_{t,u}\left(G(u,y_{u})\cdot (\delta X)_{u,v}+\txtD_2G(u,y_u)G(u,y_{u})\circ\mathbb{X}_{u,v}\right)
	\end{align*}
	and  consider the dyadic partition $\pi^k:=\{\tau^m_k:=s+\frac{m}{2^k}(t-s)~\colon~ 0\leq m\leq 2^k\}$ of $[s,t]$. Then we have
	\begin{align*}
		&\left|\int_s^t U_{t,r}G_r(y_r)~\txtd \mathbf{X}_r-U_{t,s}\left(G(s,y_s)\cdot (\delta X)_{s,t}-\txtD_2G(s,y_s)G(s,y_s)\circ \mathbb{X}_{s,t}\right)\right|_{\alpha-i\gamma}\\
		&\leq \sum_{k\geq 0} \sum_{0\leq m < 2^k} |\Xi_{s,t}^{\tau_{k}^{m},\tau_{k+1}^{2m+1}}+\Xi_{s,t}^{\tau_{k+1}^{2m+1},\tau_{k}^{m+1}}-\Xi_{s,t}^{\tau_{k}^{m},\tau_{k}^{m+1}}| .
	\end{align*}
	Using Chen's relation and Taylor's theorem, we obtain 		for $s\leq u\leq v\leq w\leq t$
	\begin{align}\label{eq:IntApprox}
		\begin{split}
			\Xi_{s,t}^{u,v}&+\Xi_{s,t}^{v,w}-\Xi_{s,t}^{u,w}\\
			&= U_{t,u}\left(G(v,y_v)-G(u,y_v)\right)(\delta X)_{v,w}\\
			&+U_{t,u}\left(\int_0^1 \int_0^1 \change{\tilde{r}} \txtD_2^2 G(u,y_u+r\tilde{r}(\delta y)_{u,v})(G(u,y_u)\cdot (\delta X)_{u,v})(G(u,y_u)\cdot (\delta X)_{u,v})~\txtd r\txtd\tilde{r}\right)\cdot (\delta X)_{v,w}\\
			&+U_{t,u}\left(\int_0^1 \int_0^1 \change{\tilde{r}} \txtD_2^2 G(u,y_u+r\tilde{r}(\delta y)_{u,v})(G(u,y_u)\cdot (\delta X)_{u,v})(R^y_{u,v})~\txtd r\txtd\tilde{r}\right)\cdot (\delta X)_{v,w}\\
			&+U_{t,u} \left(\int_0^1 \txtD_2G(u,y_u+r (\delta y)_{u,v})R^y_{u,v}~\txtd r\right)\cdot (\delta X)_{v,w}\\
			&+U_{t,u} \left(\txtD_2G(v,y_v)\int_0^1 \txtD_2G(v,y_u+r (\delta y)_{u,v})G(u,y_u)\cdot (\delta X)_{u,v}~\txtd r)\right)\circ \mathbb{X}_{v,w}\\
			&+U_{t,u} \left(\txtD_2G(v,y_v)\int_0^1 \txtD_2G(v,y_u+r (\delta y)_{u,v})[R^y_{u,v}]~\txtd r)\right)\circ \mathbb{X}_{v,w}\\
			&+U_{t,u}\left(\int_0^1\txtD_2^2G(v,y_u+r (\delta y)_{u,v})G(u,y_u)\cdot (\delta X)_{u,v}G(v,y_u)~\txtd r\right)\circ \mathbb{X}_{v,w}\\
			&+ \change{U_{t,u}\left(\int_0^1\txtD_2^2G(v,y_u+r (\delta y)_{u,v})R^y_{u,v}G(v,y_u)~\txtd r\right)\circ \mathbb{X}_{v,w}}\\	&+U_{t,u}\left((\txtD_2G(v,y_u)-\txtD_2G(u,y_u))G(v,y_u)+\txtD_2G(u,y_u)(G(v,y_u)-G(u,y_u))\right)\circ \mathbb{X}_{v,w}\\
			&-U_{t,v}(U_{v,u}-\Id)\left(G(v,y_v)\cdot (\delta X)_{v,w}+\txtD_2G(v,y_v)G(v,y_v)\circ \mathbb{X}_{v,w}\right).
		\end{split}
	\end{align}
	
	We show how to treat the term in the first line, since the other terms {can be handled} by analogous arguments. We refer to \cite[Lemma 2.5]{GVR25} for similar computations.\\
	For $i=0,1,2$ we obtain using the smoothing property of the evolution family \eqref{PTR}, the $\gamma$-H\"older continuity of $X$ and the $2\gamma$-H\"older continuity of $G(\cdot,y)$ 
	\begin{align*}
		\sum_{k\geq 0}\sum_{0\leq m< 2^k}&|U_{t,u}\left(G(v,y_v)-G(u,y_v)\right)(\delta X)_{v,w}|_{\alpha-i\gamma}\\
		&\lesssim \sum_{k\geq 0}\sum_{0\leq m< 2^k} (t-u)^{(i-2)\gamma-\sigma}(w-v)^\gamma |G(v,y_v)-G(u,y_v)|_{\alpha-2\gamma-\sigma}\\
		&\lesssim \sum_{k\geq 0}\sum_{0\leq m< 2^k} (t-u)^{(i-2)\gamma-\sigma}(w-v)^\gamma (v-u)^{2\gamma}\\
		&\lesssim (t-s)^{i\gamma+\gamma-\sigma} \sum_{k\geq 0}\sum_{0\leq m< 2^k} \left(1-\frac{2m}{2^{k+1}}\right)^{(i-2)\gamma-\sigma}\left(\frac{1}{2^{k+1}}\right)^{3\gamma}\\
		&\lesssim (t-s)^{i\gamma+\gamma-\sigma} \sum_{k\geq 0}\left(\frac{1}{2^{k+1}}\right)^\varepsilon\sum_{0\leq m < 2^k} \left(1-\frac{2m}{2^{k+1}}\right)^{(i+1)\gamma-\sigma-\varepsilon-1}\frac{1}{2^{k+1}}\\
		&\lesssim (t-s)^{i\gamma+\gamma-\sigma} \sum_{k\geq 0}\left(\frac{1}{2^{k+1}}\right)^\varepsilon \frac{1}{2} \int_0^1 (1-x)^{(i+1)\gamma-\sigma-\varepsilon-1}~\txtd x\lesssim (t-s)^{i\gamma+\gamma-\sigma},
	\end{align*}
	which provides the necessary regularity stated in~\eqref{ineq:VarzanehRiedel}.
\end{proof}
\begin{remark}\label{rem:Gbound}
	We highlight why the boundedness assumption of $G$ cannot be relaxed in order to obtain integrable bounds. 
	For example, for $u=\tau_n^{m+1}, v=\tau_{n+1}^{m+1}$ and $w=\tau_{n+2}^{m+1}$ we obtain 
	\begin{align*}
		| &U_{t,v}(U_{v,u}-\textrm{Id})G(v,y_v)\cdot (\delta X)_{v,w}|_{\alpha-i\gamma}\lesssim (t-v)^{-\sigma_1}(v-u)^{\sigma_2} |G(v,y_v)|_{\alpha-i\gamma+\sigma_2-\sigma_1}|(\delta X)_{v,w}|\\
		&\lesssim (t-v)^{-\sigma_1}(v-u)^{\sigma_2} (w-v)^{\change{\sigma}+\varepsilon}W^{\gamma-\sigma-\varepsilon}_{\mathbf{X},\gamma,\sigma+\varepsilon}(v,w)  |G(v,y_v)|_{\alpha-i\gamma+\sigma_2-\sigma_1}\\
		&\lesssim (t-s)^{\sigma_2-\sigma_1+\sigma+\varepsilon}\left(1-\frac{2n}{2^{m+1}}\right)^{-\sigma_1}\left(\frac{1}{2^{m+1}}\right)^{\sigma_2+\sigma+\varepsilon}W^{\gamma-\sigma-\varepsilon}_{\mathbf{X},\gamma,\sigma+\varepsilon}(v,w) |G(v,y_v)|_{\alpha-i\gamma+\sigma_2-\sigma_1},
	\end{align*}
	with suitable choices of $\sigma_1, \sigma_2$. Using that $(y,G(\cdot,y))\in \cD^{\gamma}_{\mathbf{X},\alpha}$  is a solution of~\eqref{Main_Equation} together with a bound of the form $$|G(v,y_v)|_{\alpha-i\gamma+\sigma_2-\sigma_1}\leq \|G(v,y)\|_{\infty,\alpha-\gamma}\leq  \lVert y,G(\cdot,y)\rVert_{\mathcal{D}^\gamma_{\mathbf{X},\alpha}},$$ would lead to the choice $-i\gamma+\sigma_2-\sigma_1=-\gamma$, which entails $\sigma_2-\sigma_1+\sigma+\varepsilon=i\gamma+\sigma+\varepsilon-\gamma$. Since we assume $\sigma+\varepsilon<\gamma$, we see that the time regularity, i.e.~the exponent of $(t-s)$, is less than $i\gamma$. On the other hand, one could try to bound $G(v,y_v)$ by its Hölder norm
	\begin{align*}
		|G(v,y_v)|_{\alpha-i\gamma+\sigma_2-\sigma_1}\leq |G(0,y_0)|_{\alpha-i\gamma+\sigma_2-\sigma_1}+ v[G(\cdot,y)]_{\gamma,\alpha-2\gamma},
	\end{align*}
	but such a bound is only helpful if we further assume $G(0,y_0)=0$. In conclusion, using the control defined in~\eqref{def:control}, we cannot drop the boundedness of $G$. 
    
    \change{\begin{remark}
        This limitation has been removed in \cite{BGV25} by different techniques using another concept of controlled rough paths and control functions. The results in \cite{BGV25} also allow one to treat rough paths of lower regularity, i.e.~$\gamma\in(1/4,1/3)$. 
    \end{remark} } 
\end{remark}
\subsection{\change{An integrable a-priori bound}}\label{sec:integrable_bound}
In order to obtain integrable bounds for the rough path norm of the solution of~\eqref{Main_Equation}, we need to make certain assumptions on the noise. To be more precise, we need a Gaussian process such that the corresponding abstract Wiener and Cameron-Martin space {satisfies} the following property.

\begin{itemize}
	\item[\textbf{(N)}\namedlabel{ass:Noise}{\textbf{(N)}}]
	Let $X$ be a $d$-dimensional continuous and centered Gaussian process defined on an abstract Wiener space with associated Cameron-Martin space $\mathcal{H}$ and let $\gamma^\prime>0$ such that $\gamma+\gamma^\prime -2(\sigma+\varepsilon)>1$ for some arbitrary small $\varepsilon>0$.
	We assume that $X$ has independent \change{and identically distributed} components and the covariance $R_{X}(s,t)\coloneqq \E[X_s\otimes X_t]$ has finite $q$-variation \change{such that} $[R_{\change{X^i}}]_{q-\textnormal{var},[s,t]^2}\lesssim (t-s)^{\frac{1}{q}}$ \change{holds for every $i\in\{1,\ldots, d\}$ and} $q\in [1,2)$ where
	\begin{align*}
		[R_{\change{X^i}}]^q_{q-\textnormal{var},[s,t]^2}\coloneqq \sup_{\pi,\pi^\prime\subset [s,t]} \sum_{\substack{[u,v]\in \pi\\ [u^\prime,v^\prime]\in \pi^\prime}} |\E[X^i_{u,v} X^i_{u^\prime,v^\prime}]|^q,
	\end{align*}
	and that the $\frac{1}{\gamma^\prime}$-variation for every $h\in \mathcal{H}$ is finite, i.e.
	\begin{align*}
		\sup\limits_{\pi\subset [s,t]} \sum_{[u,v]\in \pi} |h_{v}-h_{u}|^{\frac{1}{\gamma^\prime}} < \infty,
	\end{align*}
	where the supremum is taken over all partitions $\pi$ of $[s,t]$. 
	Then it is known that $h$ can be enhanced to a rough path $\mathbf{h}:=\left(h, \int h~\txtd h\right)$. 
	Further, assume
	\begin{align}\label{CM}
		W_{\mathbf{h},\gamma^\prime,\sigma+\varepsilon}(0,1)\lesssim |h|_{\mathcal{H}}^{\frac{1}{\gamma^\prime-\sigma-\varepsilon}}.
	\end{align}
	for all $h\in \mathcal{H}$. 
\end{itemize}
In particular, assumption \ref{ass:Noise} entails that $X$ can be enhanced to a geometric $\gamma$-H\"older rough path $\mathbf{X}=(X,\mathbb{X})$, see \cite[Theorem 10.4 c)]{FH20}.
\begin{theorem}\label{thm:IntegrableBounds}
	Suppose \ref{ass:A} and \ref{ass:Noise} are fulfilled, the nonlinearities $F$ and $G$ satisfy \ref{ass:F} and \ref{ass:G1}-\ref{ass:G2} respectively and the initial condition has moments of all order, i.e. $\E[|u_0|^p_\alpha]<\infty$ for every $p\geq 1$. We further assume that $\sigma\in [0,\frac{1-\gamma}{2})$. Then there exists an integrable bound for the solution $u$ of \eqref{Main_Equation} meaning that
	\begin{align}\label{isigma}
		\|u,G(\cdot,u)\|_{\cD^{\gamma}_{\mathbf{X}(\omega),\alpha}([0,T])}\in \bigcap_{p\geq 1} L^p(\Omega).
	\end{align}
\end{theorem}
\begin{proof}
	
	Based on~\eqref{ineq:VarzanehRiedel} we obtain similar to~\cite{GVR25} and~\cite[Theorem 2.15, Lemma 2.18]{BS24}   
	\begin{align}\label{INTEGAR}
		\|u,G(\cdot,u)\|_{\cD^{\gamma}_{\mathbf{X}(\omega),\alpha}([0,T])}\leq |u_0(\omega)|_\alpha P_1(\omega, [0,T]) +P_2(\omega, [0,T]),
	\end{align}
	for some $P_1(\cdot, [0,T]), P_2(\cdot, [0,T])\in \bigcap_{p\geq 1} L^p(\Omega)$
	which proves the statement. 
\end{proof}

\begin{remark}
	Note that the restriction $\sigma\in[0,\frac{1-\gamma}{2})$ is required only for~\eqref{isigma} and arises from Lemma~\ref{lem:IneqRV}. Since $\gamma\in(\frac{1}{3},\frac{1}{2})$ this leads to a spatial regularity loss $\sigma\in(\frac{1}{4},\frac{1}{3})$. The range $\sigma\in[0,\gamma)$ which is enough for local and global well-posedness of~\eqref{Main_Equation} as established in Theorems~\ref{ex:nona} and~\ref{thm:GlobEx} \change{is treated in \cite{BGV25}}.  
\end{remark}

\subsection{Cameron-Martin space associated to the noise}\label{cm}
The main goal of this subsection is to investigate which stochastic processes satisfy Assumption \ref{ass:Noise}. In \cite[Proposition 2.12]{GVR25} this condition was verified for the rough path lift of the fractional Brownian motion with Hurst parameter $H\in (\frac{1}{3},\frac{1}{2})$. Here we focus on Gaussian Volterra processes~\cite{Cass21}. To this aim, we let $(B_t)_{0\leq t \leq T}$ be a real-valued Brownian motion.
\begin{definition}\label{def:VoltPr}
	A Volterra process is a centered, Gaussian process $(V_t)_{t\in[0,T]}$  which is represented by the It\^o integral
	\begin{align}\label{def:Volterra}
		V_{t}=\int_{0}^{t}K(t,s)~\mathrm{d}B_{s},
	\end{align}
	for a kernel $K:[0,T]\times [0,T]\to \R$.
	
\end{definition}
The covariance function of $V$ is given by
\[ R_V(t,s)=\E[V_t V_s]=\int_{0}^{t\wedge s} K(t,r) K(s,r)~\txtd r. \]
We further make the following assumptions on the kernel.

\begin{assumption}\label{kernelAssumption}
	\begin{itemize}
		\item[i)] $K(0,s)=0$ for all $s\in[0,T]$ and $K(t,s)=0$ for $0<t<s\leq T$.
		\item[ii)] There exists a constant $C>0$ and a parameter $\iota>0$
	\begin{align*}
		\int_0^T (K(t,r)-K(s,r))^2~\txtd r\leq C |t-s|^\iota, \text{ for all } s,t\in [0,T].
	\end{align*}
	\item [iii)] There exists a constant $C>0$ and a parameter $\beta\in[0,\frac{1}{4})$ such that
	\begin{enumerate}
		\item $|K(t,s)|\leq C s^{-\beta}(t-s)^{-\beta}$ for all $0<s<t \leq T$,
		\item $K(\cdot,s)\in C^1$ and 
		\[ \Big| \frac{\partial K(t,s)}{\partial t} \Big| \leq C (t-s)^{-(\beta+1)},~~0<s<t\leq T.\]
	\end{enumerate}
\end{itemize}
\end{assumption}
Furthermore, we can associate to each Volterra kernel a Hilbert-Schmidt operator $\mathcal{K}:L^2([0,T];\R)\to L^2([0,T];\R)$ defined as
\[ (\mathcal{K}f)(t):=\int_0^T K(t,s)f(s)~\txtd s,~~f\in L^2([0,T];\R). \]

\change{\begin{lemma}\label{lem:VoltRP}
	Let $(V_t)_{t\geq 0}$ be a Volterra process with kernel $K$. Then there exists a $\beta$-Hölder continuous modification for every $\beta\in (0,\frac{\iota}{2})$. If additionally $\iota\in\left(\frac{2}{3},1\right]$, then there exists a two-parameter function $\mathbb V$ such that $(V,\mathbb V)$ is a (weakly geometric) $\beta$-Hölder rough path for every $\beta\in \left(\frac{1}{3},\frac{\iota}{2}\right)$.
\end{lemma}
\begin{proof}
	The existence of a Hölder-continuous modification follows directly from Assumption \eqref{kernelAssumption} ii) and Kolmogorov's continuity theorem \cite[1.8.1]{Kun19}.
	
	To prove the existence of a rough path lift, we use \cite[Theorem 10.4 c)]{FH20}. Let $(u,v),(\tilde{u},\tilde{v})\in \Delta_{[0,T]}$, then we have $\min\{u,\tilde{u}\},\min\{u,\tilde{v}\},\min\{\tilde{u},v\}\leq \min\{v,\tilde{v}\}$, which leads to
	\begin{align*}
		\E\big[V_{u,v} V_{\tilde{u},\tilde{v}}\big]&=\int_0^{\min\{v,\tilde{v}\}}K(v,r)K(\tilde{v},r)~\txtd r-\int_0^{\min\{v,\tilde{u}\}}K(v,r)K(\tilde{u},r)~\txtd r\\
		&\quad-\int_0^{\min\{u,\tilde{v}\}}K(u,r)K(\tilde{v},r)~\txtd r+\int_0^{\min\{u,\tilde{u}\}}K(u,r)K(\tilde{u},r)~\txtd r\\
		&=\int_0^{\min\{v,\tilde{v}\}}\big(K(v,r)-K(u,r)\big)\big(K(\tilde{v},r)-K(\tilde{u},r)\big)~\txtd r,
	\end{align*}
	using $K(s,t)=0$ for $s<t$. With this equality, Assumption \eqref{kernelAssumption} ii) as well as the Hölder and Young inequalities, we obtain
	\begin{align*}
		\E\big[V_{u,v} V_{\tilde{u},\tilde{v}}\big]&\lesssim \left(\int_0^{\min\{v,\tilde{v}\}}\big(K(v,r)-K(u,r)\big)^2~\txtd r\right)^{\frac{1}{2}}\left(\int_0^{\min\{v,\tilde{v}\}}\big(K(\tilde{v},r)-K(\tilde{u},r)\big)^2~\txtd r\right)^{\frac{1}{2}}\\
		&\lesssim (v-u)^{\frac{\iota}{2}}(\tilde{v}-\tilde{u})^{\frac{\iota}{2}}\lesssim (v-u)^{\iota}+(\tilde{v}-\tilde{u})^{\iota}.
	\end{align*}
	In particular, this implies that
	\begin{align*}
		\left[R_V\right]^{\frac{1}{\iota}}_{\frac{1}{\iota}\textnormal{-var},[s,t]}\lesssim \sup_{\pi\subset [s,t]} \sum_{[u,v]\in \pi} |v-u|^{\frac{1}{\iota}\iota}= |t-s|.
	\end{align*}
	Due to $\iota\in\left(\frac{2}{3},1\right]$, the assumptions of \cite[Theorem 10.4 c)]{FH20} are fulfilled, which means that $V$ can be enhanced to a weakly geometric rough path.
	\end{proof}}
In particular, we can assume that the Volterra process $(V_t)_{t\geq 0}$ is $\gamma$-Hölder continuous for $\gamma\in(\frac{1}{3},\frac{1}{2})$ choosing $\iota$ accordingly.

\begin{remark}
Standard examples of Volterra processes are the fractional Brownian motion, which satisfies iii) for $\beta=\frac{1}{2}-H$ provided that $H\in(\frac{1}{4},\frac{1}{2})$, and the fractional Ornstein-Uhlenbeck process. Another example is the L\'evy fractional Brownian motion (or Liouville fractional Brownian motion)~\cite{DVolterra,Cass21} with Hurst index $H\in(0,1)$ whose kernel is given by
\[K(t,s)=\frac{1}{\Gamma(H+\frac{1}{2})}(t-s)^{H-\frac{1}{2}}\mathbbm{1}_{[0,t)}(s),  \]
where $\Gamma$ denotes the Gamma function.~This is an example of a Volterra process whose increments are not stationary. 
\end{remark}

We further denote by $\mathcal{H}$ the associated Cameron–Martin space.~For Volterra processes, it is known that the Cameron–Martin space is given by $ \mathcal{H}=\change{\mathcal K}(\change{L}^{2}([0,T];\R))$, see~\cite[Section 3]{DVolterra}, meaning that every $h\in\cH$ has the representation $h(t)=\int_{0}^{t}K(t,s)g(s)~\mathrm{d}s$, where $g\in \change{L}^{2}([0,T];\R)$ and $|h|_{\mathcal{H}}=\Vert g\Vert_{\change{L}^{2}([0,T];\R)}$.~Furthermore, for every $h\in\mathcal{H}$, one can show that $h(t)=\mathbb{E}[Z V_t]$, where $Z$ is an
element of the $\change{L}^{2}$-closure of the span of $(V_{t})_{t\in [0,T]}$ and $\mathcal{H}$ is a Hilbert space with the  inner product given by
\begin{align*}
\langle h^1,h^2\rangle_{\cH}=\mathbb{E}[Z_1Z_2],
\end{align*}
where $h^{1}(t)=\E[Z_1V_t]$ and $h^2(t)=\E[Z_2 V_t]$. \\

In order to prove that $V$ satisfies \eqref{CM}, we further assume that $K$ satisfies 
\begin{itemize}
\item[\textbf{(K1)}\namedlabel{ass:K1}{\textbf{(K1)}}] $\sup\limits_{s\in [0,1-t]}\int_{0}^{1}\vert K(t+s,\tau)- K(s,\tau)\vert~\mathrm{d}\tau=\mathcal{O}(t^{\gamma+\frac{1}{2}})$,
\item[\textbf{(K2)}\namedlabel{ass:K2}{\textbf{(K2)}}] $\sup\limits_{\tau\in [0,1]}\int_{0}^{1-t}\vert K(t+s,\tau)- K(s,\tau)\vert ~\mathrm{d}s=\mathcal{O}(t^{\gamma+\frac{1}{2}})$,
\end{itemize}
for all $t\in [0,1]$.
\begin{lemma}\label{lem:CamMart}
We assume that the kernel $K$ satisfies \ref{ass:K1}-\ref{ass:K2}. Then, for every $\frac{1}{2}<\gamma^{\prime}<\gamma+\frac{1}{2}$, there exists a constant $C(\change{\gamma},\gamma^\prime)>0$
such that 
\begin{align}\label{AA_1}
	\forall h\in \mathcal{H}: \ \  \vert h\vert_{W^{\gamma^{\prime},2}}:=\bigg(\int_{[0,1]^2}\frac{\vert h(u)-h(v)\vert^2}{\vert u-      v\vert^{1+2\gamma^\prime}}~\mathrm{d}u\mathrm{d}v\bigg)^{\frac{1}{2}}\leq C(\change{\gamma},\gamma^\prime)\vert h\vert_{\mathcal{H}}.
\end{align}
In addition, for every $0\leq \tilde{\eta}<\gamma^\prime-\frac{1}{2}$ \change{there exists} a constant $\tilde{C}(\gamma,\change{\gamma^\prime},\tilde{\eta})>0$
\begin{align}\label{CC_1}
	W_{\textbf{h},\gamma^\prime,\tilde{\eta}}(0,1)\leq \tilde{C}(\gamma,\change{\gamma^\prime},\tilde{\eta})\vert h\vert_{\mathcal{H}}^{{\frac{1}{\gamma^\prime-\tilde{\eta}}}}.
\end{align}
\end{lemma}
\begin{proof}
We begin by proving \eqref{AA_1}. A similar statement for the Cameron-Martin space of the fractional Brownian motion can be looked up in \cite[Theorem 3]{FV06b}. Recall, that every $h\in \mathcal{H}=\change{\mathcal K}(\change{L}^2([0,T];\R))$ can be written as $h(t)=\int_{0}^{t}K(t,\tau)g(\tau)~\mathrm{d}\tau$ for some $g\in \change{L}^2([0,T];\R)$. Then we obtain
\begin{align*}
	h(u)-h(v)=
	\begin{cases}
		\int_{0}^{v}(K(u,\tau)-K(v,\tau))g(\tau)~\mathrm{d}\tau+\int_{v}^{u}K(u,\tau)g(\tau)~\mathrm{d}\tau,& 1\geq u\geq v\geq 0 \\
		\int_{0}^{u}(K(u,\tau)-K(v,\tau))g(\tau)~\mathrm{d}\tau+\int_{u}^{v}K(v,\tau)g(\tau)~\mathrm{d}\tau,& 0\leq  u<v\leq 1
	\end{cases},
\end{align*}
which leads to 
\begin{align}\label{ineq:SobNorm}
	\begin{split}
		|h|^2_{W^{\gamma^\prime,2}_0}&\leq 2 \int_0^1 \int_v^1 \frac{\left(\int_0^v (K(u,\tau)-K(v,\tau))g(\tau)~\txtd \tau\right)^2}{|u-v|^{1+2\gamma^\prime}}~\txtd u \txtd v +2 \int_0^1 \int_v^1 \frac{\left(\int_v^u K(u,\tau)g(\tau)~\txtd \tau\right)^2}{|u-v|^{1+2\gamma^\prime}}~\txtd u \txtd v\\
		&+2 \int_0^1 \int_0^v \frac{\left(\int_0^u (K(u,\tau)-K(v,\tau))g(\tau)~\txtd \tau\right)^2}{|u-v|^{1+2\gamma^\prime}}~\txtd u \txtd v+2 \int_0^1 \int_0^v \frac{\left(\int_u^v K(v,\tau)g(\tau)~\txtd \tau\right)^2}{|u-v|^{1+2\gamma^\prime}}~\txtd u \txtd v.
	\end{split}
\end{align}
Due to the Cauchy-Schwarz inequality and \ref{ass:K1} we further obtain 	for $v \leq u\leq 1$
\begin{align}\label{ineq:KernelEst}
	\begin{split}
		\Big(\int_0^v &(K(u,\tau)-K(v,\tau))g(\tau)~\txtd \tau\Big)^2\leq \int_0^v |K(u,\tau)-K(v,\tau)|~\txtd \tau\int_0^v |K(u,\tau)-K(v,\tau)|g^2(\tau)~\txtd \tau \\
		&\leq \sup\limits_{s\in [0,1-\change{(u-v)}]} \int_0^1 |K(u-v+s,\tau)-K(s,\tau)|~\txtd \tau\int_0^v |K(u,\tau)-K(v,\tau)|g^2(\tau)~\txtd \tau\\
		&=\mathcal{O}((u-v)^{\gamma+\frac{1}{2}})\int_0^v |K(u,\tau)-K(v,\tau)|g^2(\tau)~\txtd \tau,
	\end{split}
\end{align}
and similarly for $u<v\leq 1$ $$\left(\int_0^u (K(u,\tau)-K(v,\tau))g(\tau)~\txtd \tau\right)^2\leq \mathcal{O}((u-v)^{\gamma+\frac{1}{2}})\int_0^u |K(u,\tau)-K(v,\tau)|g^2(\tau)~\txtd \tau.$$
Using \ref{ass:K2}, \eqref{ineq:KernelEst} and Tonelli's theorem we can estimate the first term in \eqref{ineq:SobNorm} 
\begin{align*}
	\int_{0}^{1}\int_{v}^{1}&\frac{\big(\int_{0}^{v}(K(u,\tau)-K(v,\tau))g(\tau)~\mathrm{d}\tau\big)^2}{|u-v|^{1+2\gamma^\prime}}~\mathrm{d}u\mathrm{d}v\lesssim \int_{0}^{1}\int_{v}^{1}\frac{\int_{0}^{v}\vert K(u,\tau)-K(v,\tau)\vert  g^2(\tau)~\mathrm{d}\tau}{|u-v|^{2\gamma^\prime-\gamma+\frac{1}{2}}}~\mathrm{d}u\mathrm{d}v\\
	&= \int_{0}^{1}\int_{0}^{1-v}\frac{\int_{0}^{v}\vert K(v+x,\tau)-K(v,\tau)\vert g^2(\tau)~\mathrm{d}\tau}{|x|^{2\gamma^\prime-\gamma+\frac{1}{2}}}~\mathrm{d}x\mathrm{d}v\\
	&=\int_{0}^{1} g^2(\tau)\int_{0}^{1-\tau}\frac{\int_{\tau}^{1-x}\vert K(v+x,\tau)-K(v,\tau)\vert~\mathrm{d}v}{|x|^{2\gamma^\prime-\gamma+\frac{1}{2}}}~\mathrm{d}x\mathrm{d}\tau\\
	&\leq \int_{0}^{1} g^2(\tau)\int_{0}^{1-\tau}\frac{\int_{0}^{1-x}\vert K(v+x,\tau)-K(v,\tau)\vert~\mathrm{d}v}{|x|^{2\gamma^\prime-\gamma+\frac{1}{2}}}~\mathrm{d}x\mathrm{d}\tau\leq |h|^{2}_{\mathcal{H}}\int_{0}^{1}|x|^{2\gamma-2\gamma^{\prime}}~\mathrm{d}x\lesssim |h|^{2}_{\mathcal{H}}.
\end{align*}
A similar computation can be used to estimate the third term in \eqref{ineq:SobNorm}, since with Tonelli's theorem and \eqref{ineq:KernelEst} we get 
\begin{align*}
	\int_{0}^{1}\int_{0}^{v}&\frac{\big(\int_{0}^{u}(K(u,\tau)-K(v,\tau))g(\tau)\mathrm{d}\tau\big)^2}{|u-v|^{1+2\gamma^\prime}}~\mathrm{d}u\mathrm{d}v\lesssim \int_{0}^{1}\int_{0}^{v}\frac{\int_{0}^{u}|K(u,\tau)-K(v,\tau)|g^2(\tau)\mathrm{d}\tau}{|u-v|^{\frac{1}{2}+2\gamma^\prime-\gamma}}~\mathrm{d}u\mathrm{d}v\\
	&=\int_{0}^{1}\int_{u}^{1}\frac{\int_{0}^{u}|K(u,\tau)-K(v,\tau)|g^2(\tau)\mathrm{d}\tau}{|u-v|^{\frac{1}{2}+2\gamma^\prime-\gamma}}~\mathrm{d}v\mathrm{d}u\lesssim |h|^{2}_{\mathcal{H}}.
\end{align*}
To estimate the second and fourth term in \eqref{ineq:SobNorm}, we use the fact that $K(s,t)=0$ for $s\geq t$. Then, similar as in \eqref{ineq:KernelEst}, we obtain
\begin{align*}
	\begin{split}
		\Big(\int_v^u &(K(u,\tau)-\underbrace{K(v,\tau)}_{=0})g(\tau)~\txtd \tau\Big)^2\leq \int_v^u |K(u,\tau)-K(v,\tau)|~\txtd \tau\int_v^u |K(u,\tau)-K(v,\tau)|g^2(\tau)~\txtd \tau \\
		&=\mathcal{O}((u-v)^{\gamma+\frac{1}{2}})\int_0^u |K(u,\tau)-K(v,\tau)|g^2(\tau)~\txtd \tau,
	\end{split}
\end{align*}
for $u>v$. This leads to 
\begin{align*}
	\int_{0}^{1}\int_{v}^{1}&\frac{\big(\int_{v}^{u}K(u,\tau)g(\tau)\mathrm{d}\tau\big)^2}{|u-v|^{1+2\gamma^\prime}}~\mathrm{d}u\mathrm{d}v=\int_{0}^{1}\int_{v}^{1}\frac{\big(\int_{v}^{u}(K(u,\tau)-K(v,\tau))g(\tau)~\mathrm{d}\tau\big)^2}{|u-v|^{1+2\gamma^\prime}}~\mathrm{d}u\mathrm{d}v\\
	&\lesssim \int_{0}^{1}\int_{v}^{1}\frac{\int_{0}^{u}|K(u,\tau)-K(v,\tau)|g^2(\tau)~\mathrm{d}\tau}{|u-v|^{\frac{1}{2}+2\gamma^\prime-\gamma}}~\mathrm{d}u\mathrm{d}v \\
	&=\int_{0}^{1}\int_{v}^{1}\frac{\int_{0}^{v+x}|K(v+x,\tau)-K(v,\tau)|g^2(\tau)~\mathrm{d}\tau}{|x|^{\frac{1}{2}+2\gamma^\prime-\gamma}}~\mathrm{d}x\mathrm{d}v \\
	&=\int_{0}^{1}\int_{0}^{1}\frac{\int_{\max\{\tau-x,0\}}^{1-x}|K(v+x,\tau)-K(v,\tau)|g^2(\tau)~\mathrm{d}v}{|x|^{\frac{1}{2}+2\gamma^\prime-\gamma}}~\mathrm{d}x\mathrm{d}\tau \\
	&\leq \int_{0}^{1}\int_{0}^{1}\frac{\int_{0}^{1-x}|K(v+x,\tau)-K(v,\tau)|g^2(\tau)~\mathrm{d}v}{|x|^{\frac{1}{2}+2\gamma^\prime-\gamma}}~\mathrm{d}x\mathrm{d}\tau \lesssim |h|^{2}_{\cH}
\end{align*}
and again with a similar computation $ \int_{0}^{1}\int_{0}^{v}\frac{\big(\int_{u}^{v}K(v,\tau)g(\tau)\mathrm{d}\tau\big)^2}{|u-v|^{1+2\gamma^\prime}}~\mathrm{d}u\mathrm{d}v\lesssim |h|^{2}_{\cH}$. This shows \eqref{AA_1}.

In order to show \eqref{CC_1}, note that \eqref{AA_1} together with the Besov–variation embedding~ \cite[Corollary A.3]{FV10}, yields that the $\frac{1}{\gamma^\prime}$-variation of every $h\in \cH$ is finite. Since $\gamma^\prime>\frac{1}{2}$, the Young integral
\begin{align*}
	\Delta_{[0,1]}\to \R, (s,t)\mapsto\mathbbm{h}_{s,t}\coloneqq \int_s^t h(r)-h(s)~\txtd h(r)
\end{align*}
is well-defined. Using the Besov-Hölder embedding \cite[Corollary A.2]{FV10} we obtain
\begin{align*}
	\vert h(t)-h(s)\vert^{2}\lesssim  \vert t-s\vert^{2\gamma^\prime-1}\int_{[s,t]^2}\frac{\vert h(u)-h(v)\vert^2}{\vert u-v\vert^{1+2\gamma^\prime}}~\mathrm{d}u\mathrm{d}v.
\end{align*}
for any $s,t\in [0,1]$. This yields 
\begin{align}\label{BB_1}
	\frac{\vert h(t)-h(s)\vert^{\frac{1}{\gamma^\prime-\tilde{\eta}}}}{\vert t-s\vert^{\frac{\tilde{\eta}}{\gamma^\prime-\tilde{\eta}}}}\lesssim \vert t-s\vert^{\frac{\gamma^\prime-\frac{1}{2}-\tilde{\eta}}{\gamma^\prime-\tilde{\eta}}}\Big(\int_{[s,t]^2}\frac{\vert h(u)-h(v)\vert^2}{\vert u-v\vert^{1+2\gamma^\prime}}~\mathrm{d}u\mathrm{d}v\Big)^{\frac{1}{2(\gamma^\prime-\tilde{\eta})}}\coloneqq w(s,t).
\end{align}
Note that the right-hand side is a control function. Indeed, $(s,t)\mapsto t-s$ and the integral are obviously controls and then the product is also a control function due to $\frac{\gamma^\prime-\frac{1}{2}-\tilde{\eta}}{\gamma^\prime-\tilde{\eta}}+\frac{1}{2(\gamma^\prime-\tilde{\eta})}= 1$ and $\gamma^\prime> \frac{1}{2}+\tilde{\eta}$, see \cite[Exercise 1.10]{FV10}. In particular, $w$ is subadditive, which leads to
\begin{align*}
	W_{\textbf{h},\tilde{\eta},\gamma^\prime}(0,1) &= \sup_{\pi\subset [0,1]} \left\{ \sum_{[u,v]\in\pi}(v-u)^{\frac{-\eta}{\gamma^\prime-\eta}}\big{[}|h(v)-h(u)|^{\frac{1}{\gamma^\prime-\eta}}+|\mathbbm{h}_{u,v}|^{\frac{1}{2(\gamma^\prime-\eta)}} \big{]} \right\}\\
	&\lesssim \sup_{\pi\subset [0,1]} \left\{ \sum_{[u,v]\in\pi}(v-u)^{\frac{-\eta}{\gamma^\prime-\eta}}|h(v)-h(u)|^{\frac{1}{\gamma^\prime-\eta}} \right\}\lesssim \sup_{\pi\subset [0,1]} \left\{\ \sum_{[u,v]\in\pi}w(u,v)\right\}\\
	&\leq w(0,1)=\Big(\int_{[0,1]^2}\frac{\vert h(u)-h(v)\vert^2}{\vert u-v\vert^{1+2\gamma^\prime}}~\mathrm{d}u\mathrm{d}v\Big)^{\frac{1}{2(\gamma^\prime-\tilde{\eta})}}\lesssim \vert h\vert_{\mathcal{H}}^{{\frac{1}{\gamma^\prime-\tilde{\eta}}}},
\end{align*}
where we used \eqref{AA_1} and \eqref{BB_1}. 
\end{proof}
\change{\begin{remark}
	Note that Assumption \ref{ass:K1} can be replaced by 
	\begin{align*}
		\int_{0}^{1}\vert K(t,\tau)- K(s,\tau)\vert~\mathrm{d}\tau=\mathcal{O}(|t-s|^{\gamma+\frac{1}{2}}).
	\end{align*}
	However, this is more difficult to verify in applications, which is why we impose~\ref{ass:K1}. 
	\end{remark}}
	In particular,  choosing $\tilde{\eta}=\sigma+\varepsilon$ and $\tilde{\eta}+\frac{1}{2}<\gamma^\prime<\gamma+\frac{1}{2}$, it can easily be seen that $\gamma+\gamma^\prime>1+\tilde{\eta}$ holds and therefore the condition on the Cameron-Martin space in \ref{ass:Noise} is fulfilled. Now we want to state some examples of Volterra processes which satisfies the assumptions of Lemma \ref{lem:CamMart}.
	\begin{example}
\begin{itemize}
	\item[i)] {\em (Fractional Brownian motion).} The fractional Brownian motion can be represented as a Volterra process using the kernel
	\begin{align*}
		K(t,s)\coloneqq \frac{(t-s)^{H-\frac{1}{2}}}{\Gamma(H+\frac{1}{2})} f_h\left(\frac{1}{2}-H,H-\frac{1}{2},h+\frac{1}{2},1-\frac{t}{s}\right)\mathbbm{1}_{[0,t)}(s),
	\end{align*}
	where $\Gamma$ is the Gamma- and $f_h$ the hypergeometric function. This kernel satisfies the Assumption \ref{kernelAssumption} i), ii) and iii) for $\beta=\frac{1}{2}-H$ provided that $H\in(\frac{1}{4},\frac{1}{2})$, which in particular covers our range $\gamma\in(\frac{1}{3},\frac{1}{2})$. Moreover, it can be shown that this kernel satisfies \ref{ass:K1}-\ref{ass:K2}, see \cite[Appendix A]{FV06b}. 
	\item[ii)] {\em (Ornstein-Uhlenbeck process).} The Ornstein-Uhlenbeck process has the kernel
	\begin{align*}
		K(t,s)\coloneqq e^{a(t-s)}\mathbbm{1}_{[0,t)}(s),
	\end{align*}
	for some $a<0$. It can be shown that this kernel satisfies Assumption \ref{kernelAssumption} i), ii), and iii) with $\beta=0$, as well as \ref{ass:K1}-\ref{ass:K2} since $a<0$. 
	\item[iii)] {\em (Liouville fractional Brownian motion).} We recall that the kernel for the Liouville fractional Brownian motion is given by 
	\begin{align*}
		K(t,s)=\frac{1}{\Gamma(H+\frac{1}{2})}(t-s)^{H-\frac{1}{2}}\mathbbm{1}_{[0,t)}(s).
	\end{align*}
	for $H\in (0,1)$. One can prove that this kernel satisfies Assumption \ref{kernelAssumption} i), ii) and iii) for $\iota=H$ provided that $H\in(\frac{1}{4},\frac{1}{2})$.  Furthermore,~\ref{ass:K1}-\ref{ass:K2} can easily be verified.  
\end{itemize}
\end{example}


\begin{remark}
Note that Gaussianity and condition~\eqref{CM} are essential for our arguments, and we therefore work with Gaussian Volterra processes, in contrast to Volterra rough paths which are given by $V_t=\int_0^t K(t,s)~\txtd X_s$ for a rough input $X$, as considered by~\cite{HT21}.
\end{remark}

\section{Rough Gronwall's inequality}\label{sec:gronwall}
\subsection{The mild Gronwall Lemma}\label{sec:gronwall1}
In this section, we establish a mild Gronwall lemma for the solution of~\eqref{Main_Equation} on an arbitrary interval $[s,t]$. Therefore, we consider for $t>0$ the path component of the mild solution of~\eqref{Main_Equation} given by
\begin{align*}
\change{u}_t:=U_{t,s}\change{u}_s+\int_s^t U_{t,r}F(r,\change{u}_r)~\txtd r+\int_s^t U_{t,r}G(r,\change{u}_r)~\txtd\mathbf{X}_r,
\end{align*}
with initial condition $\change{u}_s\in E_\alpha$. Under suitable assumptions, recall~\ref{ass:G1}, the Gubinelli derivative is given by $\change{u}_t^\prime=G(t,\change{u}_t)$. The goal is to obtain a bound for $(\change{u},\change{u}')=(\change{u},G(\cdot,\change{u}))\in \cD^{\gamma}_{\mathbf{X},\alpha}$ of the form
\begin{align*}
\|\change{u},G(\cdot,\change{u})\|_{\cD^\gamma_{\mathbf{X},\alpha}([s,t])}\lesssim (|\change{u}_s|_\alpha+|\change{G(s,u_s)}|_{\alpha\change{-\gamma}}) e^{C (t-s)}, 
\end{align*}
for suitable constants, similar to the classical Gronwall inequality. Furthermore, this inequality will be applied to the linearization of the equation during the course of this section. 
We note that there is also a different notion of a rough Gronwall introduced in \cite{DGHT19,Hof18}, which uses energy estimates in the framework of unbounded rough drivers instead of the mild formulation. 

Before stating the Gronwall inequality, we first specify a straightforward auxiliary result that is required in the proof. 
\begin{lemma}
Let $(\change{y},\change{y}^\prime)\in \mathcal{D}^\gamma_{\mathbf{X},\alpha}([s,t])$. Then we have
\begin{align}\label{devide_interval}
	\Vert \change{y},\change{y}^\prime\Vert_{\mathcal{D}_{\mathbf{X},{\alpha}}^{\gamma}([s,t])}\leq \rho_{\gamma,[s,t]}(\mathbf{X})\Vert \change{y},\change{y}^\prime\Vert_{\mathcal{D}_{\mathbf{X},{\alpha}}^{\gamma}([s,\change{r}])}+\Vert \change{y},\change{y}^\prime\Vert_{\mathcal{D}_{\mathbf{X},{\alpha}}^{\gamma}([\change{r},t])},
\end{align}
for every $s\leq \change{r}\leq t$.
\end{lemma}
\begin{lemma}\label{lem:RoughGronwall}{\em(Mild rough Gronwall inequality).}
Suppose $A,F$ and $G$ satisfy the Assumptions \ref{ass:A},\ref{ass:F} and \ref{ass:G1}-\ref{ass:G2}. Then \change{the solution of \eqref{Main_Equation}} satisfies $(\change{u},G(\cdot,\change{u}_\cdot))\in \mathcal{D}^\gamma_{\mathbf{X},\alpha}([s,t])$ and we obtain the estimate
\begin{align}\label{est:NonlinRoughGron}
	\|\change{u},G(\cdot,\change{u})\|_{\cD^\gamma_{\mathbf{X},\alpha}([s,t])}\leq C_1\rho_{\gamma,[s,t]}(\mathbf{X})\left(1+|\change{u}_s|_\alpha+|G(s,\change{u}_s)|_{\alpha-\gamma}\right)e^{C_2(t-s)},
\end{align}
where the constants are given by 
\begin{align*}
	C_1&:=e^{C_2}\max\left\{\frac{1-C\kappa^\nu \Phi_3}{2C\Phi_2-1+C\kappa^\nu\Phi_3},\frac{(1-C\kappa^\nu \Phi_3)C\Phi_1}{(C\kappa^\nu\Phi_3+2C\Phi_2-1)^2}\right\},\quad C_2:=\frac{1}{\kappa}\ln{\left(\frac{2C\Phi_2}{1-C\kappa^\nu \Phi_3}\right)},
\end{align*}
with \change{$C:=C(U,\alpha,\sigma,\delta, \gamma)>1$}, $\nu:=\min\{1-2\gamma,1-\delta,\gamma-\sigma\}$, \change{$\kappa>0$} such that \change{$C\kappa^\nu\Phi_3<1$} and
\begin{align*}
	\Phi_1&:=C_F+C_G\rho_{\gamma,[s,t]}(\mathbf{X})^2+C_G\rho_{\gamma,[s,t]}(\mathbf{X}),\quad \Phi_2:=\max\{1,C_G\rho_{\gamma,[s,t]}(\mathbf{X})\}\\
	\Phi_3&:=C_F+C_G\rho_{\gamma,[s,t]}(\mathbf{X})^2.
\end{align*}
\end{lemma}
\begin{proof}
Due to Theorem \ref{thm:GlobEx} we have $(\change{u},G(\cdot,\change{u}))\in \cD^{\gamma}_{\mathbf{X},\alpha}$. Then the following estimates can easily be obtained for  $s\leq \change{v}\leq \change{w}\leq t$ with $\change{w}-\change{v}<1$: 
\begin{align*}
	\|U_{\cdot,\change{v}}\change{u}_\change{v},0\|_{\cD^{\gamma}_{\mathbf{X},\alpha}([\change{v},\change{w}])}&\lesssim |\change{u}_\change{v}|_\alpha,\\
	\left\|\int_{\change{v}}^\cdot U_{\cdot,r} F(r,\change{u}_r)~\txtd r,0\right\|_{\cD^{\gamma}_{\mathbf{X},\alpha}([\change{v},\change{w}])}&\leq C_F (\change{w}-\change{v})^{\min\{1-\delta,1-2\gamma\}}(1+\|\change{u},G(\cdot,\change{u})\|_{\cD^\gamma_{\mathbf{X},\alpha}([\change{v},\change{w}])}),\\
	\|G(\cdot,\change{u}), (G(\cdot,\change{u}))^\prime\|_{\cD^\gamma_{\mathbf{X},\alpha-\sigma}([\change{v},\change{w}])}&\leq C_G\rho_{\gamma,[s,t]}(\mathbf{X})(1+\|\change{u},G(\cdot,\change{u})\|_{\cD^\gamma_{\mathbf{X},\alpha}([\change{v},\change{w}])}).
\end{align*}
Combining these estimates with \eqref{est:RPIntegral} we obtain
\begin{align}\label{ineq:MainIneqGron}
	\begin{split}
		\|\change{u}&,G(\cdot,\change{u})\|_{\cD^\gamma_{\mathbf{X},\alpha}([\change{v},\change{w}])}\lesssim |\change{u}_\change{v}|_\alpha + C_F (\change{w}-\change{v})^{\min\{1-\delta,1-2\gamma\}}(1+\|\change{u},G(\cdot,\change{u})\|_{\cD^\gamma_{\mathbf{X},\alpha}([\change{v},\change{w}])})\\
		&+\rho_{\gamma,[s,t]}(\mathbf{X}) (|G(\change{v},\change{u}_\change{v})|_{\alpha-\sigma}+|(G(\change{v},\change{u}_\change{v}))^\prime|_{\alpha-\sigma-\gamma}+(\change{w}-\change{v})^{\gamma-\sigma}\|G(\cdot,\change{u}), (G(\cdot,\change{u}))^\prime\|_{\cD^\gamma_{\mathbf{X},\alpha}([\change{v},\change{w}])})\\
		&\lesssim C_F+C_G\rho_{\gamma,[s,t]}(\mathbf{X})^2+C_G\rho_{\gamma,[s,t]}(\mathbf{X})+|\change{u}_\change{v}|_\alpha+C_G\rho_{\gamma,[s,t]}(\mathbf{X})|G(\change{v},\change{u}_\change{v})|_{\alpha-\gamma}\\
		&+(C_F+C_G\rho_{\gamma,[s,t]}(\mathbf{X})^2)(\change{w}-\change{v})^{\nu}\|\change{u},G(\cdot,\change{u})\|_{\cD^\gamma_{\mathbf{X},\alpha}([\change{v},\change{w}])}\\
		&=:\Phi_1+\Phi_2(|\change{u}_\change{v}|_\alpha+|G(\change{v},\change{u}_\change{v})|_{\alpha-\gamma})+\Phi_3(\change{w}-\change{v})^\nu\|\change{u},G(\cdot,\change{u})\|_{\cD^\gamma_{\mathbf{X},\alpha}([\change{v},\change{w}])}.
	\end{split}
\end{align}
We now choose a sequence of intervals $I_n:=[\kappa_n,\kappa_{n+1}]$ with $\kappa_n:=\min\{s+n\kappa,t\}$ and $N(\kappa):=\inf\{n\in \N~\colon~\kappa_n=t\}$ where \change{$\kappa>0$} is fixed, such that
\change{
\begin{align*}
	C\kappa^\nu\Phi_3<1.
\end{align*}
}
So we obtain for $n<N(\kappa)$
\begin{align*}
	\|\change{u},G(\cdot,\change{u})\|_{\cD^\gamma_{\mathbf{X},\alpha}(I_{n})}\leq C\Phi_1+2C\Phi_2\|\change{u},G(\cdot,\change{u})\|_{\cD^\gamma_{\mathbf{X},\alpha}(I_{n-1})}+C\kappa^\nu\Phi_3\|\change{u},G(\cdot,\change{u})\|_{\cD^\gamma_{\mathbf{X},\alpha}(I_{n})},
\end{align*}
which leads to
\begin{align*}
	\|\change{u},G(\cdot,\change{u}_)\|_{\cD^\gamma_{\mathbf{X},\alpha}(I_{n})}<\frac{C\Phi_1}{1-C\kappa^\nu \Phi_3}+\frac{2C\Phi_2}{1-C\kappa^\nu \Phi_3}\|\change{u},G(\cdot,\change{u})\|_{\cD^\gamma_{\mathbf{X},\alpha}(I_{n-1})}.    
\end{align*}
Iterating these estimates leads to
\begin{align*}
	&\|\change{u},G(\cdot,\change{u})\|_{\cD^\gamma_{\mathbf{X},\alpha}(I_{n})}\leq \left(\frac{2C\Phi_2}{1-C\kappa^\nu \Phi_3}\right)^{n+1}(|\change{u}_s|_\alpha+|G(s,\change{u}_s)|_{\alpha-\gamma})+\frac{C\Phi_1}{1-C\kappa^\nu \Phi_3}\sum_{j=0}^n \left(\frac{2C\Phi_2}{1-C\kappa^\nu \Phi_3}\right)^j\\
	&= \left(\frac{2C\Phi_2}{1-C\kappa^\nu \Phi_3}\right)^{n+1}(|\change{u}_s|_\alpha+|G(s,\change{u}_s)|_{\alpha-\gamma})+\frac{C\Phi_1}{1-C\kappa^\nu \Phi_3}\frac{1-\left(\frac{2C\Phi_2}{1-C\kappa^\nu \Phi_3}\right)^{n+1}}{1-\frac{2C\Phi_2}{1-C\kappa^\nu \Phi_3}}\\
	&=\left(\frac{2C\Phi_2}{1-C\kappa^\nu \Phi_3}\right)^{n+1}(|\change{u}_s|_\alpha+|G(s,\change{u}_s)|_{\alpha-\gamma})+\frac{C\Phi_1}{C\kappa^\nu\Phi_3+2C\Phi_2-1}\left(\left(\frac{2C\Phi_2}{1-C\kappa^\nu \Phi_3}\right)^{n+1}-1\right)\\
	&\leq \left(\frac{2C\Phi_2}{1-C\kappa^\nu \Phi_3}\right)^{n+1}\left(|\change{u}_s|_\alpha+|G(s,\change{u}_s)|_{\alpha-\gamma}+\frac{C\Phi_1}{C\kappa^\nu\Phi_3+2C\Phi_2-1}\right),
\end{align*}
where we used $2C\Phi_2+Cr^\nu\Phi_3-1>0$. Using now \eqref{devide_interval} we derive
\begin{align*}
	&\|\change{u},G(\cdot,\change{u})\|_{\cD^\gamma_{\mathbf{X},\alpha}([s,t])}\leq \rho_{\gamma,[s,t]}(\mathbf{X}) \sum_{n=0}^{N(\kappa)-1} \|\change{u},G(\cdot,\change{u})\|_{\cD^\gamma_{\mathbf{X},\alpha}(I_n)}\\
	&\leq \rho_{\gamma,[s,t]}(\mathbf{X})\left(|\change{u}_s|_\alpha+|G(s,\change{u}_s)|_{\alpha-\gamma}+\frac{C\Phi_1}{C\kappa^\nu\Phi_3+2C\Phi_2-1}\right)\frac{\left(\frac{2C\Phi_2}{1-C\kappa^\nu \Phi_3}\right)^{N(\kappa)+1}-\frac{2C\Phi_2}{1-C\kappa^\nu \Phi_3}}{\frac{2C\Phi_2}{1-C\kappa^\nu \Phi_3}-1}\\
	&\leq \rho_{\gamma,[s,t]}(\mathbf{X})\frac{1-C\kappa^\nu \Phi_3}{2C\Phi_2-1+C\kappa^\nu\Phi_3}\left(|\change{u}_s|_\alpha+|G(s,\change{u}_s)|_{\alpha-\gamma}+\frac{C\Phi_1}{C\kappa^\nu\Phi_3+2C\Phi_2-1}\right)e^{(N(\kappa)+1)\ln{\left(\frac{2C\Phi_2}{1-C\kappa^\nu \Phi_3}\right)}}.
\end{align*}
Finally, the bound $N(\kappa)\leq (t-s) \kappa^{-1}$ entails \eqref{est:NonlinRoughGron}.
\end{proof}
\begin{remark}
\begin{itemize}
	\item[i)] The Gronwall inequality stated in Lemma \ref{lem:RoughGronwall} is also valid for autonomous equations, with obvious modifications. 
	\item[ii)] While the mild Gronwall lemma is of interest in its own, we require a more general result for our purposes.~In order to apply the multiplicative ergodic theorem in Section \ref{sec:InvSets}, we have to linearize \eqref{Main_Equation} around a stationary solution and derive integrable bounds for this linearization. This is the topic of the next section. 
\end{itemize}
\end{remark}
\subsection{Linearization of the rough PDE}\label{sec:LinDynamics}
Since we aim to investigate Lyapunov exponents for rough PDEs using the multiplicative ergodic theorem stated in Section \ref{sec:InvSets}, we first analyze the linearization of~\eqref{Main_Equation} along an arbitrary trajectory. The main goal is to show that the solution of the linearization has finite moments using the rough Gronwall inequality, see Proposition \ref{MET_INTEG}.  \\

\change{The required version of Gronwall's inequality is stated the for non-autonomous nonlinearities $F$ and $G$. However, throughout the rest of the subsection we deal with autonomous nonlinearities $F$ and $G$, for notational simplicity.~Their time dependence would only lead to a more complicated representation of the remainders in Lemma \ref{lem:LinCRP} and Lemma \ref{lem:LinCRPDiff}. The resulting estimates remain the same as in the non-autonomous situation using the same adjustments as in Section \ref{integrable} and Subsection \ref{sec:gronwall1}.  }  For this reason, we consider here
\begin{align}\label{rpde1}
\begin{cases}
	\txtd u_t = [A(t) u_t + F(u_t)]~\txtd t + G(u_t)~\txtd \mathbf{X}_t,\\
	u_0\in E_\alpha.
\end{cases}
\end{align}
The linearization  $\txtD u^{u_0}_{t}$ of~\eqref{rpde1} along an arbitrary solution $u^{u_0}_{t}$, with initial value $u_0$, is defined as the solution $v^{u_0,v_0}_t$ of the following equation given by
\begin{align}\label{linearization:rpde}
\begin{cases}
	\txtd v_t = [A(t) v_t + \txtD F(u^{u_0}_t) v_t]~\txtd t + \txtD G(u^{u_0}_t) v_t~\txtd\mathbf{X}_t\\
	v_0\in E_\alpha,
\end{cases}
\end{align}
also called the first variation equation. Here, $\txtD F$ and $\txtD G$ denote the Fr\'echet derivatives of the nonlinear terms $F$ and $G$. 
Suppressing the dependency of $u$ on the initial condition $u_0$, the Gubinelli derivative of $H(u,v):=\txtD G(u) v$ is given by 
\[ (\txtD G (u_t)v_t)' =\txtD^2 G(u_t)u'_tv_t +\txtD G(u_t)v'_t \]
using the chain rule and the product rule for two controlled rough paths $(u,u'), (v,v')\in \mathcal{D}^{\gamma}_{\mathbf{X},\alpha}$. 
We first show that  $(H(u,v),(H(u,v))^\prime)\in \cD^{\gamma}_{\mathbf{X},\alpha-\sigma}$ together with an a-priori estimate. Based on this, we obtain a bound for the solution of the linearization~\eqref{linearization:rpde} using the mild rough Gronwall lemma.

\begin{lemma}\label{lem:LinCRP}
Let $(u,u'), (v,v')\in \mathcal{D}^{\gamma}_{\mathbf{X},\alpha}$ be the solution to \eqref{rpde1} with initial value $u_0\in E_\alpha$ and the linearization along the solution given by \eqref{linearization:rpde}.
We have $(H(u,v),(H(u,v))')\in\cD^{\gamma}_{\mathbf{X},\alpha-\sigma}$ and  
\begin{align}\label{crp:linearization}
	\|H(u,v), (H(u,v))'\|_{\cD^{\gamma}_{\mathbf{X},\alpha{-\sigma}}} \lesssim C_G \rho_{\gamma,[s,t]} (\mathbf{X}) )^2(1+\|u,u'\|_{\cD^{\gamma}_{\mathbf{X},\alpha}})^2 \|v,v'\|_{\cD^{\gamma}_{\mathbf{X},\alpha}}.
\end{align}
\end{lemma}
\begin{proof}
We obviously have that 
\[ \|\txtD G(u) v\|_{\infty,\alpha-\sigma}\leq C_G\|v\|_{\infty,\alpha} \]
as well as
\begin{align*} 
	\|(\txtD G(u) v)'\|_{\infty,\alpha-\sigma-\gamma}& \lesssim C_G ( \|u'\|_{\infty,\alpha-\gamma} \|v\|_{\infty,\alpha} +\|v'\|_{\infty,\alpha-\gamma} )\\
	& \leq C_G (1+\|u,u'\|_{\cD^\gamma_{\mathbf{X},\alpha}}) \|v,v'\|_{\cD^\gamma_{\mathbf{X},\alpha}}.
\end{align*}
The $\gamma$-H\"older regularity of $(H(u,v))'$ in $E_{\alpha-2\gamma-\sigma}$ is straightforward using that
\begin{align*}
	&  \txtD^2 G(u_t) u'_t v_t - \txtD^2 G(u_s) u'_s v_s + \txtD G(u_t) v'_t - \txtD G(u_s) v'_s\\
	& = (\txtD^2 G(u_t) -\txtD^2 G(u_s) ) u'_t v_t +\txtD^2  G(u_s) (u'_t v_t  -u'_s v_s) \\
	&+ (\txtD G(u_t) -\txtD G(u_s) )v'_t + \txtD G(u_s) (v'_t -v'_s).
\end{align*}
For the first term we have, using Remark~\ref{rem:Hcontpath}
\begin{align*}
	|(\txtD^2 G(u_t) - \txtD^2 G(u_s)) u'_t v_t|_{{\alpha-\sigma-2\gamma}}&\lesssim C_G (t-s)^\gamma [u]_{\gamma,\alpha-2\gamma} |u'|_{\infty,\alpha-\gamma} |v|_{\infty,\alpha}\\
	& \leq C_G (t-s)^\gamma \rho_{\gamma,[s,t]}(\mathbf{X}) \|u,u'\|^2_{\cD^{\gamma}_{\mathbf{X},\alpha}} \|v,v'\|_{\cD^{\gamma}_{\mathbf{X},\alpha}}. 
\end{align*}
The second term can be controlled using 
\begin{align*} 
	|(u'_t - u'_s)v_s|_{{\alpha-2\gamma}} &\lesssim (t-s)^\gamma [u']_{\gamma, \alpha-2\gamma} \|v\|_{\infty,\alpha} \leq (t-s)^\gamma \|u,u'\|_{\cD^{\gamma}_{\mathbf{X},\alpha}} \|v,v'\|_{\cD^{\gamma}_{\mathbf{X},\alpha}}  \\
	|u'_t(v_t -v_s)|_{{\alpha-2\gamma}} &\lesssim (t-s)^\gamma\|u'\|_{\infty,\alpha-\gamma}[v]_{\gamma,\alpha-2\gamma}\leq (t-s)^\gamma \|u,u'\|_{\cD^{\gamma}_{\mathbf{X},\alpha}} \|v,v'\|_{\cD^{\gamma}_{\mathbf{X},\alpha}} .  \end{align*}
The third term results in
\begin{align*} |(\txtD G (u_t) -\txtD G(u_s) )v'_t|_{{\alpha-\sigma-2\gamma}} & \lesssim C_G (t-s)^\gamma [u]_{\gamma,\alpha-2\gamma}\|v'\|_{\infty,\alpha-\gamma}\\ 
	& \leq C_G (t-s)^\gamma \rho_{\gamma,[s,t]}(\mathbf{X}) \|u,u'\|_{\cD^\gamma_{\mathbf{X},\alpha}} \|v,v'\|_{\cD^\gamma_{\mathbf{X},\alpha}}. 
\end{align*}
Finally, based on the boundedness of $\txtD G$, we obtain for the last term
\begin{align*}
	|\txtD G(u_s) (v'_t -v'_s)|_{{\alpha-\sigma-2\gamma}}\leq C_G (t-s)^\gamma [v']_{\gamma,\alpha-2\gamma}.
\end{align*}
For the remainder of $H(u,v)$, denoted by $R^{H}$, we get
\begin{align*}
	R^{H}_{s,t}&= \txtD G(u_t) (v_t -v_s) +(\txtD G (u_t) -\txtD G(u_s) )v_s- (\txtD^2 G(u_s)u'_sv_s +\txtD G(u_s)v'_s )\cdot (\delta X)_{s,t}\\
	& =\txtD G(u_t) (R^v_{s,t} +v'_s\cdot (\delta X)_{s,t})+(\txtD G (u_t) -\txtD G(u_s) )v_s- (\txtD^2 G(u_s)u'_sv_s +\txtD G(u_s)v'_s ) \cdot(\delta X)_{s,t}\\
	& = \txtD G(u_t) R^v_{s,t} + (\txtD G(u_t) -\txtD G(u_s) ) v'_s \cdot(\delta X)_{s,t} \\
	& +  \int_0^1 \txtD^2 G(r u_t +(1-r) u_s)(\delta u)_{s,t}v_s ~\txtd r -\txtD^2 G(u_s) u'_s v_s \cdot(\delta X)_{s,t}\\
	& = \txtD G(u_t) R^v_{s,t} + (\txtD G(u_t) -\txtD G(u_s) ) v'_s \cdot(\delta X)_{s,t} \\
	& + \int_0^1 \txtD^2 G(r u_t +(1-r) u_s)(u'_s\cdot (\delta X)_{s,t}+ R^u_{s,t})v_s~\txtd r -\txtD^2 G(u_s) u'_s v_s \cdot(\delta X)_{s,t}\\
	& = \txtD G(u_t) R^v_{s,t} + (\txtD G(u_t) -\txtD G(u_s) ) v'_s \cdot(\delta X)_{s,t} \\
	& +  \int_0^1 \txtD^2 G(r u_t +(1-r) u_s)R^u_{s,t}v_s ~\txtd r +\int_0^1 \big(\txtD^2 G(r u_t + (1-r) u_s) - \txtD^2 G(u_s)\big) u'_s v_s ~\txtd r \cdot(\delta X)_{s,t}\\
	& = \txtD G(u_t) R^v_{s,t} + (\txtD G(u_t) -\txtD G(u_s) ) v'_s \cdot(\delta X)_{s,t} +   \int_0^1 \txtD^2 G(r u_t +(1-r) u_s)R^u_{s,t}v_s~\txtd r \\
	&+\int_0^1\int_0^1 \tilde{r} \txtD^3 G(\tilde{r}(r u_t + (1-r) u_s) +(1-\tilde{r}) u_s)(\delta u)_{s,t} u'_s v_s ~\change{\txtd r~\txtd \tilde{r}}  \cdot (\delta X)_{s,t}.
\end{align*}
Using this representation we can obtain that the remainder $R^{H}$ is $\gamma$-H\"older in $E_{\alpha-\sigma-\gamma}$ respectively $2\gamma$-H\"older in $E_{\alpha-\sigma-2\gamma}$. Indeed, let $i=1,2$, then for the first term we have 
\begin{align*}
	|\txtD G(u_t) R^v_{s,t}|_{\alpha-\sigma-i\gamma}\leq C_G (t-s)^{i\gamma} [R^v]_{i\gamma,\alpha-i\gamma}.
\end{align*}
For the second one, we obtain
\begin{align*}
	| (\txtD G (u_t) -\txtD G(u_s))v'_s X_{s,t} |_{\alpha-\sigma-i\gamma} &\lesssim C_G \rho_{\gamma,[s,t]}(\mathbf{X})(t-s)^{2\gamma} [u]_{\gamma,\alpha-i\gamma} |v'|_{\infty,\alpha-\gamma} \\
	& \leq C_G \rho^2_{\gamma,[s,t]}(\mathbf{X}) (t-s)^{2\gamma} \|u,u'\|_{\cD^\gamma_{\mathbf{X},\alpha}} \|v,v'\|_{\cD^\gamma_{\mathbf{X},\alpha}}.
\end{align*}
The third one can be estimated similarly
\begin{align*}
	\Big|   \int_0^1 \txtD^2 G(r u_t +(1-r) u_s)R^u_{s,t}v_s~\txtd r\Big|_{\alpha-\sigma -i\gamma} \lesssim C_G (t-s)^{i\gamma} [R^u]_{i\gamma,\alpha-i\gamma} \|v\|_{\infty,\alpha}
\end{align*}
whereas the fourth one finally entails
\begin{align*}
	& \Big| \int_0^1\int_0^1 \tilde{r} \txtD^3 G\big(\tilde{r}(r u_t + (1-r) u_s) +(1-\tilde{r}) u_s\big)u'_s v_s (\delta u)_{s,t}~\change{\txtd r~\txtd \tilde{r}} \cdot (\delta X)_{s,t} \Big|_{\alpha-\sigma-i\gamma} \\
	&\lesssim C_G \rho_{\gamma,[s,t]}(\mathbf{X})(t-s)^{2\gamma} \|u'\|_{\infty,\alpha-\gamma} \|v\|_{\infty,\alpha} [u]_{\gamma,\alpha-i\gamma}\\
	& \leq C_G (t-s)^{2\gamma} \rho^2_{\gamma,[s,t]}(\mathbf{X}) \|u,u'\|^2_{\cD^{\gamma}_{\mathbf{X},\alpha}} \|v,v'\|_{\cD^{\gamma}_{\mathbf{X},\alpha}}.
\end{align*}
Putting all these estimates together entail \eqref{crp:linearization}.
\end{proof}
Now we are able to formulate a Gronwall inequality for the solution of the linearized equation. We recall that 
\begin{align}\label{SolLinearized}
v_t=U_{t,s} v_s+\int_{s}^{t}U_{t,r}\txtD F(u_r)v_r~\txtd r+\int_{s}^{t}U_{t,r}\txtD G(u_r)v_r ~\txtd \mathbf{X}_r,
\end{align}
is the mild solution of the linearized equation \eqref{linearization:rpde}. In order to handle the second integral, we need to impose more conditions on $F$. \change{We state them in the non-autonomous case for generality. }
\begin{itemize}
\item[\textbf{(DF)}\namedlabel{ass:DF}{\textbf{(DF)}}] 
We assume that $F$ is Fr\'echet differentiable for every, $t\in [0,T]$ there exists a constant $L_{DF,t}>0$ such that $DF(t,\cdot)$ is Lipschitz and $L_{DF}\coloneqq \sup_{t\in [0,T]} L_{DF,t}<\infty$. In particular, we have
\begin{align}\label{AssDF}
	\begin{split}
		\|\txtD F(t,x)-\txtD F(s,y)\|_{\mathcal{L}(E_\alpha;E_{\alpha-\delta})}&\leq L_{DF} |x-y|_{\alpha},\\
		\|\txtD F(t,x)\|_{\mathcal{L}(E_\alpha;E_{\alpha-\delta})}&\leq C_{DF} (1+|x|_\alpha),
	\end{split}
\end{align}
for $x,y\in E_\alpha$, $s,t\in [0,T]$, $L_{DF}>0$ and $C_{DF}:=\max\{L_{DF},\sup_{t\in [0,T]}|\txtD_2F(t,0)|_{\alpha-\delta}\}<\infty$. 
\end{itemize}
\begin{remark}   It is possible to extend our results to the case where the Fréchet derivative of $F$ satisfies a polynomial growth condition for every $t\in[0,T]$, for e.g.~$\|\txtD F(t,x)\|_{\mathcal{L}(E_\alpha;E_{\alpha-\delta})}\lesssim q(|x|_\alpha)$ for some polynomial $q$. For computational simplicity,  we work with the linear growth assumption.
\end{remark}
\begin{corollary}\label{cor:LinGronwall}
Suppose $A,F$ and $G$ satisfy the Assumptions \ref{ass:A},\ref{ass:F}-\ref{ass:DF} and \ref{ass:G1}-\ref{ass:G2}. Let $(u,u')\in \mathcal{D}^{\gamma}_{\mathbf{X},\alpha}$ be the solution to \eqref{rpde1} with initial value $u_0\in E_\alpha$ and $(v,v')\in \mathcal{D}^{\gamma}_{\mathbf{X},\alpha}$ the linearization along this solution satisfying the equation~\eqref{linearization:rpde}. Then $(v,v')=(v,\txtD_2G(\cdot,u)v)\in \mathcal{D}^\gamma_{\mathbf{X},\alpha}([s,t])$ and satisfies the estimate
\begin{align}\label{est:LinGronwall}
	\|v,\txtD_2G(\cdot,u)v\|_{\cD^\gamma_{\mathbf{X},\alpha}([s,t])}\leq \widetilde{C}_1\rho_{\gamma,[s,t]}(\mathbf{X})\left(|v_s|_\alpha+|\txtD_2G(s,u_s)v_s|_{\alpha-\gamma}\right)e^{\widetilde{C}_2(t-s)},
\end{align}
where the constants are given by
\begin{align}\label{eq:ConstantsLinGronwall}
	\widetilde{C}_1&:=e^{\widetilde{C}_2}\frac{1-C\kappa^\nu \widetilde{\Phi}_3}{2C\widetilde{\Phi}_2-1+C\kappa^\nu\widetilde{\Phi}_3},\quad \widetilde{C}_2:=\frac{1}{\kappa}\ln{\left(\frac{2C\widetilde{\Phi}_2}{1-C\kappa^\nu \widetilde{\Phi}_3}\right)},
\end{align}
with \change{$C:=C(U,\alpha,\sigma,\delta, \gamma)>1$}, $\nu=\min\{1-2\gamma,1-\delta,\gamma-\sigma\}$, \change{$\kappa>0$} such that \change{$C\kappa^\nu\widetilde{\Phi}_3<1$} and 
\begin{align*}
	\widetilde{\Phi}_2&:=\max\left\{1,C_G\rho_{\gamma,[s,t]}(\mathbf{X}),C_G^2\rho_{\gamma,[s,t]}(\mathbf{X})\right\},\\~ \widetilde{\Phi}_3&:=C_{DF}(1+\|u,u^\prime\|_{\cD^\gamma_{\mathbf{X},\alpha}([s,t])})+C_G\rho_{\gamma,[s,t]}(\mathbf{X})^3(1+\|u,u^\prime\|_{\cD^\gamma_{\mathbf{X},\alpha}([s,t])})^2. 
\end{align*}
\end{corollary}
\begin{proof}
Using Lemma \ref{lem:LinCRP} we obtain for $(u,u')$,  $(v,v')\in \cD^{\gamma}_{\mathbf{X},\alpha}$ and $t-s<1$
\begin{align*}
	\|U_{\cdot,s}v_s,0\|_{\cD^{\gamma}_{\mathbf{X},\alpha}([s,t])}&\lesssim |v_s|_\alpha,\\
	\left\|\int_s^\cdot U_{\cdot,r} \txtD_2F(r,u_r)v_r~\txtd r,0\right\|_{\cD^{\gamma}_{\mathbf{X},\alpha}([s,t])}&\lesssim C_{DF} (t-s)^{1-\change{\max\{2\gamma,\delta\}}}(1+\|u,u^\prime\|_{\cD^\gamma_{\mathbf{X},\alpha}([s,t])})\|v,v^\prime\|_{\cD^\gamma_{\mathbf{X},\alpha}([s,t])},\\
	\|\txtD_2G(\cdot,u)v, (\txtD_2G(\cdot,u)v)^\prime\|_{\cD^\gamma_{\mathbf{X},\alpha-\sigma}([s,t])}&\lesssim C_G \rho_{\gamma,[s,t]}(\mathbf{X})^2(1+\|u,u^\prime\|_{\cD^\gamma_{\mathbf{X},\alpha}([s,t])})^2\|v,v^\prime\|_{\cD^{\gamma}_{\mathbf{X},\alpha}([s,t])}.
\end{align*}
Combining these estimates with \eqref{est:RPIntegral} entails
\begin{align*}
	\|v&,\txtD_2G(\cdot,u)v\|_{\cD^\gamma_{\mathbf{X},\alpha}([s,t])} \lesssim |v_s|_\alpha+C_{DF} (t-s)^{1-\max\{2\gamma,\delta\}}(1+\|u,u^\prime\|_{\cD^\gamma_{\mathbf{X},\alpha}([s,t])})\|v,v^\prime\|_{\cD^\gamma_{\mathbf{X},\alpha}([s,t])} \\
	&+\rho_{\gamma,[s,t]}(\mathbf{X}) (|\txtD_2G(s,u_s)v_s|_{\alpha-\sigma}+|(\txtD_2G(s,u_s)v_s)^\prime|_{\alpha-\sigma-\gamma})\\
	&+\rho_{\gamma,[s,t]}(\mathbf{X})(t-s)^{\gamma-\sigma}\|\txtD_2G(\cdot,u)v, (\txtD_2G(\cdot,u)v)^\prime\|_{\cD^\gamma_{\mathbf{X},\alpha-\sigma}([s,t])},\\
	&\lesssim \widetilde{\Phi}_2(|v_s|_\alpha+|\txtD_2G(s,u_s)v_s|_{\alpha-\gamma})+\widetilde{\Phi}_3(t-s)^\nu\|v,\txtD_2G(\cdot,u)v\|_{\cD^\gamma_{\mathbf{X},\alpha}([s,t])}.
\end{align*}
Here, we used the fact that $u^\prime_s=G(s,u_s)$ to obtain
\begin{align*}
	|(\txtD_2G(s,u_s)v_s)^\prime|_{\alpha-\sigma-\gamma}\leq C_G(|u^\prime_s|_{\alpha-\gamma}|v_s|_{\alpha-\gamma}+|v_s^\prime|_{\alpha-\gamma})\lesssim C_G^2 |v_s|_{\alpha}+C_G|v_s^\prime|_{\alpha-\gamma}.
\end{align*}
The remaining proof can be shown as in Lemma~\ref{lem:RoughGronwall}. 
\end{proof}
This yields the fowling result.
\change{
\begin{corollary}\label{cor:LinGronwall_B}
Consider the setting of Corollary~\ref{cor:LinGronwall} and assume that $t-s<1$. 
Then there exists a polynomial $P$ such that 
\[
   \max\Bigl\{ 
      \widetilde{C}_1(u,\mathbf{X},s,t), \; 
      \widetilde{C}_2(u,\mathbf{X},s,t) 
   \Bigr\}
   \;\leq\;
   P\!\left(
      \|u, u'\|_{D^\gamma_{\mathbf{X},\alpha}([s,t])}, \;
      \rho_{\gamma,[s,t]}(\mathbf{X})
   \right),
\]
where 
$\widetilde{C}_1(u,\mathbf{X},s,t)$ and 
$\widetilde{C}_2(u,\mathbf{X},s,t)$ 
highlight the dependence of $\widetilde{C}_1$ and $\widetilde{C}_2$ 
on the corresponding parameters.~The polynomial $P$ is increasing with respect to both arguments.
\end{corollary}
\begin{proof}
From Corollary~\ref{cor:LinGronwall}, the parameter $\kappa$ satisfies
\begin{align}
   0 < \kappa^{\nu} < \frac{1}{C \Phi_{3}}. \label{eq:kappa-condition}
\end{align}
Choosing 
\[
   \kappa^{\nu} := \frac{1}{2 C \Phi_{3}}
\]
and substituting this into the expressions for 
$\widetilde{C}_1$ and $\widetilde{C}_2$ in \eqref{eq:ConstantsLinGronwall} 
yields the desired result. 
\end{proof}
}
In order to obtain stability statements (see for e.g. Theorem~\ref{stable_manifold}), we further need an estimate of the difference between two linearizations for two different initial data. Therefore, we let $u_0,\tilde{u}_0\in E_\alpha$ be two initial conditions and $u_t:=u_t^{u_0},\tilde{u}_t:=u_t^{\tilde{u}_0}$ the corresponding solutions to \eqref{rpde1}, with linearization $v_t$ and $\tilde{v}_t$. Then we are interested in the difference between the two solutions
\begin{align}\label{difference}
v_t-\tilde{v}_t&=U_{t,s}(v_s-\tilde{v}_s)+\int_{s}^t U_{t,r} \left[\txtD_2F(u_r)v_r-\txtD_2F(\tilde{u}_r)\tilde{v}_r \right]~\txtd r\\
&+\int_s^t U_{t,r}\left[\txtD_2G(u_r)v_r-\txtD_2G(\tilde{u}_r)\tilde{v}_r\right]~\txtd \mathbf{X}_r. \nonumber
\end{align}
Similar to Lemma \ref{lem:LinCRP} we first investigate 
\begin{align*}
\widetilde{H}(u_t,\tilde{u}_t,v_t,\tilde{v}_t)=\txtD G(u_t)v_t-\txtD G(\tilde{u}_t)\tilde{v}_t=H(u_t,v_t)-H(\tilde{u}_t,\tilde{v}_t),
\end{align*}
with Gubinelli derivative
\begin{align}\label{eq:GubDerivLinDiff}
(\widetilde{H}(u_t,\tilde{u}_t,v_t,\tilde{v}_t))^\prime=\txtD^2 G(u_t)u'_tv_t +\txtD G(u_t)v'_t-(\txtD^2 G(\tilde{u}_t)\tilde{u}'_t\tilde{v}_t +\txtD G(\tilde{u}_t)\tilde{v}'_t).
\end{align}
Now we derive a bound for $\tilde{H}$ depending on the difference between the controlled rough path norms of $(u-\tilde{u},u'-\tilde{u}')$, respectively $(v-\tilde{v},v'-\tilde{v}')$. 
\begin{lemma}\label{lem:LinCRPDiff}
Let $(u,u')\in \mathcal{D}^{\gamma}_{\mathbf{X},\alpha}, (\tilde{u},\tilde{u}')\in \mathcal{D}^{\gamma}_{\mathbf{X},\alpha}$ be two solutions of \eqref{rpde1} with initial data $u_0,\tilde{v}_0\in E_\alpha$ and $(v,v'), (\tilde{v},\tilde{v}')\in \mathcal{D}^{\gamma}_{\mathbf{X},\alpha}$ be the corresponding  linearizations. Additionally, we assume that $G$ is four times-Fréchet differentiable. 

Then we have $(\widetilde{H}(u,\tilde{u},v,\tilde{v}),(\widetilde{H}(u,\tilde{u},v,\tilde{v}))^\prime)\in\cD^{\gamma}_{\mathbf{X},\alpha-\sigma}$ and 
\begin{align}\label{crp:linearization2}
	\begin{split}
		\|\widetilde{H}(u,\tilde{u}&,v,\tilde{v}),(\widetilde{H}(u,\tilde{u},v,\tilde{v}))^\prime\|_{\cD^{\gamma}_{\mathbf{X},\alpha-\sigma}} \\
		&\leq CC_G \rho_{\gamma,[s,t]} (\mathbf{X})^2 \\
		&\times\Big(\|v-\tilde{v},v'-\tilde{v}^\prime\|_{\cD^{\gamma}_{\mathbf{X},\alpha}}\big((1+\|u,u'\|_{\cD^{\gamma}_{\mathbf{X},\alpha}})(1+\|\tilde{u},\tilde{u}'\|_{\cD^{\gamma}_{\mathbf{X},\alpha}})+\|\tilde{u},\tilde{u}'\|_{\cD^{\gamma}_{\mathbf{X},\alpha}}^2\big)\\
		&+\|u-\tilde{u},u'-\tilde{u}^\prime\|_{\cD^{\gamma}_{\mathbf{X},\alpha}}\big((1+\|u,u'\|_{\cD^{\gamma}_{\mathbf{X},\alpha}}+\|\tilde{u},\tilde{u}'\|_{\cD^{\gamma}_{\mathbf{X},\alpha}}+\|u,u'\|_{\cD^{\gamma}_{\mathbf{X},\alpha}}^2)\|v,v'\|_{\cD^{\gamma}_{\mathbf{X},\alpha}}\\
		&+(1+\|u,u'\|_{\cD^{\gamma}_{\mathbf{X},\alpha}}+\|\tilde{u},\tilde{u}'\|_{\cD^{\gamma}_{\mathbf{X},\alpha}})\|\tilde{v},\tilde{v}'\|_{\cD^{\gamma}_{\mathbf{X},\alpha}}\big)\Big).
	\end{split}
\end{align}
\end{lemma}
\begin{proof}
We have to derive estimates for the path component, Gubinelli derivative \eqref{eq:GubDerivLinDiff} and the remainder. We only focus on the bounds for the Gubinelli derivative and remainder. The other estimates follow {by a similar approach as in} Lemma~\ref{crp:linearization}. The path component 
\begin{align*}
	\txtD G(u_t)v_t-\txtD G(\tilde{u}_t)\tilde{v}_t=\big(\txtD G(u_t)-\txtD G(\tilde{u}_t)\big)v_t+\txtD G(\tilde{u}_t)(v_t-\tilde{v}_t),
\end{align*}
as well as the supremum norm of the Gubinelli derivative is straightforward to estimate
\begin{align*}
	\|\widetilde{H}(u,\tilde{u},v,\tilde{v})\|_{\infty,\alpha-\sigma} &\lesssim C_G\big(\|u-\tilde{u},u^\prime-\tilde{u}^\prime\|_{\mathcal{D}^\gamma_{\mathbf{X},\alpha}}\|v,v^\prime\|_{\mathcal{D}^\gamma_{\mathbf{X},\alpha}}+\|v-\tilde{v},v^\prime-\tilde{v}^\prime\|_{\mathcal{D}^\gamma_{\mathbf{X},\alpha}}\big),\\
	\| (\widetilde{H}(u,\tilde{u},v,\tilde{v}))^\prime\|_{\infty,\alpha-\sigma-\gamma}&\lesssim C_G\Big(\|u-\tilde{u},u^\prime-\tilde{u}^\prime\|_{\mathcal{D}^\gamma_{\mathbf{X},\alpha}}\|v,v^\prime\|_{\mathcal{D}^\gamma_{\mathbf{X},\alpha}}(1+\|u,u^\prime\|_{\mathcal{D}^\gamma_{\mathbf{X},\alpha}})\\
	&+\|v-\tilde{v},v^\prime-\tilde{v}^\prime\|_{\mathcal{D}^\gamma_{\mathbf{X},\alpha}}(1+\|\tilde{u},\tilde{u}^\prime\|_{\mathcal{D}^\gamma_{\mathbf{X},\alpha}})\Big).
\end{align*}
The estimates for the Hölder continuity of the Gubinelli derivative and the remainder are more involved. We compute 
\begin{align*}
	&\big(\widetilde{H}(u_t,\tilde{u}_t,v_t,\tilde{v}_t)-\widetilde{H}(u_s,\tilde{u}_s,v_s,\tilde{v}_s)\big)^\prime\\
	&=\Big(\big(\txtD^2G(u_t)-\txtD^2G(u_s)\big)-\big(\txtD^2G(\tilde{u}_t)-\txtD^2G(\tilde{u}_s)\big)\Big)u^\prime_t v_t \\
	&+\big(\txtD^2G(\tilde{u}_t)-\txtD^2G(\tilde{u}_s)\big)\big((u^\prime_t-\tilde{u}^\prime_t)v_t+\tilde{u}^\prime_t(v_t-\tilde{v}_t)\big)\\
	&+(\txtD^2G(u_s)-\txtD^2G(\tilde{u}_s))\big((\delta u^\prime)_{s,t}v_t+u^\prime_s(\delta v)_{s,t}\big)\\
	&+\txtD^2G(\tilde{u}_s)\big(((\delta u^\prime)_{s,t}-(\delta \tilde{u}^\prime)_{s,t})v_t+u_s^\prime((\delta v)_{s,t}-(\delta \tilde{v})_{s,t})+(u^\prime_s-\tilde{u}^\prime_s)(\delta \tilde{v})_{s,t}+(\delta \tilde{u}^\prime)_{s,t}(v_t-\tilde{v}_t)\big)\\
	&+\big(\txtD G(u_t)-\txtD G(u_s)\big)(v^\prime_t-\tilde{v}^\prime_t) \\
	&+\Big(\big(\txtD G(u_t)-\txtD G(u_s)\big)-\big(\txtD G(\tilde{u}_t)-\txtD G(\tilde{u}_s)\big)\Big)\tilde{v}^\prime_t \\
	&+\txtD G(u_s)((\delta v^\prime)_{s,t}- (\delta\tilde{v}^\prime)_{s,t})+\big(\txtD G(u_s)-\txtD G(\tilde{u}_s)\big)(\delta \tilde{v}^\prime)_{s,t}.
\end{align*}
Most of the terms above can easily be estimated as in Lemma \ref{lem:LinCRP}, the only non-trivial ones are the first and the second last line. These we can represent as
\begin{align*}
	\Big(\big(\txtD^2G(u_t)-\txtD^2G(u_s)\big)&-\big(\txtD^2G(\tilde{u}_t)-\txtD^2G(\tilde{u}_s)\big)\Big)u^\prime_t v_t \\
	&=\int_0^1 \big(\txtD^3G(r u_t+(1-r)u_s)-\txtD^3G(r \tilde{u}_t+(1-r)\tilde{u}_s)\big)(\delta u)_{s,t}u^\prime_tv_t~\txtd r \\
	&+\int_0^1 \txtD^3G(r \tilde{u}_t+(1-r)\tilde{u}_s)((\delta u)_{s,t}-(\delta \tilde{u})_{s,t})u^\prime_tv_t~\txtd r,\\
	\Big(\big(\txtD G(u_t)-\txtD G(u_s)\big)&-\big(\txtD G(\tilde{u}_t)-\txtD G(\tilde{u}_s)\big)\Big)\tilde{v}^\prime_t  \\
	&=\int_0^1 \big(\txtD ^2G(r u_t+(1-r)u_s)-\txtD ^2G(r \tilde{u}_t+(1-r)\tilde{u}_s)\big)(\delta u)_{s,t}\tilde{v}^\prime_t~\txtd r \\
	&+\int_0^1 \txtD ^2G(r \tilde{u}_t+(1-r)\tilde{u}_s)((\delta u)_{s,t}-(\delta \tilde{u})_{s,t})\tilde{v}^\prime_t~\txtd r.
\end{align*}
To estimate these integrals, we rely on a Lipschitz estimate for $\txtD^3 G$, which explains the assumption $G\in C_b^4$. Using similar estimates as in Lemma \ref{lem:LinCRP}, we obtain
\begin{align*}
	&\big[(\widetilde{H}(u,\tilde{u},v,\tilde{v}))^\prime\big]_{\gamma,\alpha-\sigma-2\gamma}\\
	&\lesssim C_G\rho_{\gamma,[s,t]}(\mathbf{X})\Big(\|v-\tilde{v},v^\prime-\tilde{v}^\prime\|_{\mathcal{D}^\gamma_{\mathbf{X},\alpha}}\big(1+\|u,u^\prime\|_{\mathcal{D}^\gamma_{\mathbf{X},\alpha}}+\|\tilde{u},\tilde{u}^\prime\|^2_{\mathcal{D}^\gamma_{\mathbf{X},\alpha}}+\|\tilde{u},\tilde{u}^\prime\|_{\mathcal{D}^\gamma_{\mathbf{X},\alpha}}\big)\\
	&+\|u-\tilde{u},u^\prime-\tilde{u}^\prime\|_{\mathcal{D}^\gamma_{\mathbf{X},\alpha}}\big(\|v,v^\prime\|_{\mathcal{D}^\gamma_{\mathbf{X},\alpha}}(1+\|u,u^\prime\|_{\mathcal{D}^\gamma_{\mathbf{X},\alpha}}+\|u,u^\prime\|_{\mathcal{D}^\gamma_{\mathbf{X},\alpha}}^2+\rho_{\gamma,[s,t]}(\mathbf{X})\|\tilde{u},\tilde{u}^\prime\|_{\mathcal{D}^\gamma_{\mathbf{X},\alpha}})\\
	&+(1+\|u,u^\prime\|_{\mathcal{D}^\gamma_{\mathbf{X},\alpha}})\|\tilde{v},\tilde{v}^\prime\|_{\mathcal{D}^\gamma_{\mathbf{X},\alpha}}\big)\Big).
\end{align*}
Using the representation of the remainder in Lemma \ref{lem:LinCRP} we obtain here for the remainder of $\tilde{H}$ denoted by $R^{\tilde{H}}$
\begin{align*}
	&R^{\widetilde{H}}_{s,t}=\big(\txtD G(u_t)-\txtD G(\tilde{u}_t))\big)R^v_{s,t}+\txtD G(\tilde{u}_t)\big(R^{v}_{s,t}-R^{\tilde{v}}_{s,t}\big)+\big(\txtD G(\tilde{u}_t)-\txtD G(\tilde{u}_s)\big)(v_s^\prime-\tilde{v}_s^\prime)\cdot(\delta X)_{s,t}\\
	&+\int_0^1 (\txtD^2G(r u_t+(1-r)\tilde{u}_t)-\txtD^2G(r u_s+(1-r)\tilde{u}_s))(u_t-\tilde{u}_t)v_s^\prime~\txtd r\cdot(\delta X)_{s,t}\\
	&+\int_0^1 \txtD^2G(r u_t+(1-r)\tilde{u}_t)((\delta u)_{s,t}-(\delta \tilde{u})_{s,t})v_s^\prime ~\txtd r\cdot(\delta X)_{s,t}\\
	&+\int_0^1 (\txtD^2G(r u_t+(1-r)u_s)-\txtD^2G(r \tilde{u}_t+(1-r)\tilde{u}_s)R^u_{s,t}v_s~\txtd r\\
	&+\int_0^1 \txtD^2G(r \tilde{u}_t+(1-r)\tilde{u}_s)(R^u_{s,t}-R^{\tilde{u}}_{s,t})v_s~\txtd r{+\int_0^1 \txtD^2G(r \tilde{u}_t+(1-r)\tilde{u}_s)R^{\tilde{u}}_{s,t}(v_s~-\tilde{v}_s)\txtd r}\\
	&+\int_0^1\int_0^1 \tilde{r}\big(\txtD^3G(\tilde{r}(u_s+r\tilde{r} (\delta u)_{s,t})-\txtD^3G(\tilde{u}_s+r\tilde{r} (\delta \tilde{u})_{s,t})\big)u^\prime_sv_s(\delta u)_{s,t} ~\txtd r\txtd \tilde{r}\cdot(\delta X)_{s,t}\\
	&+\int_0^1\int_0^1 \tilde{r} \txtD^3G(\tilde{u}_s+r\tilde{r} (\delta \tilde{u})_{s,t})\big((u^\prime_s-\tilde{u}^\prime_s)v_s(\delta u)_{s,t}+\tilde{u}^\prime_s(v_s-\tilde{v}_s) (\delta u)_{s,t}\big) ~\txtd r\txtd \tilde{r}\cdot(\delta X)_{s,t}\\
	&+\int_0^1\int_0^1 \tilde{r} \txtD^3G(\tilde{u}_s+r\tilde{r} (\delta \tilde{u})_{s,t})\tilde{u}^\prime \tilde{v}_s((\delta u)_{s,t}-(\delta\tilde{u})_{s,t}) ~\txtd r\txtd \tilde{r}\cdot(\delta X)_{s,t}.
\end{align*}
In conclusion
\begin{align*}
	\big[R^{\widetilde{H}}]_{i\gamma,\alpha-\sigma-i\gamma}&\lesssim C_G\rho_{\gamma,[s,t]}(\mathbf{X})^2\Big(\|v-\tilde{v},v^\prime-\tilde{v}^\prime\|_{\mathcal{D}^\gamma_{\mathbf{X},\alpha}}\big(1+\|\tilde{u},\tilde{u}^\prime\|_{\mathcal{D}^\gamma_{\mathbf{X},\alpha}}+\|\tilde{u},\tilde{u}^\prime\|_{\mathcal{D}^\gamma_{\mathbf{X},\alpha}}\|u,u^\prime\|_{\mathcal{D}^\gamma_{\mathbf{X},\alpha}}\big)\\
	&+\|u-\tilde{u},u^\prime-\tilde{u}^\prime\|_{\mathcal{D}^\gamma_{\mathbf{X},\alpha}}\big((1+\|u,u^\prime\|_{\mathcal{D}^\gamma_{\mathbf{X},\alpha}}+\|\tilde{u},\tilde{u}^\prime\|_{\mathcal{D}^\gamma_{\mathbf{X},\alpha}}+\|u,u^\prime\|^2_{\mathcal{D}^\gamma_{\mathbf{X},\alpha}})\|v,v^\prime\|_{\mathcal{D}^\gamma_{\mathbf{X},\alpha}}\\
	&+\|\tilde{v},\tilde{v}^\prime\|_{\mathcal{D}^\gamma_{\mathbf{X},\alpha}}\|\tilde{u},\tilde{u}^\prime\|_{\mathcal{D}^\gamma_{\mathbf{X},\alpha}}\big)\Big),
\end{align*}
which leads to \eqref{crp:linearization2}.
\end{proof}
\begin{remark}
The bound on the right-hand side of \eqref{crp:linearization2} naturally depends on $\|u,u'\|_{\cD^{\gamma}_{\mathbf{X},\alpha}}$, $\|\tilde{u},\tilde{u}'\|_{\cD^{\gamma}_{\mathbf{X},\alpha}}$, $\|v,v'\|_{\cD^{\gamma}_{\mathbf{X},\alpha}}$, $\|\tilde{v},\tilde{v}'\|_{\cD^{\gamma}_{\mathbf{X},\alpha}}$. For notational simplicity, we use further on
\begin{align}\label{crp:linearization3}
	\begin{split}
		\|&\widetilde{H}(u,\tilde{u},v,\tilde{v}),(\widetilde{H}(u,\tilde{u},v,\tilde{v}))^\prime\|_{\cD^{\gamma}_{\mathbf{X},\alpha-\sigma}}\\
		&\leq CC_G \rho_{\gamma,[s,t]} (\mathbf{X})^2 p(u,\tilde{u},v,\tilde{v})\Big(\|v-\tilde{v},v'-\tilde{v}^\prime\|_{\cD^{\gamma}_{\mathbf{X},\alpha}}
		+\|u-\tilde{u},u'-\tilde{u}^\prime\|_{\cD^{\gamma}_{\mathbf{X},\alpha}}\Big),
	\end{split}
\end{align}
for a polynomial $p(u,\tilde{u},v,\tilde{v})$.

\end{remark}

Applying Gronwall's inequality, stated in Lemma~\ref{lem:RoughGronwall}, to~\eqref{difference}, we obtain the following result. 
\begin{corollary}\label{cor:LinGronwall2}
Suppose $A,F$ and $G$ satisfy the Assumptions \ref{ass:A},~\ref{ass:F}-\ref{ass:DF},~\ref{ass:G1}-\ref{ass:G2} and additionally that $G$ is four times Fréchet-differentiable. Let $(u,u')\in \mathcal{D}^{\gamma}_{\mathbf{X},\alpha}, (\tilde{u},\tilde{u}')\in \mathcal{D}^{\gamma}_{\mathbf{X},\alpha}$ be two solutions of \eqref{rpde1} with initial data $u_0,\tilde{v_0}\in E_\alpha$ and $(v,v'), (\tilde{v},\tilde{v}')\in \mathcal{D}^{\gamma}_{\mathbf{X},\alpha}$ be the corresponding  linearizations. Then we obtain
\begin{align}\label{est:LinGronwall2}
	\begin{split}
		\|v-\tilde{v}&,\txtD_2G(\cdot,u)v-\txtD_2G(\cdot,\tilde{u})\tilde{v}\|_{\cD^\gamma_{\mathbf{X},\alpha}([s,t])}\leq \widehat{C}_1\rho_{\gamma,[s,t]}(\mathbf{X})\left(|v_s-\tilde{v}_s|_\alpha+|v^\prime_s-\tilde{v}^\prime_s|_{\alpha-\gamma}\right)e^{\widehat{C}_2(t-s)},
	\end{split}
\end{align}
where the constants are given by
\begin{align*}
	\widehat{C}_1&:=e^{\widehat{C}_2}\max\left\{\frac{1-C\theta^\nu \widehat{\Phi}_3}{2C\widehat{\Phi}_2-1+C\theta^\nu\widehat{\Phi}_3},\frac{(1-C\theta^\nu \widehat{\Phi}_3)C\widehat{\Phi}_1}{(C\theta^\nu\widehat{\Phi}_3+2C\widehat{\Phi}_2-1)^2}\right\},\quad \widehat{C}_2:=\frac{1}{\theta}\ln{\left(\frac{2C\widehat{\Phi}_2}{1-C\theta^\nu \widehat{\Phi}_3}\right)},
\end{align*}
with $C(U,\alpha,\sigma,\delta,\gamma)>0$, $\nu=\min\{1-2\gamma,1-\delta,\gamma-\sigma\}$, $\theta<1$ such that $2C\widehat{\Phi}_2>1-C\theta^\nu\widehat{\Phi}_3>0$ and 
\begin{align*}
	\widehat{\Phi}_1&:= \|v,v^\prime\|_{\mathcal{D}_{\mathbf{X},\alpha}^\gamma}+\|u-\tilde{u},u^\prime-\tilde{u}^\prime\|_{\mathcal{D}_{\mathbf{X},\alpha}^\gamma}\bigg( C_{DF} (t-s)^{1-\change{\max\{2\gamma,\delta\}}}\|v,v^\prime\|_{\mathcal{D}_{\mathbf{X},\alpha}^\gamma}\\
	&+(t-s)^{\gamma-\sigma}C_G \rho_{\gamma,[s,t]} (\mathbf{X})^3 p(u,\tilde{u},v,\tilde{v})+\rho_{\gamma,[s,t]} (\mathbf{X})\\
	&+C_G\big(\|\tilde{v},\tilde{v}^\prime\|_{\mathcal{D}_{\mathbf{X},\alpha}^\gamma}+\|u,u^\prime\|_{\mathcal{D}_{\mathbf{X},\alpha}^\gamma}\|v,v^\prime\|_{\mathcal{D}_{\mathbf{X},\alpha}^\gamma}+\|v,v^\prime\|_{\mathcal{D}_{\mathbf{X},\alpha}^\gamma}\big)\bigg),\\
	\widehat{\Phi}_2&:=1+\rho_{\gamma,[s,t]}(\mathbf{X})C_G(1+\|\tilde{u},\tilde{u}^\prime\|_{\mathcal{D}_{\mathbf{X},\alpha}^\gamma})\\
	\widehat{\Phi}_3&:=C_{DF} (t-s)^{1-\change{\max\{2\gamma,\delta\}}}(1+\|u,u^\prime\|_{\mathcal{D}_{\mathbf{X},\alpha}^\gamma})+(t-s)^{\gamma-\sigma}C_G \rho_{\gamma,[s,t]} (\mathbf{X})^3 p(u,\tilde{u},v,\tilde{v}).
\end{align*}
\end{corollary}
\begin{proof}       
Similar to Corollary \ref{cor:LinGronwall}, we obtain $\|U_{\cdot,s}(v_s-\tilde{v}_s),0\|_{\cD^{\gamma}_{\mathbf{X},\alpha}([s,t])}\lesssim |v_s-\tilde{v}_s|_\alpha$ and
\begin{align*}
	&\left\|\int_s^\cdot U_{\cdot,r} \left(\txtD_2F_r(u_r)v_r-\txtD_2F_r(\tilde{u}_r)\tilde{v}_r\right)~\txtd r,0\right\|_{\cD^{\gamma}_{\mathbf{X},\alpha}([s,t])}\\
	&\lesssim C_{DF} (t-s)^{1-\change{\max\{2\gamma,\delta\}}}\Big((1+\|u,u^\prime\|_{\cD^\gamma_{\mathbf{X},\alpha}([s,t])})\|v-\tilde{v},v^\prime-\tilde{v}^\prime\|_{\cD^\gamma_{\mathbf{X},\alpha}([s,t])}\\
	&+ \|v,v^\prime\|_{\cD^{\gamma}_{\mathbf{X},\alpha}([s,t])}(1+\|u-\tilde{u},u^\prime-\tilde{u}^\prime\|_{\cD^{\gamma}_{\mathbf{X},\alpha}([s,t])})\Big),
\end{align*}
where $t-s<1$. Together with \eqref{crp:linearization3} and \eqref{est:RPIntegral} we obtain
\begin{align*}
	\|v-\tilde{v}&,\txtD_2G(\cdot,u)v-\txtD_2G(\cdot,\tilde{u})\tilde{v}\|_{\cD^\gamma_{\mathbf{X},\alpha}([s,t])}\lesssim |v_s-\tilde{v}_s|_{\alpha}\\
	&+C_{DF} (t-s)^{1-\change{\max\{2\gamma,\delta\}}}\Big[(1+\|u,u^\prime\|_{\cD^\gamma_{\mathbf{X},\alpha}([s,t])})\|v-\tilde{v},v^\prime-\tilde{v}^\prime\|_{\cD^\gamma_{\mathbf{X},\alpha}([s,t])}\\
	&+ \|v,v^\prime\|_{\cD^{\gamma}_{\mathbf{X},\alpha}([s,t])}(1+\|u-\tilde{u},u^\prime-\tilde{u}^\prime\|_{\cD^{\gamma}_{\mathbf{X},\alpha}([s,t])})\Big]\\
	&+ \rho_{\gamma,[s,t]}(\mathbf{X}) \Big(|DG(s,u_s)v_s-DG(s,\tilde{u}_s)\tilde{v}_s|_{\alpha-\sigma}+|(DG(s,u_s)v_s-DG(s,\tilde{u}_s)\tilde{v}_s)^\prime|_{\alpha-\sigma-\gamma}\\
	&+(t-s)^{\gamma-\sigma}\|DG(\cdot,u)v-DG(\cdot,\tilde{u})\tilde{v},(DG(\cdot,u)v-DG(\cdot,\tilde{u})\tilde{v})^\prime\|_{\cD^\gamma_{\mathbf{X},\alpha-\sigma}([s,t])}\Big)\\
	&\lesssim |v_s-\tilde{v}_s|_{\alpha}\\
	&+C_{DF} (t-s)^{1-\change{\max\{2\gamma,\delta\}}}\Big[(1+\|u,u^\prime\|_{\cD^\gamma_{\mathbf{X},\alpha}([s,t])})\|v-\tilde{v},v^\prime-\tilde{v}^\prime\|_{\cD^\gamma_{\mathbf{X},\alpha}([s,t])}\\
	&+ \|v,v^\prime\|_{\cD^{\gamma}_{\mathbf{X},\alpha}([s,t])}(1+\|u-\tilde{u},u^\prime-\tilde{u}^\prime\|_{\cD^{\gamma}_{\mathbf{X},\alpha}([s,t])})\Big]\\
	&+ \rho_{\gamma,[s,t]}(\mathbf{X}) \Big(C_G\big(|v_s-\tilde{v}_s|_{\alpha}+|u_s-\tilde{u}_s|_{\alpha}|\tilde{v}_s|_\alpha\big)+C_G\big(|u_s-\tilde{u}_s|_{\alpha}|u^\prime_s|_{\alpha-\gamma}|v_s|_{\alpha}+|u_s^\prime-\tilde{u}^\prime_s|_{\alpha-\gamma}|v_s|_{\alpha}\\
	&+|v_s-\tilde{v}_s|_{\alpha}|\tilde{u}_s|_{\alpha}+|u_s-\tilde{u}_s|_{\alpha}|v^\prime_s|_{\alpha-\gamma}+|v^\prime_s-\tilde{v}^\prime_s|_{\alpha-\gamma}\big)\\
	&+(t-s)^{\gamma-\sigma}C_G \rho_{\gamma,[s,t]} (\mathbf{X})^2 p(u,\tilde{u},v,\tilde{v})\big(\|v-\tilde{v},v'-\tilde{v}^\prime\|_{\cD^{\gamma}_{\mathbf{X},\alpha}}
	+\|u-\tilde{u},u'-\tilde{u}^\prime\|_{\cD^{\gamma}_{\mathbf{X},\alpha}}\big)\Big)\\
	&\lesssim \widehat{\Phi}_1+\widehat{\Phi}_2\big(|v_s-\tilde{v}_s|_{\alpha}+|v^\prime_s-\tilde{v}^\prime_s|_{\alpha-\gamma}\big)+\widehat{\Phi}_3(t-s)^\nu\|v-\tilde{v},v^\prime-\tilde{v}^\prime\|_{\mathcal{D}_{\mathbf{X},\alpha}^\gamma}.
\end{align*}
As in the proof of Lemma \ref{lem:LinCRP}, this yields the claim.
\end{proof}
\begin{remark}
Note that the constants $\widehat{C}_1$ and $\widehat{C}_2$ used in \eqref{est:LinGronwall2} depend on the controlled rough path norms of the linearizations $v,\tilde{v}$. It is possible to use \eqref{est:LinGronwall} in order to bound those norms, resulting in a Gronwall inequality where the right-hand side only depends on $u,\tilde{u}$ and the initial conditions $v_s$ and $v_s^\prime$.
\end{remark}
\section{An application. Lyapunov exponents for random dynamical systems}\label{sec:InvSets}
In this section, we present a possible application of the rough Gronwall's inequality. The goal is to prove the existence of Lyapunov exponents. This can be done by using a multiplicative ergodic theorem for linearized rough partial differential equations in Subsection \ref{SubsecMET}. As a consequence, we obtain in Subsection \ref{inv:m} invariant manifolds, as for example stable and unstable manifolds. 

Since we are working in a parabolic setting on a scale of function spaces $(E_\alpha)_{\alpha\in \R}$ it is a natural question whether the Lyapunov exponents depend on the threshold $\alpha$. We will show in Subsection \ref{indep:LY} that this is not the case.
\subsection{Generation of a random dynamical system}
First, we give an overview on the theory of random dynamical systems \cite{Arn98} and invariant sets in order to investigate the long-time behavior of the solution of \eqref{Main_Equation} in form of Lyapunov exponents. To this aim, we shortly recall the concept of a non-autonomous random dynamical system in the context of rough paths.

Therefore, we fix a probability space $(\widetilde{\Omega},\widetilde{\mathcal{F}},\widetilde{\P})$ and recall the notion of a metric dynamical system, which describes a model of the noise.   
\begin{definition}\label{def:MDS}
The quadrupel $(\widetilde{\Omega},\widetilde{\mathcal{F}},\widetilde{\P},(\tilde{\theta}_t)_{t\in \R})$, where $\tilde{\theta}_t:\widetilde{\Omega}\to\widetilde{\Omega}$ is a measure-preserving transformation, is called a metric dynamical system if
\begin{itemize}
	\item[i)] $\change{\tilde{\theta}}_0=\Id_{\widetilde{\Omega}}$,
	\item[ii)] $(t,\tilde{\omega})\mapsto \tilde{\theta}_t\tilde{\omega}$ is $\mathcal{B}(\R)\otimes \widetilde{\mathcal{F}}-\widetilde{\mathcal{F}}$ measurable,
	\item[iii)] $\tilde{\theta}_{t+s}=\tilde{\theta}_t\circ\tilde{\theta}_s$ for all $t,s \in \R$. 
\end{itemize}
We call it an ergodic metric dynamical system if for any $(\tilde{\theta}_t)_{t\in\R}$-invariant set $A\in \widetilde{\mathcal{F}}$ we have $\widetilde{\P}(A)\in \{0,1\}$.
\end{definition}

We further specify the concept of rough path cocycles introduced in \cite[Definition 2]{BRS17}.
\begin{definition}\label{def:RPCocycle}
We call a pair
\begin{align*}
	\mathbf{X}=(X,\mathbb{X}):\widetilde{\Omega}\to C^{\gamma}_{{\rm loc}}(\R;\mathbb{R}^d) \times C^{2\gamma}_{{\rm loc}}(\Delta_\R;\R^d\otimes \R^d) 
\end{align*}
a ($\gamma$-H\"older) rough path cocycle if $\mathbf{X}|_{[0,T]}(\tilde{\omega})$ is a $\gamma$-H\"older rough path for every $T>0$ and $\tilde{\omega}\in\widetilde{\Omega}$ and the cocycle property $X_{s,s+t}(\tilde{\omega})= X_t(\tilde{\theta}_s\tilde{\omega})$ as well as $\mathbb{X}_{s,s+t}(\tilde{\omega})=\mathbb{X}_{t,0}(\tilde{\theta}_s\tilde{\omega})$ holds true for every $s\in \R,t\in[0,\infty)$ and $\tilde{\omega}\in\widetilde{\Omega}$.
\end{definition}
To define non-autonomous random dynamical systems, let $(\widetilde{\Omega}, \widetilde{\mathcal{F}}, \widetilde{\P},(\tilde{\theta}_{t})_{t\in\R})$ be an ergodic metric dynamical system as defined in Definition \ref{def:MDS}. We further need the so-called symbol space. Similar to how the metric dynamical system describes the time evolution of the noise, the symbol space describes the temporal change of the non-autonomous terms. 
\begin{definition}\label{def:SymbolSpace}
We call $(\Sigma,(\vartheta_{t})_{t\in\R})$ a symbol space, if $\Sigma$ is a Polish metric space and $\vartheta:\R\times \Sigma\to \Sigma$ satisfies 
\begin{itemize}
	\item[i)] $\vartheta_{0}=\textrm{Id}_{\Sigma}$,
	\item[ii)] $(t,\hat{\omega})\mapsto \vartheta_{t}(\hat{\omega})$ is continuous,
	\item[iii)] $\vartheta_{t+s}=\vartheta_{t}\circ \vartheta_{s}$ for all $t,s \in \R$. 
\end{itemize}
\end{definition}


The construction of $(\Sigma,(\vartheta_{t})_{t\in\R})$ in our specific setting will be discussed later on. First, we conclude with the definition of a random dynamical system for non-autonomous systems. Note that we can recover the classical definition of an autonomous random dynamical system by setting $\Sigma=\emptyset$.
\begin{definition}\label{n:rds}
A continuous non-autonomous random dynamical system on a separable Banach space $E$ over a metric dynamical system $(\widetilde{\Omega},\widetilde{\mathcal{F}},\widetilde{\P},(\tilde{\theta}_{t})_{t\in\R})$ and symbol space $(\Sigma,(\vartheta_{t})_{t\in \R})$ is a mapping 
\[\phi:[0,\infty)\times \widetilde{\Omega}\times \Sigma\times E\to E, (t,\tilde{\omega},\hat{\omega},x)\mapsto \phi(t,\tilde{\omega},\hat{\omega},x),\]
which is $(\mathcal{B}([0,\infty))\otimes  \widetilde{\mathcal{F}}\otimes \mathcal{B}(\Sigma)\otimes\mathcal{B}(E),\mathcal{B}(E))$-measurable and satisfies
\begin{itemize}
	\item[i)] $\phi(0,\tilde{\omega},\hat{\omega},\cdot)=\Id_E$ for every $\tilde{\omega}\in \widetilde{\Omega},\hat{\omega}\in \Sigma$,
	\item[ii)] $\phi(t+s,\tilde{\omega},\hat{\omega},x)=\phi(t,\tilde{\theta}_{s}\tilde{\omega},\vartheta_{s} \hat{\omega},\phi(s,\tilde{\omega},\hat{\omega},x))$ for all $\tilde{\omega}\in \widetilde{\Omega},\hat{\omega}\in \Sigma$, $t,s\in [0,\infty)$ and $x\in E$,					
	\item[iii)] the map $\phi(t,\tilde{\omega},\hat{\omega},\cdot):E\to E$ is continuous for every $t\in [0,\infty)$ and $\tilde{\omega}\in \widetilde{\Omega},\hat{\omega}\in\Sigma$.
\end{itemize}
\end{definition}
The strategy in this article is now the following: Instead of using the non-autonomous random dynamical system directly, we treat the time-dependencies as another random forcing. To be precise, we enlarge the probability space by the symbol space, which enables us to use results for autonomous random dynamical systems and makes the presentation clearer. 

In order to incorporate the time-dependence in a larger probability space, \change{we have to assume that the linear operator satisfies the structural assumption $A(t)=A(\xi(t))$, which means that $\xi$ collects the time-dependence of the linear part of the equation, for example $A(t)=\xi(t)\Delta=A(\xi(t))$. Further details and examples can be looked up in Chepyzhov and Vishik \cite[Chapter IV]{CV02}. 
} Together with the time-dependencies incorporated by the nonlinearities, we define the time symbol of the equation \eqref{Main_Equation} by $$\mathfrak{S}:\R\to\mathcal{X}: t\mapsto \mathfrak{S}(t):=(\xi(t), F(t,\cdot),G(t,\cdot))$$ for some topological Hausdorff function space $\mathcal{X}$. 

We note that the long-time behavior of the solution of \eqref{Main_Equation} should not be affected if we shift $\mathfrak{S}(t)$ in time $\mathfrak{S}(t+s)$ by some $s\in \R$. Therefore, we look for a space $\Sigma$ which is invariant under the time shift $\vartheta_{t}y(\cdot):=y(\cdot + t)$. The natural choice of $\Sigma$ would be the collection of all time shifts of the original time symbol.~Therefore, we define the hull of $\mathfrak{S}$ 
\begin{align*}
\mathcal{H}(\mathfrak{S}):=\overline{\{\mathfrak{S}(\cdot +s)~\colon~s\in \R\}}^{\mathcal{X}}
\end{align*}
as the completion of the set of time shifts with respect to the topology of $\mathcal{X}$. Indeed, $\mathcal{H}(\mathfrak{S})$ is invariant under $(\vartheta_t)_{t\in \R}$. So, we define $\Sigma\coloneqq \mathcal{H}(\mathfrak{S})$. 

As the symbol space is now constructed, we can discuss how to enlarge the probability space to incorporate $\Sigma$. The main task  is to equip $(\Sigma,\mathcal{B}(\Sigma))$ with a probability measure $\P_\Sigma$, which leaves $(\vartheta_{t})_{t\in \R}$ invariant. Afterward, we consider the extended metric dynamical system 
\begin{align}\label{ExtendedProbSpace}
\big(\Omega,\mathcal{F},\P,(\theta_t)_{t\in \R}\big)\coloneqq\big(\widetilde{\Omega}\times \Sigma,\widetilde{\mathcal{F}}\otimes\mathcal{B}(\Sigma),\widetilde{\P}\otimes \P_\Sigma,(\tilde{\theta}_{t},\vartheta_{t})_{t\in\R}).
\end{align}
The construction of the probability measure on $(\Sigma, \mathcal{B}(\Sigma))$ follows from the Krylov-Bogolyubov theorem, which needs the compactness of $\Sigma$. With a translation compactness condition for $\mathfrak{S}$, one can prove that the hull is a compact Polish metric space. We refer to Appendix \ref{AppendixA} for more details. Keeping this in mind, we impose the following assumption:
\begin{itemize}
\item[\textbf{(S)}\namedlabel{ass:SymbolSpace}{\textbf{(S)}}] The hull $\mathcal{H}(\mathfrak{S})$ is a compact Polish metric space. 
\end{itemize}
\change{If Assumption \ref{ass:SymbolSpace} is satisfied, we define the symbol space  $\Sigma:= \mathcal{H}(\mathfrak{S})$ with translation operator $\vartheta_{t}y:=y(\cdot+t)$ for every $y\in \Sigma$.}
\begin{theorem}\label{thm:KrylovBogol}
There exists at least one probability measure $\P_\Sigma$ on  $(\Sigma,\mathcal{B}(\Sigma))$ such that $(\vartheta_{t})_{t\in \R}$ is invariant under $\P_\Sigma$ such that $\P_\Sigma(\{\mathfrak{S}(\cdot+h)~\colon~h\in \R\})=1$.
\end{theorem}
\begin{proof}
Due to the compactness of $\Sigma$, a direct application of the Krylov-Bogolyubov theorem \cite[Theorem 1.1]{BCDJMPSSV89} entails that 
\[ \nu \coloneq\lim\limits_{T\to \infty}\frac{1}{T}\int_0^T \delta_{\vartheta_{t}\mathfrak{S}(\cdot)}~\txtd t\] is a probability measure on $(\Sigma,\mathcal{B}(\Sigma))$. 
Since $$\delta_{\vartheta_{t}\mathfrak{S}(\cdot)}(\{\mathfrak{S}(\cdot+h)~\colon~h\in \R\})=\delta_{\mathfrak{S}(\cdot+t)}(\{\mathfrak{S}(\cdot+h)~\colon~h\in \R\})=1,$$ we obtain $\nu(\{\mathfrak{S}(\cdot+h)~\colon~h\in \R \})=1$, which proves the claim.
\end{proof}
The ergodicity of the resulting metric dynamical system \eqref{ExtendedProbSpace} follows by the existence of an ergodic decomposition of $\P_\Sigma$, see \cite[Page 539]{Arn98}.
\begin{corollary}
The quadrupel defined in \eqref{ExtendedProbSpace} is an ergodic metric dynamical system.
\end{corollary}
\subsection{Multiplicative ergodic theorem}\label{SubsecMET}
In this section, we use the integrable bounds obtained in Section \ref{integrable} and apply 
Gronwall's lemma is used to verify the integrability condition of the multiplicative ergodic theorem. This entails the existence of Lyapunov exponents for the rough PDE~\eqref{Main_Equation}. 
These values are essential for determining various dynamical phenomena, including stability, instability, chaos, and bifurcations. 

As a consequence of the rough Gronwall lemma and the computations on the linearized equation in Section \ref{sec:LinDynamics} we can now state the conditions that we need in order to use the multiplicative ergodic theorem.~Based on the sign of the Lyapunov exponents, one can further derive stable, unstable and center manifolds.~First, we recall that the probability space is given by $\Omega=\widetilde{\Omega}\times\Sigma$, where $\widetilde{\Omega}$ represents the randomness described by the noise and the symbol space $\Sigma$ is constructed in order to incorporate the time dependencies. 
To compress the notation, we define $\varphi^t_\omega(\change{Y_\omega}):=\varphi(t,\omega,\change{Y_\omega})$ as the solution of~\eqref{Main_Equation} with initial condition $\change{Y_\omega}$ for $\omega=(\tilde{\omega},\hat{\omega})\in\Omega$, compare Definition~\ref{n:rds}. 
\begin{definition}\label{def:stationary_point}
A random point \( Y \colon \Omega \rightarrow E_{\alpha} \) is referred to as a {stationary point} for the cocycle \( \varphi \) if it satisfies the following conditions:
\begin{enumerate}
	\item The map \( \omega \mapsto \vert Y_{\omega} \vert_{\alpha} \) is measurable,
	\item for every \( t > 0 \) and \( \omega \in \Omega \) {we have} \( \varphi^{t}_{\omega}(Y_{\omega}) = Y_{\theta_{t}\omega} \).
\end{enumerate}
Note that a stationary point can be regarded as an invariant measure in the sense of random dynamical systems by setting \( \mu := \delta_{Y_{\omega}} \times \P(\mathrm{d}\omega) \); see also \cite[Lemma 7.2.1]{Arn98}. 
\end{definition}

Now we fix a stationary point $(Y_\omega)_{\omega\in \Omega}$ and let $\psi(t,\omega,\cdot)=:\psi^t_\omega$ be the linearization along $(Y_\omega)_{\omega\in \Omega}$, as investigated in Section \ref{sec:LinDynamics}. More precisely, recalling that $\mathbf{X}=\mathbf{X}(\tilde{\omega})$ is a rough path cocycle as introduced in Definition~\ref{def:RPCocycle},
the linearization of~\eqref{Main_Equation} around $Y_\omega$ is given {by the solution of}
\begin{align}
\begin{cases}
	\txtd v =[A(t) v    + \txtD_2 F(t, Y_{\theta_t\omega})] v_t~\txtd t + \txtD_2 G(t, Y_{\theta_t\omega})v_t~\txtd \mathbf{X}_t(\tilde{\omega})\\
	v_0\in E_\alpha.
\end{cases}
\end{align}
We set $\psi^t_\omega(v_0):=v^t_\omega(v_0)$.

\begin{lemma}\label{lem:ExRDS}
Under the Assumptions \ref{ass:A1}-\ref{ass:A3}, \ref{ass:F}, \ref{ass:G1}-\ref{ass:G2} and \ref{ass:SymbolSpace} the solution operator $\phi$ of \eqref{Main_Equation} generates a continuous random dynamical system. If further \ref{ass:DF} is satisfied and $A(t)$ admits a compact inverse for every $t\in [0,T]$, then the solution operator $\psi$ of the linearized equation along the stationary point $(Y_\omega)_{\omega\in \Omega}$ is a compact linear random dynamical system, \change{ meaning that $\psi(t,\omega,\cdot):E_\alpha\to E_\alpha$ is a compact linear operator}.
\end{lemma}
\begin{proof}
We first prove that \eqref{Main_Equation} generates a continuous random dynamical system. \change{For $\omega\in \Omega$, let $\big\|u(\omega),\big(u(\omega)\big)^\prime\big\|_{D^\gamma_{\mathbf X(\tilde{\omega}), \alpha}}$ be the global solution of  \eqref{Main_Equation} and denote path component by $\varphi^t_{\omega}(x)\coloneqq u_t(\omega)$, where $x\in E_\alpha$ is the initial condition. Using the fact that the path component satisfies the mild formulation, we obtain
	\begin{align*}
		\varphi^{t+s}_{\omega}(x)&=U_{t+s,0}x+\int_0^{t+s} U_{t+s,r}F\big(r,\varphi^r_\omega(x)\big)~\txtd r+\int_0^{t+s} U_{t+s,r}G\big(r,\varphi^r_\omega(x)\big)~\txtd \mathbf{X}_r(\tilde{\omega})\\
		&=U_{t+s,s}U_{s,0}x+U_{t+s,s}\int_0^{s} U_{s,r}F\big(r,\varphi^r_\omega(x)\big)~\txtd r+U_{t+s,s}\int_0^{s} U_{s,r}G\big(r,\varphi^r_\omega(x)\big)~\txtd \mathbf{X}_r(\tilde{\omega})\\
		&+\int_s^{t+s}U_{t+s,r}F\big(r,\varphi^r_\omega(x)\big)~\txtd r+\int_s^{t+s} U_{t+s,r}G\big(r,\varphi^r_\omega(x)\big)~\txtd \mathbf{X}_r(\tilde{\omega})\\
		&=U_{t+s,s}\varphi(t,\omega,x)+\int_s^{t+s}U_{t+s,r}F\big(r,\varphi^r_\omega(x)\big)~\txtd r+\int_s^{t+s} U_{t+s,r}G\big(r,\varphi^r_\omega(x)\big)~\txtd \mathbf{X}_r(\tilde{\omega}).
	\end{align*}
	Furthermore, we emphasize that the evolution family also depends on the symbol $\hat \omega\in \Sigma$, but this dependence is often omitted for notational simplicity. In particular, in this situation we have $U^{\hat \omega}_{t+s,r+s}=U_{t,r}^{\vartheta_s \hat \omega}$.
	Together with the shift property of the rough convolution, see \cite[Lemma 8]{HN20}, this yields
	\begin{align*}
		\varphi^{t+s}_\omega(x)&=U^{\hat \omega}_{t+s,s}\varphi^t_\omega(x)+\int_s^{t+s}U^{\hat \omega}_{t+s,r}F\big(r,\varphi^r_\omega(x)\big)~\txtd r+\int_s^{t+s} U^{\hat \omega}_{t+s,r}G\big(r,\varphi^r_\omega(x)\big)~\txtd \mathbf{X}_r(\tilde{\omega})\\
		&=U^{\hat \omega}_{t+s,s}\varphi^t_\omega(x)+\int_0^{t}U^{\hat \omega}_{t+s,r+s}F\big(r+s,\varphi^{r+s}_{\omega}(x)\big)~\txtd r\\
		&+\int_0^{t} U^{\hat \omega}_{t+s,r+s}G\big(r+s,\varphi(r+s,\omega,x)\big)~\txtd \big(\tilde{\theta}_s\mathbf{X}_r\big)(\tilde{\omega})\\
		&=U^{\vartheta_s {\hat \omega}}_{t,0}\varphi^s_\omega(x)+\int_0^{t} U^{\vartheta_s {\hat \omega}}_{t,r}F\big(r+s,\varphi^{r+s}_\omega(x)\big)~\txtd r\\
		&+\int_0^{t} U^{\vartheta_s {\hat \omega}}_{t,r}G\big(r+s,\varphi^{r+s}_\omega(x)\big)~\txtd \big(\tilde{\theta}_s\mathbf{X}_r\big)(\tilde{\omega})=\varphi^t_{\theta_s \omega}\big(\varphi^s_\omega(x)\big),
	\end{align*}
	which verifies the cocycle property. }


The measurability follows from well-known arguments, using a sequence of classical solutions to \eqref{Main_Equation} corresponding to smooth approximations of $\mathbf{X}$. Since the solution depends continuously on the rough input $\mathbf{X}$, the approximating sequence of solutions converges to the solution corresponding to $\mathbf{X}$. Using this, it is easy to see that $\varphi^t:\Omega\times E_\alpha \to E_\alpha$ is measurable and $\varphi^\cdot_\omega(x):[0,\infty)\to E_\alpha$ is continuous. Then \cite[Lemma 3.14]{CV77} yields the measurability of $\varphi$.~Moreover, $\psi$ is obviously a random dynamical system. 
We only need to show the compactness. Since $A(t)$ has a compact inverse, we know that the  Banach spaces $(E_\alpha)_{\alpha\in \R}$ are compactly embedded~\cite[Theorem V.1.5.1]{Amann1995}. 
Using the smoothing property of the parabolic evolution family, one can show that $\psi_\omega^t\in \cL(E_\alpha;E_{\alpha+\varepsilon})$ for some small $\varepsilon>0$.~Then the compactness of the embedding $E_{\alpha+\varepsilon}\hookrightarrow E_\alpha$ yields the claim. 
\end{proof}
\begin{proposition}\label{MET_INTEG}
Let the same assumptions of Lemma~\ref{lem:ExRDS} be satisfied as well as \ref{ass:Noise} and fix a time $0<t_0<1$. \change{Moreover, we further assume that \ref{ass:F}, \ref{ass:G1}-\ref{ass:G2} hold for $t\in\R$. }We further impose that the stationary point fulfills for every $p\geq 1$ that
\begin{align}\label{eq:IntInitCond}
	\big(\omega \mapsto \vert Y_{\omega} \vert_{\alpha}\big) \in \bigcap_{p \geq 1} L^{p}(\Omega).
\end{align}
Then we have
\begin{align}
	\E\big[\sup_{0\leq t\leq t_0} \log^+(\|\psi^t_\cdot\|_{\mathcal{L}(E_\alpha)})\big]<\infty,\label{intCond1}\\
	\E\big[\sup_{0\leq t\leq t_0} \log^+(\|\psi^{t_0-t}_{\theta_t\cdot}\|_{\mathcal{L}(E_\alpha)})\big]< \infty,\label{intCond2}            
\end{align}
\change{where $\psi$ denotes the solution of the linearization around the stationary point $(Y_\omega)_{\omega\in \Omega}$.}
\end{proposition}
\begin{proof}
From the mild Gronwall inequality in Corollary~\ref{cor:LinGronwall}, it follows for $t \in [0,t_0]$ that
\begin{align*}
	\|\psi^t_\omega\|_{\mathcal{L}(E_\alpha)}
	&=\sup_{|x|_\alpha=1} |\psi^t_\omega(x)|_\alpha \\
	&\leq \widetilde{C}_1\!\left(Y_{\omega},\mathbf{X}(\tilde{\omega}),0,t\right)\,
	\rho_{\gamma,[0,t]}(\mathbf{X}(\tilde{\omega}))\, 
	e^{t \widetilde{C}_2\!\left(Y_{\omega},\mathbf{X}(\omega),0,t\right)}(1+C_G).
\end{align*}
\change{In particular, this yields
\begin{align}\label{ineq:LogSol}
\begin{split}
	\sup_{0 \leq t \leq t_0} 
	\log^+ \left(\|\psi^t_{\omega}\|_{\mathcal{L}(E_\alpha)}\right)
	&\leq \sup_{t\in [0,t_0]}\log\!\Bigl(\widetilde{C}_1(Y_{\omega},\mathbf{X}(\tilde{\omega}),0,t)\,
	\rho_{\gamma,[0,t_0]}(\mathbf{X}(\tilde{\omega}))(1+C_G)\Bigr) \\
	&\quad + t_0 \sup_{t\in [0,t_0]}\widetilde{C}_2\left(Y_{\omega},\mathbf{X}(\tilde{\omega}),0,t\right).
\end{split}
\end{align}
By Corollary~\ref{cor:LinGronwall_B} there exists a polynomial $P$, which is increasing in both arguments, such that
\begin{align*}
    \sup_{0 \leq t \leq t_0} &\max\Bigl\{ 
      \widetilde{C}_1\bigl(Y_{\omega},\mathbf{X}(\tilde{\omega}),0,t\bigr), \; 
      \widetilde{C}_2\bigl(Y_{\omega},\mathbf{X}(\tilde{\omega}),0,t\bigr) 
   \Bigr\} \\
   &\leq P\!\left(
      \|Y_{\omega}, (Y_{\omega})^{\prime}\|_{D^\gamma_{\mathbf{X}(\tilde{\omega}),\alpha}([0,t_0])}, \;
      \rho_{\gamma,[0,t_0]}(\mathbf{X}(\tilde{\omega}))
   \right).
\end{align*}
}
Since $\mathbf{X}(\tilde{\omega})$ satisfies the assumption~\ref{ass:Noise} we obtain that
\begin{align}\label{eq:IntSol1}
	\tilde{\omega}\mapsto \rho_{\gamma,[0,t_0]}(\mathbf{X}(\tilde{\omega}))\in \bigcap_{p\geq 1} L^p(\widetilde{\Omega}),
\end{align}
by \cite[Theorem 10.4 b)]{FH20}. 
\change{Furthermore, 
since $P$ is a polynomial, Theorem~\ref{thm:IntegrableBounds} and \eqref{eq:IntInitCond} imply that
\[
  P\!\left(
      \|Y_{\omega}, (Y_{\omega})^{\prime}\|_{D^\gamma_{\mathbf{X}(\tilde{\omega}),\alpha}([0,t_0])}, \;
      \rho_{\gamma,[0,t_0]}(\mathbf{X}(\tilde{\omega}))
   \right)\in \bigcap_{p\geq 1} L^p(\Omega).
\]
Here we used that the bounds in $L^p(\tilde{\Omega})$ hold for every $t_0<1$ in order to get integrability with respect to $\P_\Sigma$.
The second integrability condition \eqref{intCond2} can be shown analogously. Indeed, we obtain
\begin{align*}
	\sup_{0 \leq t \leq t_0} \log^+\!\bigl(\|\psi^{t_0-t}_{\theta_t \omega}\|_{\mathcal{L}(E_\alpha)}\bigr)
	&\leq \sup_{0 \leq t \leq t_0}\log\!\Bigl(\widetilde{C}_1\!\left(Y_{\theta_t\omega},\mathbf{X}(\tilde{\theta_t}\tilde{\omega}),0,t\right)
	\rho_{\gamma,[0,t_0]}(\mathbf{X}(\tilde{\theta_t}\tilde{\omega}))(1+C_G)\Bigr) \\
	&\quad + t_0 \sup_{t\in [0,t_0]}\widetilde{C}_2\!\left(Y_{\theta_t\omega},\mathbf{X}(\tilde{\theta_t}\tilde{\omega}),0,t\right).
\end{align*}    
This further leads to
\begin{align*}
    \sup_{0 \leq t \leq t_0} &\max\Bigl\{ 
      \widetilde{C}_1\!\left(Y_{\theta_t\omega},\mathbf{X}(\tilde{\theta_t}\tilde{\omega}),0,t\right), \; 
      \widetilde{C}_2\!\left(Y_{\theta_t\omega},\mathbf{X}(\tilde{\theta_t}\tilde{\omega}),0,t\right) 
   \Bigr\} \\
   &\leq \sup_{0 \leq t \leq t_0} P\!\left(
      \|Y_{\theta_t\omega}, (Y_{\theta_t\omega})^{\prime}\|_{D^\gamma_{\mathbf{X}(\tilde{\omega}),\alpha}([0,t_0-t])}, \;
      \rho_{\gamma,[0,t_0-t]}(\mathbf{X}(\tilde{\theta_t}\tilde{\omega}))
   \right) \\
   &\leq P\!\left(
      \|Y_{\omega}, (Y_{\omega})^{\prime}\|_{D^\gamma_{\mathbf{X}(\tilde{\omega}),\alpha}([0,t_0])}, \;
      \rho_{\gamma,[0,t_0]}(\mathbf{X}(\tilde{\omega}))
   \right)\in \bigcap_{p\geq 1} L^p(\Omega),
\end{align*}
which proves the statement.}
\end{proof}
In order to state the multiplicative ergodic theorem and its consequences we further fix some notations. The distance between two sets \(A\) and \(B\) of a Banach space $(\tilde{E},\Vert.\Vert_{\tilde{E}})$ is defined as
\[
d_{\tilde{E}}(A, B) := \inf_{a \in A, b \in B} \|a - b\|_{\tilde{E}}.
\]
For an element \(x \in \tilde{E}\) and a set \(B \subseteq \tilde{E}\), we set

\[
d_{\tilde{E}}(x, B) = d_{\tilde{E}}(B, x) := d_{\tilde{E}}(\{x\}, B).
\]
Furthermore, for $k\geq 1$ and elements \(x_1, \dots, x_k \in \tilde{E}\), we define the volume as
\[
\operatorname{Vol}_{\tilde{E}}(x_1, x_2, \dots, x_k) := \|x_1\|_{\tilde{E}} \prod_{i=2}^{k} d_{\tilde{E}}(x_i, \langle x_j \rangle_{1 \leq j < i}),
\]
where \(\langle x_j \rangle_{1 \leq j < i}\) denotes the linear span of \(x_1, \dots, x_{i-1}\). Note that $\operatorname{Vol}_{\tilde{E}}$ is not necessarily invariant under permutations unless $\tilde{E}$ is a Hilbert space. However, it still satisfies the following important property.
\begin{lemma}\label{permutation}
We assume that $\tilde{E}$ is an arbitrary Banach space and let $\sigma$ be a permutation of the set $\{1,2,\dots,k\}$. Then there exists a constant $M_k$, independent of $\tilde{E}$, such that
\begin{align*}
	\frac{1}{M_k} \leq \frac{\operatorname{Vol}_{\tilde{E}}(x_1, x_2, \dots, x_k)
	}{\operatorname{Vol}_{\tilde{E}}(x_{\sigma(1)}, x_{\sigma(2)}, \dots, x_{\sigma(k)})} \leq M_k
\end{align*}
for every set of linearly independent vectors $x_1, \dots, x_k$ in $\tilde{E}$.
\end{lemma}
\begin{proof}
By \cite[Proposition 2.14]{Blu16}, there exists an inner product $(\cdot, \cdot)_V$ on
\[
V := \langle x_i \rangle_{1 \leq i \leq k}
\]
such that
\begin{align*}
	\frac{1}{\sqrt{k}} \| x \|_{\tilde{E}} \leq \| x \|_V \leq \sqrt{k} \| x \|_{\tilde{E}}
	\quad \text{for all } x \in \langle x_i \rangle_{1 \leq i \leq k},
\end{align*}
which shows claim given that
 the volume $\operatorname{Vol}_{\tilde{E}}(x_{\sigma(1)}, x_{\sigma(2)}, \dots, x_{\sigma(k)})$ on the Hilbert space $V$ is invariant under permutations. 
\end{proof}

In the following sequel we use our previous results, in particular the mild Gronwall Lemma~\ref{lem:RoughGronwall} in order to obtain the existence of Lyapunov exponents for the random dynamical system constructed from  the linearization of the non-autonomous rough PDE \eqref{Main_Equation} along a stationary point. 
\begin{theorem}\label{METT}
We assume the same conditions as in Proposition \ref{MET_INTEG}. Let \( \varphi \) be the random dynamical system generated by the solution of~\eqref{Main_Equation}. Further, assume that \( (Y_{\omega})_{\omega \in \Omega} \) is a stationary solution for \( \varphi \) such that
\begin{align}\label{intCondMET}
	\big(\omega \mapsto \vert Y_{\omega} \vert_{\alpha}\big) \in \bigcap_{p \geq 1} L^{p}(\Omega).
\end{align}
Additionally, suppose that for some $t_0>0$, the linear operator $\psi^{t_0}_{\omega}:E_\alpha\rightarrow E_\alpha$ is compact.
{For \(\lambda \in \R\cup \lbrace-\infty\rbrace \) we define}
\[
F_{\lambda}(\omega) \coloneqq \left\{ x \in E_{\alpha} : \limsup_{t \to \infty} \frac{1}{t} \log |\psi^t_{\omega}(x)|_{\alpha} \leq \lambda \right\}.		\]
Then, on a $\theta_t$-invariant subset of \(\Omega\) having full measure, which is denoted again by $\Omega$, there exists a decreasing sequence \((\lambda_i)_{i\geq 1}\), known as Lyapunov exponents with \(\lambda_i \in [-\infty, \infty)\), such that \(\lim_{i \to \infty} \lambda_i = -\infty\). Moreover, for each \(i\geq 1\), either \(\lambda_i > \lambda_{i+1}\) or \(\lambda_i = \lambda_{i+1} = -\infty\). For every $i\geq 1$ with \(\lambda_i > -\infty\), there exist finite-dimensional subspaces \(H^i_{\omega} \subset E_\alpha\) for \(i \in \mathbb{N}\), with the following properties:
\begin{enumerate}
	\item \emph{(Invariance).}\ \ \(\psi^t_{\omega}(H^i_{\omega}) = H^i_{\theta_t \omega}\) for all \(t \geq 0\).
	\item \emph{(Splitting).}\ \ \(F_{\lambda_1}(\omega) = E_{\alpha}\) and \(H^i_{\omega} \oplus F_{\lambda_{i+1}}(\omega) = F_{\lambda_i}(\omega)\) for each \(i\). In particular for every $i$ we have
	\[
	E_\alpha = \bigoplus_{1 \leq j \leq i} H^j_{\omega} \oplus F_{\lambda_{i+1}}(\omega)
	.	\]
	\item \emph{(Fast Growing Subspace).} \ For each \(h_\omega \in H^j_{\omega}\) we have
	\[
	\lim_{t \to \infty} \frac{1}{t} \log |\psi^t_{\omega}(h)|_{\alpha} = \lambda_j
	\]
	and
	\[
	\lim_{t \to \infty} \frac{1}{t} \log |(\psi^t_{\theta_{-t} \omega})^{-1}(h_\omega)|_{\alpha} = -\lambda_j.
	\]
	\item \emph{(Angle vanishing I).} Let \(\tilde{H}^i_{\omega}\) be a subspace of \(H^{i}_{\omega}\) and let \(h_{\omega}\) be an element in \(H^{i}_{\omega} \setminus \tilde{H}^{i}_{\omega}\). Then, we have the following limits:
	
	\[
	\lim_{t \to \infty} \frac{1}{t} \log d_{E_\alpha}\left(\psi^t_{\omega}(h_{\omega}), \psi^t_{\omega}(\tilde{H}^i_{\omega})\right) = \lambda_i
	\]
	and
	\[
	\lim_{t \to \infty} \frac{1}{t} \log d_{E_\alpha}\left((\psi^t_{\theta_{-t}\omega})^{-1}(h_{\omega}), (\psi^t_{\theta_{-t}\omega})^{-1}(\tilde{H}^i_{\omega})\right) = -\lambda_i.
	\]
	In particular, if $(h^{k}_{\omega} )_{1\leqslant k\leqslant m_i}$ is a basis of $ H^{i}_{\omega}$, then
	\begin{align}
		\begin{split}
			&\lim_{t\rightarrow\infty}\frac{1}{t}\log \operatorname{Vol}_{E_\alpha}\big{(}\psi^{t}_{\omega}(h^{1}_{\omega}),...,\psi^{t}_{\omega}(h^{m_i}_{\omega})\big{)} = m_i\lambda_i \quad \text{and} \\
			&\lim_{t\rightarrow\infty}\frac{1}{t}\log \operatorname{Vol}_{E_\alpha}\big{(}(\psi^{t}_{\theta_{-t}\omega})^{-1}(h^{1}_{\omega}),...,(\psi^{t}_{\theta_{-t}\omega})^{-1}(h^{m_i}_{\omega})\big{)}=-m_{i}\lambda_{i}.
		\end{split}
	\end{align}
	\item \emph{(Angle vanishing II).} Assume that \( \lambda_i > -\infty \) for some \( i \geq 1 \) and set
	\[
	m_k = \dim\bigl(H^{k}_{\omega}\bigr)
	\]
	for each \( 1 \leq k \leq i \). Let
	\[
	m:= \sum_{k = 1}^{i} m_k,
	\]
	and suppose that \( \bigl(h^{j}_\omega\bigr)_{1 \leq j \leq m} \) is a basis for the direct sum \( \bigoplus_{k = 1}^{i} H^k_\omega \). Then
	\[
	\lim_{t \to \infty} \frac{1}{t} \log \operatorname{Vol}_{E_\alpha}\!\bigl(
	\psi^{t}_{\omega}(h^{1}_{\omega}), \ldots, \psi^{t}_{\omega}(h^{m}_{\omega})
	\bigr)
	= \sum_{k = 1}^{i} m_k \lambda_k.
	\]
\end{enumerate}
\end{theorem}
\begin{proof}
For every $t_0>0$ we can construct a discrete time random dynamical system $(\psi^{nt_0}_\omega)_{n\in\N,\omega\in\Omega}$. Due to the bounds~\eqref{intCond1} and~\eqref{intCond2}, $(\psi^{nt_0}_\omega)_{n\in\N,\omega\in\Omega}$ satisfies the integrability conditions of the multiplicative ergodic theorem obtained in~\cite[Theorem 1.21]{GVR23A}, which proves the statement for the discrete time random dynamical system. 
The extension of this result to the continuous time setting, i.e.~for $(\psi^t_\omega)_{t\geq 0,\omega\in\Omega}$ follows by standard arguments, see ~\cite[Theorem 3.3]{LL10} for more details on this procedure.
\end{proof}
\change{
We now state some important consequences of Theorem~\ref{METT} which are essential for the proof of Theorem \ref{METT2}. For their proofs we refer to Appendix \ref{appendix:c}. 
\begin{lemma}\label{AISuia9q}
Consider the setting of Theorem~\ref{METT} and assume that \( \lambda_i > -\infty \) for some \( i \geq 1 \). Let \( h^{1}_{\omega}, \dots, h^{\tilde{p}}_{\omega} \) be nonzero, linearly independent vectors in \( \bigoplus_{1 \leq k \leq i} H^k_{\omega} \). Then the limit 
\begin{align*}
\lim_{t \to \infty} \frac{1}{t} \log \operatorname{Vol}_{E_\alpha}\left( \psi^{t}_{\omega}(h^{1}_{\omega}), \dots, \psi^{t}_{\omega}(h^{\tilde{p}}_{\omega}) \right)
\end{align*}
exists.
\end{lemma}
\begin{proof}
    See \ref{AISuia9q12}.
\end{proof}
\begin{lemma}\label{liminf}
Consider the setting of Theorem~\ref{METT}, let $\tilde{p}>1$ and \(g^{1}_{\omega}, \dots, g^{\tilde{p}}_{\omega}\) be nonzero, measurable\footnote{The measurability means that for all $x \in E_\alpha$ and $1 \leq \tilde{q} \leq \tilde{p}$, the map $\omega \mapsto |x - g^{\tilde{q}}_{\omega}|_{\alpha}$ is measurable.
} and independent vectors in \(E_\alpha\) such that 
\[
\liminf_{t\rightarrow\infty} \frac{1}{t} \log \operatorname{Vol}_{E_\alpha}\big(\psi^{t}_{\omega}(g^{1}_{\omega}), \dots, \psi^{t}_{\omega}(g^{\tilde{p}}_{\omega})\big) > -\infty.
\]
Then, on a set of full measure, the limit
\[
\lim_{t\rightarrow\infty} \frac{1}{t} \log \operatorname{Vol}_{E_\alpha}\big(\psi^{t}_{\omega}(g^{1}_{\omega}), \dots, \psi^{t}_{\omega}(g^{\tilde{p}}_{\omega})\big)
\]
exists and is finite.
\end{lemma}
\begin{proof}
     See \ref{liminf2}.
\end{proof}
}
\subsection{Independence of the Lyapunov exponents on the norm of the interpolation spaces}\label{indep:LY}
Since we are working with a parabolic rough PDE on a family of interpolation spaces, the solution becomes more regular away from zero due to the regularizing effect of the evolution family. More precisely, we have the following statement~\cite[Proposition 5.5]{GH19}. 
\begin{theorem}\label{regularity}
Let $(u,G(\cdot,u_\cdot))\in \cD^\gamma_{\mathbf{X},\alpha}([0,T])$ be the global solution of~\eqref{Main_Equation}. We denote by $M_t:=\sup\limits_{s\in[0,t]}|u_s|_\alpha$~where $0\leq t \leq T$. Then for every $\alpha'>\alpha$ and $0<s<t\leq T$ we get that $(u,G(\cdot,u_\cdot))\in \cD^{\gamma}_{\mathbf{X},\alpha'}([s,t])$ and there exists positive constants $\chi=\chi(\alpha,\gamma,\sigma,\delta)$ and $C(M_t)=C(M_t, F,G, \mathbf{X})$ such that 
\begin{align*}
	\sup\limits_{r\in[s,t]} |u_r|_{\change{\alpha'}} \lesssim s^{-\change{(\alpha'-\alpha)}} \sup\limits_{r\in[0,t]}|u_r|_\alpha + C (M_t) t^\chi.
\end{align*}
\end{theorem}
\begin{remark}\label{REFFD}
Since \( \change{E_{\alpha'}} \subset E_\alpha \), we can use \( E_{\alpha'} \) as a phase space of the corresponding random dynamical system and apply Theorem \ref{METT} to obtain the Lyapunov exponents and the corresponding splitting in \(\change{ E_{\alpha'}} \).
\end{remark}

Since the Lyapunov exponents are deterministic due to the assumed ergodicity of the metric dynamical system, we naturally expect them to be related to the intrinsic properties of the problem and independent of the specific norm we use.~However, since we work with~\eqref{Main_Equation} on a scale of Banach spaces, the Lyapunov exponents could potentially depend on the $E_\alpha$-norm. This is not the case, as we show in this subsection. 

\change{\begin{remark}
\begin{itemize}
    \item [1)] The norm equivalence of Lyapunov exponents for regularizing evolution equations was also established in \cite{BLPS} by complementary techniques. For example, in the context of the 2D Navier-Stokes equation driven by white noise, under suitable assumptions on the invariant measure for the skew-product flow, the Lyapunov exponents exist in Sobolev spaces $H^s$, for certain values of $s$, and do not depend on $s$, see \cite[Theorem E]{BLPS} for more details. 
\item [2)] The main insight here is the usage of Theorem \ref{METT} in order to obtain a similar statement which is applicable to non-autonomous parabolic rough PDEs. 
\end{itemize}
\end{remark}
}

\begin{theorem}\label{METT2}
Assume the same conditions as in Theorem \ref{METT} hold. Let \( (\lambda_{i} )_{i \geq 1} \) be the Lyapunov exponents generated from Theorem \ref{METT} by choosing \( E_{\alpha} \), and let \( m_{i} \) be the corresponding multiplicity of each finite Lyapunov exponent. Let \( ( \tilde{\lambda}_{i} )_{i \geq 1} \) be the Lyapunov exponents generated from Theorem \ref{METT} on \change{\( E_{\alpha'} \) such that $E_{\alpha'}\hookrightarrow E_\alpha$} and let \( \tilde{m}_{i} \) be the corresponding multiplicity of each finite Lyapunov exponent.
Then for every \( \lambda_{i} \) with \( \lambda_i > -\infty \), it holds that \( \lambda_{i} = \tilde{\lambda}_{i} \) and \( m_{i} = \tilde{m}_{i} \).
\end{theorem}
\begin{proof}
Assume that $H_{\omega}^{i}$ is a finite-dimensional space that is obtained from Theorem \ref{METT} by choosing $E_{\alpha}$ as a phase space. 
First, note that for every \(i\), we have \(H^{i}_{\omega} \subset \change{E_{\alpha'}}\). This follows directly from the invariance property in Theorem \ref{METT} combined with Theorem \ref{regularity}. Assume \((h^{k}_{\omega} )_{1\leqslant k\leqslant m_i}\) is a basis of \(H^{i}_{\omega}\) and \(\lambda_{i} \neq -\infty\). Recalling that \(|\cdot|_{\alpha} \lesssim \change{|\cdot|_{\alpha'}}\), we have for every \(t \geq 0\) that
\begin{align*}
	\operatorname{Vol}_{E_{\alpha}}\big{(}\psi^{t}_{\omega}(h^{1}_{\omega}),\dots,\psi^{t}_{\omega}(h^{m_i}_{\omega})\big{)} 
	&\lesssim \operatorname{Vol}_{\change{E_{\alpha'}}}\big{(}\psi^{t}_{\omega}(h^{1}_{\omega}),\dots,\psi^{t}_{\omega}(h^{m_i}_{\omega})\big{)}.
\end{align*}
Consequently,
\begin{align}\label{TYB}
	\liminf_{t \to \infty} \frac{1}{t} \log \operatorname{Vol}_{\change{E_{\alpha'}}}\big{(}\psi^{t}_{\omega}(h^{1}_{\omega}),\dots,\psi^{t}_{\omega}(h^{m_i}_{\omega})\big{)} 
	&\geq m_{i} \lambda_{i}.
\end{align}
Since \(\lambda_{i} \neq -\infty\), Lemma \ref{liminf} yields that the limit
\[
\lim_{t \to \infty} \frac{1}{t} \log \operatorname{Vol}_{\change{E_{\alpha'}}}\big{(}\psi^{t}_{\omega}(h^{1}_{\omega}),\dots,\psi^{t}_{\omega}(h^{m_i}_{\omega})\big{)}
\]
exists.
Now we introduce the space
\begin{align*}
	C^{i}(M) := \left\lbrace \omega \in \Omega : \sup_{h \in H^{i}_{\omega} \setminus \lbrace 0 \rbrace} \frac{|h|_{\change{\alpha'}}}{|h|_{\alpha}} \leq M \right\rbrace.
\end{align*}
which for every $t\geq 0$ can be alternatively written as
\begin{align}\label{GBVAA}
	C^{i}(M) = \left\lbrace \omega \in \Omega : \sup_{h \in H^{i}_{\theta_{-t}\omega} \setminus \lbrace 0 \rbrace} \frac{|\psi^{t}_{\theta_{-t}\omega}(h)|_{\change{\alpha'}}}{|\psi^{t}_{\theta_{-t}\omega}(h)|_{\alpha}} \leq M \right\rbrace.
\end{align}
using the invariance property in Theorem \ref{METT}. Note that \(C^{i}(M)\) is measurable due to the measurability of $\omega\mapsto H^{i}_{\omega}$, which is a finite-dimensional subspace of $E_{\alpha}$. Additionally, since \(H^{i}_{\omega}\) is a finite-dimensional space, we can choose a sufficiently large \(M\) such that \(\mathbb{P}(C^i(M)) > 0\).
Let \(\omega \in C^i(M)\), \(t \geq 0\) and \((h^{j}_{\theta_{-t}\omega})_{1 \leq j \leq m_{i}}\) be an arbitrary basis of \(H^{i}_{\theta_{-t}\omega}\). Then, from \eqref{GBVAA} and the definition of the volume, we have
\begin{align}\label{GVBBN}
	\frac{\operatorname{Vol}_{\change{E_{\alpha'}}}\left( \psi^{t}_{\theta_{-t}\omega}(h^{1}_{\theta_{-t}\omega}), \dots, \psi^{t}_{\theta_{-t}\omega}(h^{m_i}_{\theta_{-t}\omega}) \right)}{\operatorname{Vol}_{E_{\alpha}}\left( \psi^{t}_{\theta_{-t}\omega}(h^{1}_{\theta_{-t}\omega}), \dots, \psi^{t}_{\theta_{-t}\omega}(h^{m_i}_{\theta_{-t}\omega}) \right)} \leq M^{m_i}.
\end{align}
Recalling that \(\mathbb{P}(C^i(M)) > 0\), by Poincaré's recurrence theorem, for a set of full measure, which is again denoted by $\Omega$, we can find a sequence \((n_k)_{k \geq 1}\), which depends on $\omega\in{\Omega}$, with \(n_k \to \infty\) such that \(\theta_{n_k} \omega \in C^i(M)\). Let \(H^{i}_{\omega} := \langle h^{j}_{\omega} \rangle_{1 \leq j \leq m_i}\).
Therefore, replacing $\omega$ by $\theta_{n_k}\omega$ and setting \(t := n_k\), we obtain 
\[
\frac{\operatorname{Vol}_{\change{E_{\alpha'}}}\left( \psi^{n_k}_{\omega}(h^{1}_{\omega}), \dots, \psi^{n_k}_{\omega}(h^{m_i}_{\omega}) \right)}{\operatorname{Vol}_{E_{\alpha}}\left( \psi^{n_k}_{\omega}(h^{1}_{\omega}), \dots, \psi^{n_k}_{\omega}(h^{m_i}_{\omega}) \right)} \leq M^{m_i}.
\]
Therefore we get
\[ \text{Vol}_{\change{E_{\alpha'}}} (\psi^{n_k}_\omega(h^1_\omega), \ldots \psi^{n_k}_\omega(h^{m_i}_\omega) ) \leq M^{m_i} \text{Vol}_{E_\alpha}( \psi^{n_k}_\omega(h^1_\omega),\ldots, \psi^{n_k}(h^{m_i}_\omega) ). \]
Consequently, since \(n_k \to \infty\), we have
\begin{align}\label{TYB2}
	\liminf_{t \to \infty} \frac{1}{t} \log \operatorname{Vol}_{\change{E_{\alpha'}}}\left(\psi^{t}_{\omega}(h^{1}_{\omega}),\dots,\psi^{t}_{\omega}(h^{m_i}_{\omega})\right) \leq m_i \lambda_i.
\end{align}
This together with \eqref{TYB} implies that 
\[
\lim_{t \to \infty} \frac{1}{t} \log \operatorname{Vol}_{\change{E_{\alpha'}}}\left(\psi^{t}_{\omega}(h^{1}_{\omega}),\dots,\psi^{t}_{\omega}(h^{m_i}_{\omega})\right) = m_{i} \lambda_{i}.
\]
This implies that if \( \lambda_{i} > -\infty \) is the Lyapunov exponent obtained from Theorem \ref{METT} using \( E_{\alpha} \) as the phase space, then this value is also one of the Lyapunov exponents obtained from Theorem \ref{METT} by using \( \change{E_{\alpha'}} \). Similarly, we can argue that any finite Lyapunov exponent that arises from Theorem \ref{METT} using \( \change{E_{\alpha'}} \) is equal to \( \lambda_{i} \) for some \( i \geq 1 \). Additionally, from our argument, the multiplicity of the  Lyapunov exponents \( m_{i} \) remains the same.
\end{proof}

We have shown that the Lyapunov exponents are the same using the properties of the fast-growing subspaces $F_\lambda$ entailed by Theorem~\ref{METT}.~However, these spaces are not identical, but the fast-growing subspaces turn out to be independent of the choice of norm. This is established in the next result.  
\begin{theorem}\label{unnns}
Assume the same conditions as in Theorem \ref{METT} hold. Let \(\lambda_{i} > -\infty\) and let \(H^{i}_{\omega}\) and \(\tilde{H}^{i}_{\omega}\) denote the fast-growing spaces corresponding to \(\lambda_{i}\), obtained by considering the Banach spaces \(E_{\alpha}\) and \(E_{\alpha'}\). Then \(\tilde{H}^{i}_{\omega} = H^{i}_{\omega}\).
\end{theorem}
\begin{proof}
The proof relies on the representation of fast-growing spaces $F_\lambda$, which is based on a duality argument. Throughout the proof, \((\tilde{E}^{\star}, | \cdot |_{\tilde{E}}^{\star})\) denotes the dual space of an arbitrary Banach space \((\tilde{E}, | \cdot |_{\tilde{E}})\). We frequently use the fact that for a Banach space \((\tilde{E}, | \cdot |_{\tilde{E}})\) which is continuously embedded in another Banach space \((\tilde{F}, | \cdot |_{\tilde{F}})\), then the dual space \((\tilde{F}^\star, | \cdot |_{\tilde{F}}^{\star})\) is continuously embedded in \((\tilde{E}^\star, | \cdot |_{\tilde{E}}^{\star})\). We further consider the filtrations \(F_{\lambda_{i+1}}(\omega)\), \(F_{\lambda_{i}}(\omega)\), and \(\tilde{F}_{\lambda_{i+1}}(\omega)\), \(\tilde{F}_{\lambda_{i}}(\omega)\) obtained from Theorem \ref{METT} by considering \(E_{\alpha}\) and \(\change{E_{\alpha'}}\), respectively. By definition, for \(j = i, i+1\), we have \(\tilde{F}_{\lambda_{j}}(\omega) \subset F_{\lambda_{j}}(\omega)\).
Furthermore \(F_{\lambda_{i+1}}(\omega) = F_{\lambda_{i}}(\omega) \oplus H^{i}_{\omega}\) and \(\tilde{F}_{\lambda_{i+1}}(\omega) = \tilde{F}_{\lambda_{i}}(\omega) \oplus \tilde{H}^{i}_{\omega}\). We define the following spaces
\begin{align*}
	G_{\lambda_{i+1}}^{\star}(\omega) &:= \left\lbrace h^{\star} \in (F_{\lambda_{i}}(\omega))^{\star} : \limsup_{t \to \infty} \frac{1}{n} \log \left| (\psi^{n}_{\theta_{-n} \omega})^{\star}(h^{\star}) \right|_{\alpha}^{\star} \leq \lambda_{i+1} \right\rbrace,\\
	\tilde{G}_{\lambda_{i+1}}^{\star}(\omega) &:= \left\lbrace \tilde{h}^{\star} \in (\tilde{F}_{\lambda_{i}}(\omega))^{\star} : \limsup_{n \to \infty} \frac{1}{n} \log \left| (\psi^{n}_{\theta_{-n} \omega})^{\star}(\tilde{h}^{\star}) \right|_{\change{\alpha'}}^{\star} \leq \lambda_{i+1} \right\rbrace,
\end{align*}
where $\psi^\star$ denotes the dual of the random dynamical system $\psi$. Recall that \(\change{E_{\alpha'}}\) is continuously embedded in \(E_{\alpha}\).~Thus, from the definitions of \(G_{\lambda_{i+1}}^{\star}(\omega)\) and \(\tilde{G}_{\lambda_{i+1}}^{\star}(\omega)\), we have $G_{\lambda_{i+1}}^{\star}(\omega) \subset \tilde{G}_{\lambda_{i+1}}^{\star}(\omega)$.
From the proof of \cite[Lemma 1.13]{GVR23A} we have the following representation of the fast-growing spaces
\begin{align}\label{HHHH}
	{H}^{i}_{\omega}=\lbrace h\in ((F_{\lambda_{i}}(\omega))^{\star})^{\star}: h|_{{G}_{\lambda_{i+1}}^{\star}(\omega)} =0  \rbrace, \ \ \ \text{and}\ \ \  \tilde{H}^{i}_{\omega}=\lbrace h\in ((\tilde{F}_{\lambda_{i}}(\omega))^{\star})^{\star}: h|_{\tilde{G}_{\lambda_{i+1}}^{\star}(\omega)} =0  \rbrace.
\end{align}
Now, from the inclusions \( G_{\lambda_{i+1}}^{\star}(\omega) \subset \tilde{G}_{\lambda_{i+1}}^{\star}(\omega) \) and \( ((\tilde{F}_{\lambda_{i}}(\omega))^{\star})^{\star} \subset ((F_{\lambda_{i}}(\omega))^{\star})^{\star} \), it follows from \eqref{HHHH} that \( \tilde{H}^{i}_{\omega} \subseteq H^{i}_{\omega} \). Consequently, since they both have the same dimension, they are indeed identical. This completes the proof.
\end{proof}
\begin{remark}\label{CONTYU}
Throughout the proof, we rely on \eqref{HHHH} from which we can immediately infer the claim. Alternatively, one could use the representation in \cite[Corollary 17]{GTQ15} which is applicable for reflexive Banach spaces to prove the result.~Note that for the definitions of \( G_{\lambda_{i+1}}^{\star}(\omega) \) and \( \tilde{G}_{\lambda_{i+1}}^{\star}(\omega) \), we use discrete time because this is sufficient for our aims.~However, it is possible to show that the definitions of \( G_{\lambda_{i+1}}^{\star}(\omega) \) and \( \tilde{G}_{\lambda_{i+1}}^{\star}(\omega) \) can be extended to the continuous time setting. For the convenience of the reader, we shortly sketch this argument.~We recall that \eqref{intCond1} and \eqref{intCond2} hold. For simplicity, we set {\( t_0 = 1 \).} Now, for \( h^{\star} \in G^{\star}_{\lambda_{i+1}(\omega)} \), {which is defined now only for discrete time}  assume that \( t = \lfloor t \rfloor + \{t\} \), where \( \lfloor t \rfloor \in \mathbb{N} \) and \( 0 \leq \{t\} < 1 \). \change{By the cocycle property we have 
 \[\psi^t_\omega=\psi^{[t]+\{t\}}_\omega =\psi^{[t]}_{\theta_{\{t\}\omega}}\circ \psi^{\{t\}}_\omega.\]
    Replacing $\omega$ by $\theta_{-t}\omega$ leads to 
    \[ \psi^t_{\theta_{-t}\omega}=\psi^{[t]}_{\theta_{[t]\omega}} \circ \psi^{{t}}_{\theta_{-t}\omega}.\] }
    Consequently 
$
(\psi^t_{\theta_{-t}\omega})^\star = (\psi^{\{t\}}_{\theta_{-t}\omega})^\star \circ (\psi^{\lfloor t \rfloor}_{\theta_{-\lfloor t \rfloor}\omega})^\star.
$
Thus, choosing \( h^\star \in G^{\star}_{\lambda_{i+1}}(\omega) \) we have that
\begin{align}\label{GBV128}
	\frac{1}{t}\log\left\vert(\psi^{t}_{\theta_{-t}\omega})^{\star}(h^{\star})\right\vert^{\star}_{\alpha}\leq \frac{1}{t} \left(\sup_{0\leq s<1}\log^{+}\left\Vert\left(\psi^{1-s}_{\theta_{s}\circ\theta_{-\lfloor t \rfloor-1}\omega}\right)^{\star}\right\Vert_{\mathcal{L}(E_{\alpha}^{\star};E_{\alpha}^{\star})}+\log\left\vert(\psi^{\lfloor t \rfloor}_{\theta_{-\lfloor t \rfloor}\omega})^{\star}(h^{\star})\right\vert^{\star}_{\alpha}\right).
\end{align}
Recalling the definition of \( G^{\star}_{\lambda_{i+1}(\omega)} \), we conclude that the second term on the right-hand side is bounded from above by $\lambda_{i+1}$. We claim that the first one converges to zero as $t\to\infty,$ which proves the claim. To this aim, we note that
\begin{align*}
	\sup_{0\leq s<1}\log^{+}\left(\left\Vert\left(\psi^{1-s}_{\theta_{s}\omega}\right)^{\star}\right\Vert_{\mathcal{L}(E_{\alpha}^{\star};E_{\alpha}^{\star})}\right)=\sup_{0\leq s<1}\log^{+}\left(\left\Vert\left(\psi^{1-s}_{\theta_{s}\omega}\right)\right\Vert_{\mathcal{L}(E_{\alpha};E_{\alpha})}\right)\in L^{1}(\Omega).
\end{align*}			
Therefore, from Birkhoff's ergodic theorem, we have {almost surely that} 
\[
\lim_{t \to \infty} \frac{1}{t} \sup_{0 \leq s < 1} \log^{+} \left\Vert \left(\psi^{1-s}_{\theta_{s} \circ \theta_{-\lfloor t \rfloor - 1} \omega}\right)^{\star} \right\Vert_{\mathcal{L}(E_{\alpha}^{\star}; E_{\alpha}^{\star})} = 0.
\]  
Now, from \eqref{GBV128}, we conclude that for every \( h^\star \in G^{\star}_{\lambda_{i+1}(\omega)} \), we have on a set of full measure denoted again by \( {\Omega} \) that
\[
\limsup_{t \to \infty} \frac{1}{t} \log \left\vert (\psi^{t}_{\theta_{-t}\omega})^{\star}(h^\star) \right\vert^{\star}_{\alpha} \leq \lambda_{i+1}.
\]  
By similar arguments we obtain an analogous result for \( \tilde{G}_{\lambda_{i+1}}^{\star}(\omega) \).
\end{remark}


\subsection{Invariant manifolds}\label{inv:m}
The multiplicative ergodic theorem together with further sign information on the Lyapunov exponents can be used to infer the existence of invariant manifolds (stable, unstable and center) for the random dynamical system generated by~\eqref{Main_Equation}.~To this aim, we verify the integrability conditions~\eqref{intCond1} and~\eqref{intCond2} of Theorem \ref{METT} using the integrable bounds of the linearization of~\eqref{Main_Equation} along a stationary solution. 
The following statement is similar to the results obtained in~\cite{GVR25,LNZ23} in the autonomous case under different assumptions on the noise, drift, and diffusion coefficients and using different techniques which do not rely on Gronwall's lemma. We focus only on the existence of local stable manifolds.

\begin{theorem}\label{stable_manifold}
Let all the conditions in Theorem \ref{METT} be satisfied, and define \(\lambda^- \coloneqq \sup \{ \lambda_{j} : \lambda_{j} < 0 \}\). Additionally, assume that $G$ is four times Fréchet differentiable. We fix a time step \(t_1\) with \(t_1 >0\). Then, for every \(0 < \nu < -\lambda^-\), there exists a family of immersed submanifolds \(S^{\nu}_{loc}(\omega)\) of \(E_{\alpha}\)  modeled on \(F_{\lambda^-}(\omega)\).\footnote{The local stable manifold $S^\nu_{\text{loc
	}}(\omega)$ contains the trajectories of $\phi$ which decay at an exponential rate in a neighborhood of the stationary solution $Y$. We refer to \cite[ Definition 3.1.1]{AMR88} for more details on this topic.}
	~Moreover, on a set of full measure denoted again by \(\Omega\), the following properties hold for every \(\omega \in {\Omega}\) on \(S^{\nu}_{loc}(\omega)\).
	\begin{enumerate}
\item \emph{(Exponential stability).} For two positive and finite random variables \(\rho_{1,s}^{\nu}\) and \(\rho_{2,s}^{\nu}\) such that
\begin{align}\label{eqn:rho_temp}
	\liminf_{k \to \infty} \frac{1}{k} \log \rho_{i,s}^{\nu}(\theta_{kt_1} \omega) \geq 0, \quad i = 1, 2
\end{align}
the following inclusion holds
\begin{align}\label{invari}
	\begin{split}
		&\left\{ \change{x} \in E_{\alpha} ~\colon~ \sup_{k \geq 0} e^{kt_1 \upsilon} \left| \varphi^{kt_1}_{\omega}(\change{x}) - Y_{\theta_{kt_1} \omega} \right|_{\alpha} < \rho_{1,s}^{\nu}(\omega) \right\} \\
		&\hspace{20pt} \subseteq S^{\nu}_{loc}(\omega) \\
		&\hspace{20pt} \subseteq \left\{ \change{x} \in E_{\alpha} ~\colon~ \sup_{k \geq 0} e^{kt_1 \nu} \left| \varphi^{kt_1}_{\omega}(\change{x}) - Y_{\theta_{kt_1} \omega} \right|_{\alpha} < \rho_{2,s}^{\nu}(\omega) \right\}.
	\end{split}
\end{align}
Moreover, for an initial datum $ \change{x}\in S^{\nu}_{loc}(\omega)$, the corresponding solution $\varphi^{kt_1}_\omega(\change{x})$ exhibits around the stationary point the following exponential decay 
\begin{align}\label{eqn:contr_char}
	\limsup_{k\rightarrow\infty}\frac{1}{k}\log\vert\varphi^{kt_1}_{\omega}(\change{x})-Y_{\theta_{kt_1}\omega}\vert_{\alpha}\leq  t_1\lambda^-.
\end{align}
\item \emph{(Invariance).} We can find a  random variable $K(\omega)$ such that for $ k\geq K(\omega) $ it holds that
\begin{align*}
	\varphi^{kt_1}_{\omega}(S^{\nu}_{loc}(\omega))\subseteq S^{\nu}_{loc}(\theta_{kt_1}\omega).
\end{align*}

\end{enumerate}
\end{theorem}
\begin{proof}
The proof of this result is based on the estimate of the difference between the linearization around a point close to the stationary point, the linearization around the stationary point itself and Corollary~\ref{cor:LinGronwall2}.~We only provide a sketch of the proof emphasizing the importance of Corollary~\ref{cor:LinGronwall2} which allows us to obtain results of this type. For $\change{x} \in E_\alpha$ and a fixed time point $t_1>0$ we define  
\[
H_{\omega}(\change{x}) := \varphi^{t_1}_{\omega}(\change{x} + Y_{\omega}) - \varphi^{t_1}_{\omega}(Y_\omega) - \psi^{t_{1}}_{\omega}(\change{x}).
\]  
This yields for $\change{x}_1,\change{x}_2\in E_\alpha$  that
\begin{align}\label{HHNANS}
\vert H_{\omega}(\change{x}_2)-H_{\omega}(\change{x}_1)\vert_{\alpha}\leq\int_{0}^{1}\left\vert\left(\txtD\varphi^{t_1}_{\omega}(Y_{\omega}+r\change{x}_2+(1-r){\change{x}_1})-\txtD\varphi^{t_1}_{\omega}(Y_{\omega})\right)(\change{x}_2-{\change{x}}_1)\right\vert_{\alpha} \, \mathrm{d}r.
\end{align}
Now, we apply Theorem \ref{thm:IntegrableBounds} and Corollary~\ref{cor:LinGronwall2} to estimate the right-hand side of \eqref{HHNANS}, verifying the assumptions for the existence of local stable manifolds stated in~\cite[Theorem 2.10]{GVR23A} and proving the statement. We refrain from providing further details.
\end{proof}
Since the stable manifold is modeled on \( F_{\lambda^-}(\omega) \), and when all the Lyapunov exponents are negative (which implies \( F_{\lambda^-}(\omega) = E_\alpha \)), we can conclude that, in the neighborhood of the stationary point, all solutions decay exponentially. 
\begin{corollary}\label{stability}
We assume the same setting as in Theorem \ref{stable_manifold} and that \(\lambda^- < 0\). Then, for \(0 \leq \nu < -\lambda^-\), there exists a subset of full measure denoted again by $\Omega$, together with a random variable \(R^{\nu}(\omega)\)  such that \(\liminf_{k \rightarrow \infty}\frac{1}{k} R^{\nu}(\theta_{kt_1} \omega) \geq 0\) and
\begin{align}\label{VAAYYHH}
\left\{\change{x} \in E_\alpha ~\colon~ |\change{x} - Y_{\omega}|_{\alpha} \leq R^{\nu}(\omega)\right\} = S^{\nu}_{\omega}.
\end{align}
Moreover, for every $\omega\in\Omega$ and $\change{x}\in E_\alpha$ with $|\change{x} - Y_{\omega}|_{\alpha} \leq R^{\nu}(\omega)$
\begin{align}\label{WWAAYYHH}
\limsup_{t\rightarrow\infty}\frac{1}{t}\log\vert\varphi^{t}_{\omega}(\change{x})-Y_{\theta_t\omega}\vert_{\alpha}\leq \lambda^-<0.
\end{align}
\end{corollary}
\begin{proof}
The claim \eqref{VAAYYHH} follows from the existence of the stable manifold and the fact that \( F_{\lambda^{-}}(\omega) = E_\alpha \). For a detailed proof, we refer to \cite[Lemma 4.17]{GVR25}. 
For the proof of \eqref{WWAAYYHH}, we first recall that from \eqref{eqn:contr_char} and \eqref{VAAYYHH} we have
\begin{align*}
\limsup_{n \to \infty} \frac{1}{n} \log \left| \varphi^{nt_1}_{\omega}(\change{x}) - Y_{\theta_{nt_1} \omega} \right|_{\alpha} \leq t_1 \lambda^-.
\end{align*}
For \( t = \lfloor \frac{t}{t_1} \rfloor t_1 + s = nt_1 + s \), due to the cocycle property, we have
\begin{align*}
\left| \varphi^{t}_{\omega}(\change{x}) - Y_{\theta_{t} \omega} \right|_{\alpha} = \left| \varphi^{s}_{\theta_{nt_1} \omega} \circ \varphi^{nt_1}_{\omega}(\change{x}) - \varphi^{s}_{\theta_{nt_1} \omega}(Y_{\theta_{nt_1} \omega}) \right|_{\alpha}.
\end{align*}
Then we can argue as in \cite[Remark 4.13]{GVR25} and use Birkhoff's ergodic theorem to conclude \eqref{WWAAYYHH}.
\end{proof}

\begin{remark}
The main focus here is laid on local stable manifolds. Since they are infinite-dimensional, their existence is challenging to obtain and was stated as a conjecture in~\cite{LS11} in the Young regime, i.e.~for $\gamma\in(1/2,1)$. This conjecture was positively answered in~\cite{GVR24,LNZ23}.~In our setting, the main insight is the statement of Corollary~\ref{cor:LinGronwall2} which provides a concise proof for the existence of stable manifolds, simplifying the techniques of~\cite{GVR24,LNZ23}.~By similar arguments, one can obtain unstable and center manifolds based on additional sign information of the Lyapunov exponents.~We refer to \cite[Theorem 2.14]{GVR24} for more details.
\end{remark}
\section{Examples}\label{sec:app}

\subsection{Parabolic rough PDEs with time-dependent coefficients}
We let $\cO$ be an open bounded domain $\cO\subset\R^n$ with smooth boundary and consider the non-autonomous parabolic PDE on $E:=L^p(\cO)$ for $2\leq p<\infty$ given by
\begin{align}\label{rpde2}
\begin{cases}
\txtd u_t= [A(t) u_t +F(t,u_t)]~\txtd t + G(t,u_t)~\txtd{\bf X}_t,\\
u|_{\partial \cO}=0.
\end{cases}
\end{align}
Here \[A(t)=\sum\limits_{i,j=1}^n  \partial_i (a_{ij}(t,x)\partial_j),\]
where the coefficients $a_{ij}\in C^{\rho}([0,T];C(\overline{\cO}))$, $a_{ij}(t,\cdot)\in C^1(\overline{O})$, $\txtD_ka_{ij}\in C([0,T]\times \overline{\cO})$ and $\rho\in(0,1]$. Moreover, we assume the following uniform ellipticity condition
\[ \sum\limits_{i,j=1}^n a_{ij}(t,x)\change{\zeta}_i\change{\zeta}_j \geq \change{c} |\change{\zeta}|^2,~\quad \text{ for every  } x\in \overline{\cO}, t\in [0,T], \change{\zeta}\in\R^n, \] \change{for some constant $c>0$.}
Furthermore we have that $E_1=D(A(t))=W^{2,p}(\cO)\cap W^{1,p}_0(\cO)$ compactly embeds in $L^p(\cO)$ and $E_\alpha=[E,E_1]_{\alpha}=W^{2\alpha,p}_0(\cO)$. In this case, Assumption \ref{ass:A} is fulfilled.
\begin{theorem}
Under the assumptions~\ref{ass:F},\ref{ass:G1}-\ref{ass:G2} and \ref{ass:SymbolSpace}, the solution operator of~\eqref{rpde2} generates a random dynamical system. Moreover, if $F$ additionally satisfies~\ref{ass:DF}, its linearization around a stationary point is a compact random dynamical system satisfying~\eqref{intCond1} and~\eqref{intCond2}.
\end{theorem}
Provided that there exists a stationary solution for~\eqref{rpde2}, the conditions of the multiplicative ergodic theorem, i.e.~Theorem~\ref{METT} are satisfied for this example. For more details on stationary solutions, we refer to Appendix~\ref{appendix:stationary}.
\begin{remark}
The multiplicative ergodic theorem together with the existence of random stable and unstable manifolds for equations of the form~\eqref{rpde2} with non-autonomous random generators and multiplicative linear noise have been investigated in~\cite{cdls}, whereas the well-posedness of SPDEs of the form~\eqref{rpde2} driven by the Brownian motion in Banach spaces has been investigated in~\cite{Ver10}. 
\end{remark}

\subsection{PDEs with multiplicative rough boundary noise}\label{sec:b:noise}
We provide another example, where the noise acts on the boundary of a domain.  
We let $\cO\subset \R^n$ be an open bounded domain with $C^\infty$-boundary and consider the semilinear parabolic evolution equation with multiplicative rough boundary noise in $E:=L^2(\cO)$ given by
\begin{align}\label{eq:BoundNoise}
\begin{cases}
\frac{\partial}{\partial t} u_t = \cA u_t & \text{ in } \cO,\\
\cC u_t = G(t,u_t)~\frac{\txtd}{\txtd t} \mathbf{X}_t & \text{ on } \partial \cO,\\
u(0)=u_0. 
\end{cases}
\end{align}
To keep the analysis as simple as possible, we work in $L^2(\cO)$ although it is possible to treat~\eqref{eq:BoundNoise} in $L^p(\cO)$. Here, $\mathbf{X}$ is a $\gamma$-H\"older rough path which satisfies Assumption \ref{ass:Noise} with $\gamma\in(\frac{1}{3},\frac{1}{2}]$ and $G$ a time-dependent nonlinearity. Furthermore, $\cA$ is a formal second-order differential operator in divergence form with corresponding Neumann boundary conditions $\cC$ given by 
\begin{align*}
\cA u:= \sum_{i,j=1}^n \partial_i \left(a_{ij}\partial_j \right)u-\lambda_A u,\quad \cC u:= \sum_{i,j=1}^n \nu_i\gamma_\partial a_{ij}\partial _ju,
\end{align*}
where $\nu$ is the outer normal vector, $\gamma_\partial$ the trace, $\lambda_A>0$ a constant and $(a_{ij})_{i,j=1}^n$ smooth coefficients such that there exists some constant $k>0$ with
\begin{align*}
\sum_{i,j=1}^n a_{ij}(x)\change{\zeta}_i \change{\zeta}_j\geq k |\change{\zeta}|^2,
\end{align*}
for all $\change{\zeta} \in \R^n$ and $x\in \overline{\cO}$.~We further define the $E$-realization of $(\cA,\cC)$ by $A:D(A)\subset E \to E$  with $D(A):=\{u\in H^{2}(\cO)~:~\cC u=0\}$ and $(E_\alpha)_{\alpha\in \R}$ the respective fractional power scale, which is given by
\begin{align*}
E_{\frac{\alpha}{2}}\coloneqq
\begin{cases}
\{ u\in H^{\alpha}(\cO) : \cC u =0 \},& \alpha>1+\frac{1}{2}\\
H^{\alpha}(\cO) ,& -\frac{1}{2}<\alpha<\frac{3}{2},\\ 
\left(H^{-\alpha}(\cO)\right)^{\change{\prime}}, & -\frac{3}{2}<\alpha\leq-\frac{1}{2}\\
\{ u\in H^{-\alpha}(\cO) : \cC u =0 \}^\prime, & \alpha<-\frac{3}{2},
\end{cases}
\end{align*}
see for example \cite[Theorem 7.1]{Amann93}. In this case, it is possible to verify \ref{ass:A1}-\ref{ass:A2} for $A$, \ref{ass:A3} holds trivially. Let $(S_t)_{t\geq 0}$ be the analytic semigroup generated by $A$, which is exponential stable \begin{align}\label{expStable}
\|S_t\|_{\mathcal{L}(E_0)}\leq C_S e^{-\lambda_A t}.
\end{align} 
This assumption was also made in \cite[Theorem 4.2]{BS24} for the study of attractors for~\eqref{eq:BoundNoise}.
\begin{remark}     We choose $\mathcal{A}$ to be time-independent, since a time-dependent operator $\cA(t)$ does not satisfy Assumption \ref{ass:A1}. Note that the domain $D(A(t)):=\{u\in H^{2}(\cO)~:~\cC(t) u=0\}$ of a time-dependent operator $A(t)$ is also time-dependent due to the boundary operator $\cC(t)$. This would require a notion of controlled rough paths according to a time-dependent family of interpolation spaces $E_\alpha$ which goes beyond the scope of this paper and will be pursued in future works.~We refer to~\cite{SV11} for the well-posedness of~\eqref{eq:BoundNoise} in the non-autonomous case  $(\cA(t),\cC(t))$ where the boundary noise is given by a Brownian motion.
\end{remark}
To treat the boundary data, we introduce a second Banach scale $\widetilde{E}_{\alpha}\coloneqq H^{\alpha}(\partial\cO)$ and define the Neumann operator $N\in \cL(\widetilde{E}_\alpha;E_\varepsilon)$ for some $\varepsilon<\frac{3}{4}$ and $\alpha>\frac{3}{2}$ as the solution operator to 
\begin{align*}
\cA u &= 0\quad \text{in}\quad \cO,\\
\cC u &= g\quad \text{on}\quad \partial\cO.
\end{align*}
\change{For more information on boundary value problems of this form, see for example \cite[Section 9]{Amann93}.} Because the diffusion coefficient now influences the boundary, we have to modify the conditions on $G$. For a better comprehension, we restrict ourselves to one-dimensional noise in this example. \change{The extension to multidimensional noise can be made componentwise as in the previous sections.}
\begin{itemize}
\item[\textbf{(}$\mathbf{\tilde{G}}$\textbf{)}\namedlabel{ass:Gbound}{\textbf{(}$\mathbf{\tilde{G}}$\textbf{)}}]  There exists a $\sigma>\eta+1+\frac{1}{2}$ such that for any $i=0,1,2$ the diffusion coefficient
\begin{align*}
G:[0,T]\times E_{-\eta-i\gamma} \to \widetilde{E}_{-\eta-i\gamma+\sigma}
\end{align*}
fulfills \ref{ass:G1}-\ref{ass:G2} and the Fréchet derivative of
\begin{align*}
\txtD_2G(t,\cdot)\circ A_{-\eta-\gamma}NG(t,\cdot):E_{-\eta-\gamma}\to \widetilde{E}_{-\eta-\gamma-\sigma} 
\end{align*}
is bounded. Furthermore, there exists a function $k_G:[0,\infty)\to \R$ with $k_G(s)\to 0$ for $s\searrow0$ such that \eqref{ass:GSymbol} is fulfilled.
\end{itemize}
Here $\eta:=1-\varepsilon$ and $A_{-\eta-\gamma}\in \cL(E_{1-\eta-\gamma};E_{-\eta-\gamma})$ is the unique closure of $A$ in $E_{-\eta-\gamma}$, called the extrapolated operator of $A$.~For detailed information on extrapolation operators, we refer to \cite[Chapter V]{Amann1995}.
\begin{theorem}\label{thm:GlobalSolBoundEq}
Assume that \ref{ass:A1}-\ref{ass:A2}, \ref{ass:Noise} and \ref{ass:Gbound} are fulfilled. Then \eqref{eq:BoundNoise}  can be rewritten as the semilinear problem
\begin{align}\label{eq:TransformedEQ}
\begin{split}
	\begin{cases}
		\txtd u_t = A u_t~\txtd t+ A_{-\eta-\gamma}NG(t, u_t)~\txtd \mathbf{X}_t,\\
		u(0)=u_0\in E_{-\eta}.
	\end{cases}
\end{split}
\end{align}
Furthermore, \eqref{eq:TransformedEQ} has the global solution $(u,u^\prime)\in \cD^\gamma_{\mathbf{X},-\eta}$ where $u^\prime_t=A_{-\eta-\gamma}NG(t, u_t)$ and
\begin{align}\label{eq:MildSolution}
u_t=S_{t}u_0+\int_0^t S_{t-r} A_{-\eta-\gamma}NG(r, u_r)~\textnormal{d} \mathbf{X}_r.
\end{align}
\end{theorem}
\begin{proof}
The key argument for this transformation is based on the fact that $NG(\cdot,y)$ is not in the domain of $A$ due to the definition of the Neumann operator. Therefore, one has to consider $A_{-\eta-\gamma}$ as an extension of $A$. The proof follows the same strategy as in \cite[Theorem 3.20]{NS23} and applying Theorem~\ref{ex:nona} for the local well-posedness, respectively Theorem~\ref{thm:GlobEx} for the global well-posedness.
\end{proof}

\begin{example}
We mention a similar example to~\cite[Example 5.2]{NS23} for $G$ that fulfills the Assumptions \ref{ass:G1}-\ref{ass:G2} in the $L^2$-setting. Note that the diffusion coefficient $G$ must increase the spatial regularity in order to subsequently take the trace.~One typical operator which increases spatial regularity is given by
\begin{align*}
\Lambda^{\beta_2-\beta_1}:H^{\beta_1}(\R^n)\to H^{\beta_2}(\R^n):f\mapsto \mathscr{F}^{-1}(1+|\cdot|^2)^{\frac{\beta_2-\beta_1}{2}}\mathscr{F}f,
\end{align*}
where $\beta_1,\beta_2\in\R$ and $\mathscr{F}$ denotes the Fourier transform. To extend this to an open bounded domain $\cO$, instead of the whole space $\R^n$, we use a retraction $e_\cO:H^{\beta_1}(\R^n)\to H^{\beta_1}(\cO)$ and a coretraction $r_\cO:H^{\beta_1}(\cO)\to H^{\beta_1}(\R^n)$, see \cite[Theorem 4.2.2]{Tri78}. An example of a diffusion coefficient is given by $G(t,u)\coloneqq a(t)\cdot \gamma_{\partial}r_{\cO}\Lambda^{\beta_2-\beta_1}e_{\cO}$ for suitable values of $\beta_1,\beta_2\in \R$ and $a\in C^{2\gamma}([0,T];\R)$.
\end{example}
Now we prove the existence of Lyapunov exponents for the transformed equation \eqref{eq:TransformedEQ}. Recall, that $\mathbf{X}$ is a rough cocycle, as in Definition \ref{def:RPCocycle}, that $\Omega=\widetilde{\Omega}\times \Sigma$ is the extended probability space, and $\widetilde{\Omega}$ the probability space associated to $\mathbf{X}(\tilde{\omega})$. Similar to Section \ref{sec:LinDynamics}, we consider the linearized rough PDE along the path component $u$ given by
\begin{align}\label{eq:LinBoundNoise}
\begin{split}
\begin{cases}
	\txtd v_t = A v_t~\txtd t+ A_{-\eta-\gamma}N\txtD_2 G(t,u_t)v_t~\txtd \mathbf{X}_t(\tilde{\omega}),\\
	v(0)=v_0.
\end{cases}
\end{split}
\end{align}
The solution operator of the linearization generates a random dynamical system $\psi$. In order to deduce the existence of Lyapunov exponents using the multiplicative ergodic theorem, we have to show that $\psi$ is compact.
\begin{lemma}
Assume that all conditions of Theorem \ref{thm:GlobalSolBoundEq} are satisfied and that $A$ has a compact resolvent. Then $\psi$ is a linear, compact random dynamical system.
\end{lemma}
\begin{proof}
Since $A$ has compact resolvent we conclude that the embeddings $E_\beta\hookrightarrow E_\alpha$ are compact for $\beta>\alpha$, \cite[V.1.5.1]{Amann1995}. Then the claim follows using the smoothing properties of the semigroup and compactness of the embeddings $E_{\alpha+\varepsilon}\hookrightarrow E_\alpha$ for $\varepsilon>0$, as in Lemma \ref{lem:ExRDS}.
\end{proof}
In order to apply Theorem \ref{METT}, we have to linearize~\eqref{eq:TransformedEQ} along a stationary solution. The existence of such a solution will be discussed in Appendix~\ref{appendix:stationary}. Finally, we summarize the above considerations in the next theorem. 
\begin{theorem}
Under the assumptions of Theorem~\ref{thm:GlobalSolBoundEq}, there exists Lyapunov exponents $(\lambda_i)_{i\geq 1}$ for~\eqref{eq:BoundNoise}.
\end{theorem}
\begin{proof}
The statement directly follows from Theorem \ref{METT} applied to the dynamical system obtained given by the linearization of \eqref{eq:LinBoundNoise} along a stationary solution. 
\end{proof}
\begin{remark}
One could also obtain the existence of a local stable manifold for~\eqref{eq:BoundNoise} under the assumptions of Theorem~\ref{stable_manifold}, additionally assuming that  $G$ is four times Fréchet differentiable, see Subsection \ref{inv:m}.
\end{remark}
\appendix
\section{Stationary solutions for SPDEs with boundary noise}\label{appendix:stationary}
We establish a stationary solution for \eqref{eq:TransformedEQ}, where $\mathbf{X}\coloneqq \mathbf{B}=(B,\mathbb{B}^{\textnormal{It\^o}})$ is the Itô Brownian rough path, which satisfies assumption~\ref{ass:Noise}, see Subsection \ref{cm}. In the context of SPDEs with additive boundary fractional noise, the existence of a limiting measure was proven in~\cite[Proposition 5.1]{DPDM02}.\\ 
It is known that the stationary solution of  the linear SPDE 
\begin{align*}
\txtd Z_t = A Z_t~\txtd t+ \txtd B_t
\end{align*}
is given by the stationary Ornstein-Uhlenbeck process
\begin{align*}
Z_t=\int_{-\infty}^t S_{t-r}~\txtd B_r.
\end{align*}
Consequently, we would expect that a stationary solution of \eqref{eq:TransformedEQ} has the form
\begin{align*}
y_t=\int_{-\infty}^t S_{t-r}A_{-\eta-\gamma}NG(r,y_r)~\txtd \mathbf{B}_r.
\end{align*}
To prove this, we first show that the rough convolution coincides with the stochastic convolution defined in the It\^o sense. In the finite-dimensional case, this was shown in \cite[Proposition 5.1]{FH20} and in the infinite-dimensional setting in \cite[Proposition 4.8]{GH19}. 
\begin{lemma}\label{lem:ItoRPInt}
Let $(y,y^\prime)\in\mathcal{D}^\gamma_{\mathbf{B},-\eta}([0,\infty))$ be a controlled rough path where $(B_t)_{t\geq 0}$ is a Brownian motion on the filtered probability space $(\Omega, \cF,\P,(\mathcal{F}_t)_{t\geq 0})$ and consider the Itô lift $\mathbf{B}:=(B,\mathbb{B}^{\textnormal{It\^o}})$ such that $\mathbf{B}\in \cC^\gamma$ a.s. Further, assume that there exists for every $M>0$ a time $T_M>0$ such that $|y_t|_{-\eta}+|y_t^\prime|_{-\eta-\gamma}\leq M$ holds for $t\leq T_M$ and that $t\mapsto A_{-\eta-\gamma}NG(t,y(t))$ is adapted to $(\mathcal{F}_t)_{t\geq 0}$. Then
\begin{align*}
\int_0^{t}S_{t-r} A_{-\eta-\gamma}NG(r,y_r)~\txtd \mathbf{B}_r= \int_{0}^{t} S_{t-r} A_{-\eta-\gamma}NG(r,y_r)~\txtd B_r,
\end{align*}
holds almost surely.
\end{lemma}
\begin{proof}
We can show that $z_t(\omega)\coloneqq A_{-\eta-\gamma}NG(t,y_t(\omega))$ together with $$z_t^\prime(\omega)\coloneqq A_{-\eta-\gamma}ND_2G(t,y_t(\omega))G(t,y_t(\omega))$$ is a controlled rough path $(z(\omega),z'(\omega))\in \cD^\gamma_{B(\omega),-\eta}$ for almost every $\omega\in\Omega$. The proof is similar to the autonomous case \cite[Corollary 3.15]{NS23} together with Lemma~\ref{lem:NonAutoCRP}. The claim follows then from \cite[Prop. 4.8]{GH19}.
\end{proof}
\begin{remark}
The same statement as in Lemma \ref{lem:ItoRPInt} holds also, if we consider the Stratonovich lift $(B,\mathbb{B}^{\textnormal{Strat}})$ of the Brownian motion. Likewise, all the following statements remain true if we consider $(B,\mathbb{B}^{\textnormal{Strat}})$ instead of $(B,\mathbb{B}^{\textnormal{It\^o}})$.
\end{remark}
We now show the existence of a stationary solution to \eqref{eq:TransformedEQ}. For this, let $(B_t)_{t\in \R}$ be a two-sided Brownian motion, which is adapted to the two-parameter filtration $(\cF_s^t)_{s\leq t}$ and set $\cF^t_{-\infty}\coloneqq \sigma\big(\bigcup_{s<t} \cF_s^t\big)$.
\begin{lemma}\label{lem:StatSol}
We assume that~\ref{ass:Gbound} together with the condition $\frac{C_S C_G}{\sqrt{2\lambda_A}} \|A_{-\eta-\gamma}\|_{\cL(E_\varepsilon;E_{-\eta})} \|N\|_{\cL(E_{-\eta};E_\varepsilon)}<1$ hold. Then there exists a stochastic process $y:\R\times\Omega\to E_{-\eta}$ adapted to $(\cF^t_{-\infty})_{t\in \R}$ given by
\begin{align*}
y_t=\int_{-\infty}^t S_{t-r}A_{-\eta-\gamma}NG(r,y_r)~\txtd B_r.
\end{align*}
\end{lemma}

\begin{proof}
For $t\in \R$ we define the map $\Gamma:\Lambda\to \Lambda$
\begin{align*}
\Gamma(y)(t)\coloneqq \int_{-\infty}^t S_{t-r}A_{-\eta-\gamma}NG(r,y_{r})~\txtd B_r,
\end{align*}
where
\begin{align*}
y\in\Lambda\coloneqq \left\{y\colon \R\times \Omega \to E_{-\eta}~:~y\ \textnormal{is continuous}, (\cF^t_{-\infty})_{t\in \R} \ \textnormal{adapted and}\ \sup_{t\in\R}\E[|y_t|^2_{-\eta}]^{\frac{1}{2}}<\infty] \right\}.
\end{align*}
Now we show that $\Gamma$ is well-defined and is a contraction on $\Lambda$. Due to It\^o's isometry,  \ref{ass:Gbound} and \eqref{expStable} we have
\begin{align*}
\E[|\Gamma(y)(t)|^2_{-\eta}]\leq \E\Big[ \int_{-\infty}^t |S_{t-r}A_{-\eta-\gamma}NG(r,y_{r})|_{-\eta}^2~\txtd r\Big]\lesssim \int_{-\infty}^t e^{-2(t-r)\lambda_A}~\txtd r = \int_{-\infty}^0 e^{2\lambda_A r}~\txtd r,
\end{align*}
meaning that $\Gamma(y)\in \Lambda$ for $y\in \Lambda$. In addition, we obtain for $y,\tilde{y}\in \Lambda$ that
\begin{align*}
\E[|\Gamma(y)(t)-\Gamma(\tilde{y})(t)|^2_{-\eta}]&\leq \E\Big[ \int_{-\infty}^t |S_{t-r}A_{-\eta-\gamma}N\big(G(r,y_{r})-G(r,\tilde{y}_{r})\big)|_{-\eta}^2~\txtd r\Big]\\
&\leq \frac{C_S^2 C_G^2}{2\lambda_A} \|A_{-\eta-\gamma}\|^2_{\cL(E_\varepsilon;E_{-\eta})} \|N\|^2_{\cL(E_{-\eta};E_\varepsilon)} \sup_{r\in \R} \E[|y_r-\tilde{y}_r|_{-\eta}^2].
\end{align*}
Applying Banach's fixed point theorem, we infer that there exists a $y\in\Lambda$ such that $\Gamma(y)=y$.
\end{proof}
It only remains to show that $(Y_\omega)_{\omega\in \Omega}$, defined by $Y_\omega\coloneqq y_0(\omega)$, satisfies the integrability condition \eqref{intCondMET}, where $y$ is the fixed point derived in Lemma \ref{lem:StatSol}.
\begin{lemma}
The random variable $(Y_\omega)_{\omega\in \Omega}$  is stationary with respect to the random dynamical system $\varphi$ generated by the solution of \eqref{eq:TransformedEQ} and fulfills 
\begin{align*}
(\omega\mapsto |Y_\omega|_{-\eta})\in \bigcap_{p\geq 1} L^p(\Omega).
\end{align*}
\end{lemma}
\begin{proof}
It is easy to see that $Y$ fulfills $\varphi^t_\omega(Y_\omega)=Y_{\theta_t \omega}$, which means that $Y$ is a stationary solution of~\eqref{eq:TransformedEQ}. Furthermore, we have 
\begin{align*}
y_t-y_s =\int_{-\infty}^s S_{s-r} (S_{t-s}-\text{Id}) A_{-\eta-\gamma} N G(r,y_r)~\txtd B_r + \int_s^t S_{t-r} A_{-\eta-\gamma} N G(r,y_r)~\txtd B_r,
\end{align*}
for $s\leq t$.
Using again It\^o's isometry and \ref{ass:Gbound} we obtain
\begin{align*}
\E[|y_t-y_s|_{-\eta}^{2m}]\lesssim (t-s)^{m},
\end{align*}
for $m\in \N$ and $s\leq t$. The exponential stability of the semigroup assumed in~\eqref{expStable} further leads to $\E[|y_0|_{-\eta}]]<\infty$.~Therefore, Kolmogorov's continuity theorem \cite[Theorem 1.8.1]{Kun19} entails that $y_0\in L^m(\Omega;E_{-\eta})$ for all $m\in \N$, which proves the claim.
\end{proof}
\section{Translation compact functions}\label{AppendixA}
Here we give further information on the hull of a function and \change{translation compact functions}. In particular, we focus on stating conditions for the compactness of the hull such that Assumption \ref{ass:SymbolSpace} is satisfied. For further information and detailed proofs, see for example, \cite[Chapter V]{CV02} and \cite[Section 6]{CL17}. We recall that $\mathcal{X}$ is a Hausdorff topological function space.
\begin{definition}
A function $g\in \mathcal{X}$ is called translation compact if $\mathcal{H}(g)$ is compact.
\end{definition}
The easiest way to obtain such translation compact functions is to consider periodic functions. \change{Periodicity is a common assumption for time-dependent equations see for example \cite{MShen1}.}
\begin{example}{\rm \change{(\cite[Example IV.1.1]{CV02})}}
Take $\mathcal{X}=C_b(\R;\R)$ and assume that $g\in C_b(\R;\R)$ is periodic with period $T$. Then it can be shown, by using Arzelà-Ascoli, that $\mathcal{H}(g)=\{g(t+\cdot)~\colon~t\in [0,T]\}$ is compact. There are also generalizations of the periodicity, such as almost \change{\cite[Example 1.2]{CV02}} or quasi-periodic functions \change{\cite[Section V.1]{CV02}} on $C_b(\R;\R)$, which also deliver compactness of the hull. 
\end{example}
Some other sufficient and necessary conditions for translation compactness of a function hardly depend on the choice of $\mathcal{X}$. We will mention here three special cases, which we can use in our setting of semilinear parabolic evolution equations.
\begin{proposition}{\em (\cite[Proposition 2.2, 3.3, 4.1]{CV02})}\label{prop:trcConditions}
\begin{itemize}
\item[i)] Let $(\mathcal{M}, d_{\mathcal{M}})$ be a complete metric space and define $\mathcal{X}:=C(\R;\mathcal{M})$. Then a function $g\in \mathcal{X}$ is translation compact if and only if $g$ is uniformly continuous, such that there exists a positive function $k_g$ with $k_g(s)\to 0$ for $s\searrow 0$ and
\begin{align*}
	d_{\mathcal{M}}(g(t),g(s))\leq k_g(|t-s|),
\end{align*}
for all $t,s\in \R$.
\item[ii)] Let $(\mathcal{M},|\cdot|_{\mathcal{M}})$ be a Banach space, $p\geq 1$ and define $\mathcal{X}:=L^p_{\textrm{loc}}(\R;\mathcal{M})$, which is the space of locally $L^p$-integrable functions. Then a function $g\in \mathcal{X}$ is translation compact if and only if there exists a function $k_g$ such that $k_g(s)\to 0$ for $s\searrow 0$ and
\begin{align*}
	\int_{t}^{t+1} |g(s)-g(s+t)|^p_{\mathcal{M}}~\txtd s\leq k_g(|t|),
\end{align*}
for all $t\in \R$.
\item[iii)] Let $(\mathcal{M},|\cdot|_{\mathcal{M}})$ be a reflexive Banach space, $p\geq 1$ and define $\mathcal{X}:=L^p_{\textrm{loc},w}(\R;\mathcal{M})$, which is the space $L^p_{\textrm{loc}}(\R;\mathcal{M})$ endowed with the local weak convergence topology. Then a function $g\in \mathcal{X}$ is translation compact if and only if $g$ is translation bounded in $L^p_{\textrm{loc}}(\R;\mathcal{M})$, which means
\begin{align*}
	\sup_{t\in \R} \int_t^{t+1} |g(s)|_{\mathcal{M}}^p~\txtd s<\infty
\end{align*}
\end{itemize}
\end{proposition}
In all three situations, the hull $\mathcal{H}(g)$ is a compact Polish space. \change{It is easy to see that if $\mathcal X$ is a product space, it is enough to treat every component separately.}
\begin{lemma}\label{cor:ProductOfSymbols}
Let $(\mathcal{X}_i)_{i=1}^k$ be a collection of Hausdorff topological spaces and $(g_i)_{i=1}^{k}$ such that $g_i\in \mathcal{X}_i$ is translation compact. Then $g=(g_1,\ldots, g_k)\in \mathcal{X}:=\prod_{i=1}^k \mathcal{X}_i$ is translation compact and in particular $\mathcal{H}(g)$ is compact.
\end{lemma}
\begin{example}
Consider now explicitly the situation in \eqref{Main_Equation}. We give  \change{assumptions on the time-dependent data such that} \ref{ass:SymbolSpace} is fulfilled, but note that this is not the only possible option. Due to Corollary \ref{cor:ProductOfSymbols}, it is enough to consider each component of the time symbol separately. For the first component $\xi$ define $\mathcal{X}_1:=L^p_{\textrm{loc},w}(\R;\R)$ for some $p\geq 1$. Then due to Proposition \ref{prop:trcConditions} iii) $\xi$ is translation compact if
\begin{align*}
\sup_{t\in \R} \int_t^{t+1} |\xi(s)|^p~\txtd s<\infty,
\end{align*}
which is for example fulfilled if $\xi$ is periodic.

The second component of the time symbol is the drift term $F$. Define the space $\mathcal{M}_2$ as the set of all continuous functions $f:E_{\alpha}\to E_{\alpha-\delta}$ such that             \begin{align}
|f|_{\mathcal{M}_2}:= \sup_{x\in E_{\alpha}} \frac{|f(x)|_{\alpha-\delta}}{1+|x|_\alpha}
\end{align}
is finite. Then $(\mathcal{M}_2, |\cdot|_{\mathcal{M}_2})$ is a Banach space \cite[Remark 2.10]{CV02} and we can define $\mathcal{X}_2:=L^p_{\textrm{loc},w}(\R;\mathcal{M}_2)$. Note that Assumption \ref{ass:F} implies
\begin{align*}
\sup_{t\in \R} \int_t^{t+1} \left(\sup_{x\in E_{\alpha}} \frac{|F(s,x)|_{\alpha-\delta}}{1+|x|_\alpha}\right)^p~\txtd s\leq C_F^p<\infty.
\end{align*}
The last component, the diffusion coefficient $G$, can be treated similarly. Define $\mathcal{M}_3$ as the space of three times Fréchet differentiable functions $g:E_{\alpha}\to E_{\alpha-\sigma}$ such that
\begin{align*}
|g|_{\mathcal{M}_3}:=\sup_{x\in {E_\alpha}}|g(x)|_{\alpha-\sigma}+\sup_{x\in {E_\alpha}}|\txtD g(x)|_{\cL(E_\alpha;E_{\alpha-\sigma})}+\sup_{x\in {E_\alpha}}|\txtD^2g(x)|_{\cL(E_\alpha^2;E_{\alpha-\sigma})}<\infty.
\end{align*}
Then $(\mathcal{M}_3, |\cdot|_{\mathcal{M}_3})$ is a Banach space and we can define $\mathcal{X}_3=C(\R;\mathcal{M}_3)$. \change{Assuming that $G$ satisfies \ref{ass:G1}-\ref{ass:G2}, we know in particular that $t\mapsto G(t,\cdot)$ and its derivatives are Hölder continuous. Therefore, we define $k_G(s)\coloneqq s^{2\gamma}$, which leads to $k_G(s)\to 0$ for $s\searrow 0$ and
\begin{align}\label{ass:GSymbol}
	|G(t,x)-G(s,x)|_{\mathcal{M}_3}\lesssim k_G(|t-s|).
	\end{align}}
	Then Assumption \ref{ass:SymbolSpace} is satisfied due to Proposition \ref{prop:trcConditions} and Corollary \ref{cor:ProductOfSymbols}.
\end{example}
\section{Consequences of Theorem \ref{METT}}\label{appendix:c}
We provide the proofs of Lemmas \ref{AISuia9q} and \ref{liminf}. To this aim we first state some auxiliary results.
\begin{lemma}\label{AISKAsw}
Consider the setting of Theorem~\ref{METT} and assume that \( \lambda_i > -\infty \) for some \( i \geq 1 \). For each \( 1 \leq k \leq i \), let \( \bigl(h^{k,j}_{\omega}\bigr)_{1 \leq j \leq m_k} \) be a family of linearly independent vectors such that the Lyapunov exponent associated to each \( h^{k,j}_{\omega} \) is equal to \( \lambda_k \).
Assume further that the collection of vectors
\[
\bigl(h^{k,j}_{\omega}\bigr)_{\substack{1 \leq k \leq i\\ 1 \leq j \leq m_k}}
\]
is linearly independent and thus forms a basis for \( \bigoplus_{1\leq k\leq i}H^k_{\omega} \).
Fix an element \( h_{\omega}^{k_0, j_0} \) for some \( 1 \leq k_0 \leq i \) and \( 1 \leq j_0 \leq m_{k_0} \). Let \( \tilde{R}_{\omega}^{k_0, j_0} \) be an arbitrary subspace of
\[
R_{\omega}^{k_0, j_0} := \bigl\langle h_{\omega}^{k,j} \bigr\rangle_{\substack{1 \leq k \leq i,\, 1 \leq j \leq m_k \\ (k,j) \neq (k_0, j_0)}},
\]
which is the span of all vectors in the collection excluding \( h_{\omega}^{k_0, j_0} \).
Then
\[
\lim_{t \to \infty} \frac{1}{t}\log d_{E_\alpha} \!\bigl( \psi_{\omega}^t(h_{\omega}^{k_0, j_0}),\, \psi_{\omega}^t(\tilde{R}_{\omega}^{k_0,j_0}) \bigr) = \lambda_{k_0}.
\]
\end{lemma}
\begin{proof}
First observe that
\begin{align}\label{eq:norm-growth}
\lim_{t \to \infty} \frac{1}{t} \log \ \!\bigl|\psi^t_{\omega}(h_{\omega}^{k_0, j_0})\bigr|_{\alpha}
= \lambda_{k_0}
\end{align}
by the definition of Lyapunov exponents.
Now note that for any subspace \(R\) of \(E_{\alpha}\) we have
\[
\frac{1}{t} \log \bigl|\psi_{\omega}^t(h_{\omega}^{k_0, j_0})\bigr|_{\alpha}
\geq \frac{1}{t} \log d_{E_\alpha} \bigl( \psi_{\omega}^t(h_{\omega}^{k_0, j_0}),\, \psi_{\omega}^t(R) \bigr),
\]
since the distance from a vector to a subspace cannot exceed the norm of the vector.
Therefore, since
\begin{align*}
\sum_{\substack{1 \leq k \leq i, \\ 1 \leq j \leq m_k}}
\frac{1}{t} \log \bigl|\psi_{\omega}^t(h_{\omega}^{k, j})\bigr|_{\alpha}
= \sum_{1 \leq k \leq i} m_k \lambda_k,
\end{align*}
it follows from Lemma~\ref{permutation} and the Angle Vanishing II property in Theorem~\ref{METT} that
\begin{align}\label{eq:angle-vanishing}
\lim_{t \to \infty} \frac{1}{t} \log d_{E_\alpha}
\!\bigl( \psi_{\omega}^t(h_{\omega}^{k_0, j_0}),
\psi_{\omega}^t({R}_{\omega}^{k_0, j_0}) \bigr)
= \lambda_{k_0}.
\end{align}
Finally, since
\begin{align*}
\frac{1}{t} \log d_{E_\alpha} \!\bigl( \psi_{\omega}^t(h_{\omega}^{k_0, j_0}),
\psi_{\omega}^t({R}_{\omega}^{k_0, j_0}) \bigr)
&\leq \frac{1}{t} \log d_{E_\alpha} \!\bigl( \psi_{\omega}^t(h_{\omega}^{k_0, j_0}),
\psi_{\omega}^t(\tilde{R}_{\omega}^{k_0, j_0}) \bigr) \\
&\leq \frac{1}{t} \log\  \!\bigl|\psi^t_{\omega}(h_{\omega}^{k_0, j_0})\bigr|_{\alpha},
\end{align*}
the claim follows using~\eqref{eq:norm-growth} and~\eqref{eq:angle-vanishing}.
\end{proof}
We need another auxiliary result. First, if \(\tilde{E}\) is a Banach space with closed subspaces \(\tilde{E}_1, \tilde{E}_2 \subset \tilde{E}\) such that \(\tilde{E}_1 \cap \tilde{E}_2 = \{0\}\), we denote by \(\Pi_{\tilde{E}_1 \parallel \tilde{E}_2}\) the canonical projection from \(\tilde{E}_1 \oplus \tilde{E}_2\) onto \(\tilde{E}_1\) along \(\tilde{E}_2\).
\begin{lemma}\label{ASOAaw}
Consider the setting of Theorem~\ref{METT} and assume that $\lambda_i > -\infty$ for some $i \geq 1$.
Let \( K_{\omega}^i \) be a complementary subspace of \( F_{\lambda_{i+1}}(\omega) \) in \( E_{\alpha} \). 
Then, on a set of full measure, the following statements hold true: 
\begin{enumerate}
    \item 
\begin{align}\label{IOASf}
    \lim_{t \to \infty} \frac{1}{t} \log \left\Vert \Pi_{\psi^{t}_{\omega}(K^{i}_{\omega}) \,\|\, F_{\lambda_{i+1}}(\theta_t\omega)} \right\Vert = 0.
\end{align}
    In particular
\[
\lim_{t \to \infty} \frac{1}{t} \log \left\Vert \Pi_{\bigoplus_{1 \leq k \leq i} H^k_{\theta_t\omega} \,\|\, F_{\lambda_{i+1}}(\theta_t\omega)} \right\Vert = 0.
\]
\item Let \( g^{1}_{\omega}, \dots, g^{\tilde{p}}_{\omega} \) be nonzero, linearly independent vectors in \( K^{i}_{\omega} \), and for each \( 1 \leq \tilde{q} \leq \tilde{p} \) suppose
\[
g^{\tilde{q}}_{\omega} = h^{\tilde{q}}_{\omega} + f^{\tilde{q}}_{\omega},
\]
where \( h^{\tilde{q}}_{\omega} \in \bigoplus_{1 \leq k \leq i} H^k_{\omega} \) and \( f^{\tilde{q}}_{\omega} \in F_{\lambda_{i+1}}(\omega) \). Then we have
\begin{align}\label{APspoxww}
\frac{1}{\!\bigl\Vert 
	\Pi_{\!\bigoplus_{1 \leq k \leq i} H^k_{\theta_t \omega} 
		\,\|\, F_{\lambda_{i+1}}(\theta_t \omega)}
	\bigr\Vert^{\tilde{p}}}
\leq 
\frac{
	\operatorname{Vol}_{E_\alpha}\!\bigl( \psi^{t}_{\omega}(g^{1}_{\omega}), \dots, \psi^{t}_{\omega}(g^{\tilde{p}}_{\omega}) \bigr)
}{
	\operatorname{Vol}_{E_\alpha}\!\bigl( \psi^{t}_{\omega}(h^{1}_{\omega}), \dots, \psi^{t}_{\omega}(h^{\tilde{p}}_{\omega}) \bigr)
}
\leq 
\!\bigl\Vert 
\Pi_{\psi^{t}_{\omega}(K^{i}_{\omega}) 
	\,\|\, F_{\lambda_{i+1}}(\theta_t\omega)}
\bigr\Vert^{\tilde{p}}.
\end{align}
Moreover,  the following limit exists and is finite:
\begin{align}
\lim_{t\rightarrow\infty}\frac{1}{t}\log\operatorname{Vol}_{E_\alpha}\!\bigl( \psi^{t}_{\omega}(g^{1}_{\omega}), \dots, \psi^{t}_{\omega}(g^{\tilde{p}}_{\omega}) \bigr).
\end{align}
\end{enumerate}
\end{lemma}

\begin{proof}	  
The first claim follows from \cite[Lemma~4.4]{GVRS22} and \cite[Lemma~1.18]{GVR23A}. Let us now focus on the second claim.  
For \(1 < \tilde{q} \leq \tilde{p}\), we use the definition of the projections $\Pi$ together with  the invariance of the spaces $F_{\lambda_{i+1}}(\omega)$, meaning that
\[
\psi^{t}_{\omega}\!\bigl(F_{\lambda_{i+1}}(\omega)\bigr) 
\subset F_{\lambda_{i+1}}(\theta_t \omega),
\]
to deduce that for any  
\(\tilde{\beta}_1, \dots, \tilde{\beta}_{\tilde{q}-1} \in \mathbb{R}\), we have
\begin{align*}
	&\Pi_{\psi^{t}_{\omega}(K^{i}_{\omega}) 
		\,\|\, F_{\lambda_{i+1}}(\theta_t \omega)} 
	\!\Bigl( 
	\psi^{t}_{\omega}(h^{\tilde{q}}_\omega) - \sum_{1 \leq j < \tilde{q}} \tilde{\beta}_{j} \psi^{t}_{\omega}(h^{j}_\omega)
	\Bigr)
	= \psi^{t}_{\omega}(g^{\tilde{q}}_\omega) - \sum_{1 \leq j < \tilde{q}} \tilde{\beta}_{j} \psi^{t}_{\omega}(g^{j}_\omega), \\
	&\Pi_{\!\bigoplus_{1 \leq k \leq i} H^k_{\theta_t \omega} 
		\,\|\, F_{\lambda_{i+1}}(\theta_t \omega)}
	\!\Bigl( 
	\psi^{t}_{\omega}(g^{\tilde{q}}_\omega) - \sum_{1 \leq j < \tilde{q}} \tilde{\beta}_{j} \psi^{t}_{\omega}(g^{j}_\omega)
	\Bigr)
	= \psi^{t}_{\omega}(h^{\tilde{q}}_\omega) - \sum_{1 \leq j < \tilde{q}} \tilde{\beta}_{j} \psi^{t}_{\omega}(h^{j}_\omega).
\end{align*}
In particular, this yields that
\begin{align*}
    \frac{1}{\Vert \Pi_{\!\bigoplus_{1 \leq k \leq i} H^k_{\theta_t \omega} 
		\,\|\, F_{\lambda_{i+1}}(\theta_t \omega)}\Vert}\leq\frac{d_{E_{\alpha}}\left(\psi^{t}_{\omega}(g^{\tilde{q}}_\omega), \langle  \psi^{t}_{\omega}(g^{j}_\omega) \rangle_{1 \leq j < \tilde{q}}\right)}{d_{E_{\alpha}}\left(\psi^{t}_{\omega}(h^{\tilde{q}}_\omega), \langle  \psi^{t}_{\omega}(h^{j}_\omega) \rangle_{1 \leq j < \tilde{q}}\right)}\leq \Vert \Pi_{\psi^{t}_{\omega}(K^{i}_{\omega}) 
		\,\|\, F_{\lambda_{i+1}}(\theta_t \omega)} \Vert.
\end{align*}
Given this, the inequality~\eqref{APspoxww} easily follows from the definition of $\operatorname{Vol}$.
Finally, the claim regarding the existence of the limit
\begin{align}
	\lim_{t \to \infty} \frac{1}{t} 
	\log \operatorname{Vol}_{E_\alpha}\!\bigl( 
		\psi^{t}_{\omega}(g^{1}_{\omega}), 
		\dots, 
		\psi^{t}_{\omega}(g^{\tilde{p}}_{\omega})
	\bigr)
\end{align}
follows from Lemma~\ref{AISKAsw}, \eqref{IOASf} and \eqref{APspoxww}.
\end{proof}
We are now ready to prove  Lemma~\ref{AISuia9q} and Lemma~\ref{liminf}. 
\begin{proof}[Proof of Lemma~\ref{AISuia9q}]\label{AISuia9q12}
We proceed by induction. For \( \tilde{p} = 1 \) the statement is immediate. Let \( \tilde{p} > 1 \) and assume that the statement holds for every set of \( \tilde{p} - 1 \) independent vectors in \( \bigoplus_{1 \leq k \leq i} H^k_{\omega} \). From the definition of \( \operatorname{Vol}_{E_\alpha} \), we have
\begin{align*}
&\log \operatorname{Vol}_{E_\alpha}\left( \psi^{t}_{\omega}(h^{1}_{\omega}), \dots, \psi^{t}_{\omega}(h^{\tilde{p}}_{\omega}) \right) \\
&= \log \operatorname{Vol}_{E_\alpha}\left( \psi^{t}_{\omega}(h^{1}_{\omega}), \dots, \psi^{t}_{\omega}(h^{\tilde{p}-1}_{\omega}) \right) + \log d_{E_\alpha}\left( \psi^t_{\omega}(h_{\omega}^{\tilde{p}}), \langle \psi^{t}_{\omega}(h^{j}_\omega) \rangle_{1 \leq j < \tilde{p}} \right).
\end{align*}
Therefore, by the induction hypothesis, it suffices to show that the following limit 
\begin{align*}
\lim_{t \to \infty} \frac{1}{t} \log d_{E_\alpha}\left( \psi^t_{\omega}(h_{\omega}^{\tilde{p}}), \langle \psi^{t}_{\omega}(h^{j}_\omega) \rangle_{1 \leq j < \tilde{p}} \right)
\end{align*}
exists.
To prove the claim, we define 
\begin{align}\label{sa[ds4a]}
\begin{split}
r := \max \Bigl\{\, 
& 1 \leq k \leq i \ \Bigm| \ \ \exists \ 
\tilde{\beta}_1, \dots, \tilde{\beta}_{\tilde{p}-1} \in \mathbb{R} 
\ \text{such that} \\
& h^{\tilde{p}}_{\omega} - \sum_{j=1}^{\tilde{p}-1} \tilde{\beta}_j h^{j}_{\omega} 
= A_{\omega} + B_{\omega}, 
\ \ A_{\omega} \in H_{\omega}^k \setminus \{0\}, \quad
B_{\omega} \in \bigoplus_{\,k < j \leq i} H_{\omega}^j 
\Bigr\}.
\end{split}
\end{align}
Given this, we get that 
\begin{align}\label{sa[ds4aa]}
\begin{split}
h^{\tilde{p}}_{\omega} &= \sum_{j=1}^{\tilde{p}-1} \tilde{\beta}_j h^{j}_{\omega} + A_{\omega} + B_{\omega}, \\
\text{where} \quad A_{\omega} &\in H_{\omega}^r \setminus \{0\}, \quad B_{\omega} \in \bigoplus_{r < j \leq i} H_{\omega}^j.
\end{split}
\end{align}
Thus
\begin{align*}
& d_{E_\alpha}\left( \psi^t_{\omega}(h_{\omega}^{\tilde{p}}), \langle \psi^{t}_{\omega}(h^{j}_\omega) \rangle_{1 \leq j < \tilde{p}}\right) \\
&= d_{E_\alpha}\left( \psi^t_{\omega}(A_\omega) + \psi^t_{\omega}(B_\omega), \langle \psi^{t}_{\omega}(h^{j}_\omega) \rangle_{1 \leq j < \tilde{p}} \right).
\end{align*}
From the definition of \( d_{E_\alpha} \), we have
\begin{align*}
& d_{E_\alpha}\left( \psi^t_{\omega}(A_\omega), \langle \psi^{t}_{\omega}(h^{j}_\omega) \rangle_{1 \leq j < \tilde{p}}\right) - \big\| \psi^t_{\omega}(B_\omega) \big\| \\
& \quad \leq d_{E_\alpha}\left( \psi^t_{\omega}(A_\omega) + \psi^t_{\omega}(B_\omega), \langle \psi^{t}_{\omega}(h^{j}_\omega) \rangle_{1 \leq j < \tilde{p}} \right) \\
& \quad \leq \big\| \psi^t_{\omega}(B_\omega) \big\| + d_{E_\alpha}\left( \psi^t_{\omega}(A_\omega), \langle \psi^{t}_{\omega}(h^{j}_\omega) \rangle_{1 \leq j < \tilde{p}}\right).
\end{align*}	
Since
\[
\limsup_{t \rightarrow \infty} \frac{1}{t} \log \big\| \psi^t_{\omega}(B_\omega) \big\| \leq \lambda_{r+1},
\]
and given that \( \lambda_r > \lambda_{r+1} \), the claim follows if we can establish that
\[
\lim_{t \rightarrow \infty} \frac{1}{t} \log d_{E_\alpha} \left( \psi^t_{\omega}(A_\omega), \langle \psi^{t}_{\omega}(h^{j}_\omega) \rangle_{1 \leq j < \tilde{p}} \right) = \lambda_{r}.
\]
To this end, first note that from \eqref{sa[ds4a]} and \eqref{sa[ds4aa]}, we obtain the following consequences:
\begin{itemize}
\item[(I)] $A_{\omega}$ is independent of the vectors $(h^{j}_{\omega})_{1 \leq j \leq \tilde{p}-1}$. 
\item[(II)] For each $1 \leq j \leq \tilde{p}-1$, we write
\begin{align*}
h^{j}_{\omega} = \sum_{1 \leq k \leq i} h^{k,j}_{\omega}, 
\quad \text{with } h^{k,j}_{\omega} \in H^{k}_{\omega}.
\end{align*}
If
\[
\sum_{1 \leq k \leq r-1} h^{k,j}_{\omega} = 0 \quad { and }~~ h^{r,j}_{\omega} \neq 0,
\]
 then $h^{r,j}_{\omega}$ is independent of $A_\omega$ in $H^r_{\omega}$.
\end{itemize}	
Otherwise, we obtain a contradiction with the choice of $r$ in \eqref{sa[ds4a]}.
Let us choose a subspace $\tilde{H}^{r}_{\omega}$ such that
\[
\tilde{H}^{r}_{\omega} \oplus \langle A_{\omega} \rangle = H^{r}_{\omega}.
\]
For a set of vectors $S \subset E_{\alpha}$, we denote by 
$\langle S \rangle$ the subspace of $E_{\alpha}$ spanned by the vectors in $S$ 
and set $\langle \varnothing \rangle := \{0\}$.
Using (I) and (II), we conclude that
for every $1 \leq k \leq r-1$, there exists a set of linearly independent vectors $S_k = S_k^1 \cup S_k^2$ forming a basis for $H^k_{\omega}$, and a set of independent vectors $\tilde{S}_r = \tilde{S}_r^1 \cup \tilde{S}_r^2$
 forming a basis for $\tilde{H}^{r}_{\omega}$, such that\footnote{Note that $A_{\omega} + \emptyset := \emptyset$, and some of the sets $S^{2}_{k}$ and $\tilde{S}_r^2$ may be empty.}
\begin{align*}
\left\langle h_{\omega}^1, \dots, h_{\omega}^{\tilde{p}-1} \right\rangle 
\subseteq \bigoplus_{1 \leq k \leq r-1} 
\left( \langle S_k^1 \rangle \oplus \langle S_k^2 + A_\omega \rangle \right)
\oplus \left( \langle \tilde{S}_r^1 \rangle \oplus \langle \tilde{S}_r^2 + A_\omega \rangle \right)
 \bigoplus_{r < k \leq i} H^k_{\omega}.
\end{align*}
Moreover, note that for each \( k \leq r-1 \), since \( \lambda_k > \lambda_r \), the corresponding Lyapunov exponent for every nonzero element in 
\[
\langle S_k^1 \rangle \oplus \langle S_k^2 + A_\omega \rangle
\]
is equal to \( \lambda_k \).
Moreover, by the choice of \( \tilde{H}^{r}_{\omega} \), it follows that the corresponding Lyapunov exponent for every nonzero element in
\[
\langle \tilde{S}_r^1 \rangle \oplus \langle \tilde{S}_r^2 + A_\omega \rangle
\]
is equal to \( \lambda_r \).
Thus, we are in the setting of Lemma~\ref{AISKAsw}, and therefore
\[
\lim_{t \to \infty} \frac{1}{t} \log d_{E_\alpha}\!\left( 
\psi^t_{\omega}(A_\omega),\,
\langle \psi^{t}_{\omega}(h^{j}_\omega) \rangle_{1 \leq j < \tilde{p}}
\right) 
= \lambda_{r}.
\]
This completes the proof.
\end{proof}

\begin{proof}[Proof of Lemma~\ref{liminf}]\label{liminf2}
First, we claim that 
\begin{align}\label{VB528}
\forall\, 1 \leq \tilde{q} \leq \tilde{p}:\quad  \liminf_{t \to \infty} \frac{1}{t} \log d_{E_\alpha} \left( \psi^t_{\omega}(g_{\omega}^{\tilde{q}}), \langle \psi^t_{\omega}(g_{\omega}^k) \rangle_{\substack{1 \leq k < \tilde{p} \\ k \neq \tilde{q}}} \right) > -\infty.
\end{align}
To establish this, first note that by the definition of \(\operatorname{Vol}_{E_\alpha}\), 
\begin{align}\label{FDCXZ}
\begin{split}
	-\infty &< \liminf_{t \to \infty} \frac{1}{t} \log \operatorname{Vol}_{E_\alpha}\big(\psi^{t}_{\omega}(g^{1}_{\omega}), \dots, \psi^{t}_{\omega}(g^{\tilde{p}}_{\omega})\big) \\
	&\leq \liminf_{t \to \infty}\frac{1}{t} \left(  \sum_{1 \leq k < \tilde{p}} \log \left| \psi^{t}_{\omega}(g^{k}_{\omega}) \right|_{\alpha} +\log d_{E_\alpha} \left( \psi^t_{\omega}(g_{\omega}^{\tilde{p}}), \left\langle \psi^t_{\omega}(g_{\omega}^k) \right\rangle_{\substack{1 \leq k < \tilde{p}}} \right)\right).
\end{split}
\end{align}  
Note that for every \(k \in \{1, 2, \dots, \tilde{p}\}\), we have  
\[
\limsup_{t \to \infty} \frac{1}{t} \log \left| \psi^{t}_{\omega}(g^{k}_{\omega}) \right|_{\alpha} \leq \lambda_{1} < \infty. 
\]  
It then follows from \eqref{FDCXZ} that  
\begin{align}\label{Ase23as}
\liminf_{t \to \infty} \frac{1}{t} \log d_{E_\alpha} \left( \psi^t_{\omega}(g_{\omega}^{\tilde{p}}), \left\langle \psi^t_{\omega}(g_{\omega}^k) \right\rangle_{\substack{1 \leq k < \tilde{p}}} \right) > -\infty,
\end{align}
since otherwise \(\limsup_{t \to \infty} \frac{1}{t} \log \left| \psi^{t}_{\omega}(g^{k}_{\omega}) \right|_{\alpha} = \infty\)
for some \(k \in \{1, 2, \dots, \tilde{p}-1\}\), which is a contradiction.
Note that, thanks to Lemma~\ref{permutation}, we can consider any other permutation of the set \(k \in \{1, 2, \dots, \tilde{p}-1\}\) and repeat the same argument. Thus \eqref{Ase23as} entails \eqref{VB528}. Let
\begin{align}\label{IK3a}
\Lambda:=\min\left\lbrace \liminf_{t \to \infty} \frac{1}{t} \log d_{E_\alpha} \left( \psi^t_{\omega}(g_{\omega}^{\tilde{q}}), \ \left\langle \psi^t_{\omega}(g_{\omega}^k) \right\rangle_{\substack{1 \leq k < \tilde{p} \\ k \neq \tilde{q}}} \right)\right\rbrace>-\infty.
\end{align}\label{IK3a2}
From Theorem \ref{METT} we can find \( j \geq 0 \) such that
\begin{align}\label{sds7qaw}
\lambda_{j+1} < \Lambda .
\end{align}
For an element $x \in E_{\alpha}$, we denote by $[x]_{\lambda_{j+1},\omega}$ its equivalence class in the quotient space $E_{\alpha} / F_{\lambda_{j+1}}(\omega)$.  
We claim that the vectors 
\[
\big([g^{k}_{\omega}]_{\lambda_{j+1},\omega}\big)_{1 \leq k \leq \tilde{p}}
\] 
are linearly independent in $E_{\alpha} / F_{\lambda_{j+1}}(\omega)$.
We prove this by contradiction. Without loss of generality, we may assume that
\begin{align*}
g^{\tilde{p}}_{\omega} = \sum_{1 \leq k < \tilde{p}} r_{k} g^{k}_{\omega} + \zeta_{\omega},
\end{align*}
where \( r_{k} \in \mathbb{R} \) and \( \zeta_{\omega} \in F_{\lambda_{j+1}}(\omega) \). 
Then we have 
\begin{align*}
\liminf_{t \to \infty} \frac{1}{t} \log d_{E_\alpha} \left( \psi^t_{\omega}(g_{\omega}^{\tilde{p}}), \left\langle \psi^t_{\omega}(g_{\omega}^k) \right\rangle_{\substack{1 \leq k < \tilde{p}}} \right)
\leq \limsup_{t \to \infty} \frac{1}{t} \log \left| \psi^t_{\omega}(\zeta_\omega)\right|_{\alpha} 
\leq \lambda_{j+1},
\end{align*}
which contradicts \eqref{IK3a} and \eqref{sds7qaw}.  This also yields that the vectors $\big(g^{k}_{\omega}\big)_{1 \leq k \leq \tilde{p}} $
are linearly independent. Now we can apply Lemma~\ref{ASOAaw} to complete the proof.
\end{proof} 



\newcommand{\etalchar}[1]{$^{#1}$}
\def\cprime{$'$} \def\cprime{$'$}

\end{document}